\documentclass{amsart}
\usepackage{amsmath,amssymb,graphicx,latexsym}
\usepackage[T1]{fontenc}		
\usepackage{stmaryrd}		
\usepackage{tikz}
\usepackage{hyperref}

\newtheorem{theorem}{Theorem}
\newtheorem{corollary}[theorem]{Corollary}
\newtheorem{lemma}[theorem]{Lemma}
\newtheorem{proposition}[theorem]{Proposition}
\newtheorem*{main theorem - ring}{Theorem~\ref{kernel size upper bound - ring}}
\newtheorem*{main theorem - asymptotic - ring}{Theorem~\ref{kernel size asymptotic bound - ring}}
\newtheorem*{main theorem - diagonals}{Theorem~\ref{kernel size upper bound - multivariate diagonals}}

\theoremstyle{definition}
\newtheorem{example}[theorem]{Example}
\newtheorem*{definition*}{Definition}
\newtheorem*{notation*}{Notation}
\newtheorem{remark}[theorem]{Remark}

\DeclareMathOperator{\dig}{dig}
\DeclareMathOperator{\lcm}{lcm}
\DeclareMathOperator{\mindeg}{mindeg}
\DeclareMathOperator{\orb}{orb}
\DeclareMathOperator{\ord}{ord}
\DeclareMathOperator{\partitions}{parts}
\DeclareMathOperator{\pr}{pr}

\DeclareMathOperator{\rep}{rep}
\DeclareMathOperator{\res}{res}
\DeclareMathOperator{\val}{val}

\newcommand{\F}{\mathbb{F}}
\newcommand{\Q}{\mathbb{Q}}

\newcommand{\Z}{\mathbb{Z}}

\newcommand{\field}{\F_q}
\newcommand{\ring}{\mathcal{R}_{p^\alpha}}

\newcommand{\Landau}{g}
\newcommand{\Landaulcm}{\mathcal L}

\newcommand{\subb}{_\textnormal{b}}
\newcommand{\subl}{_\ell}
\newcommand{\subr}{_\textnormal{r}}
\newcommand{\subt}{_\textnormal{t}}

\newcommand{\colonequal}{\mathrel{\mathop:}=}
\newcommand{\equalcolon}{=\mathrel{\mathop:}}
\newcommand{\nequiv}{\mathrel{\not\equiv}}

\newcommand{\size}[1]{\lvert{#1}\rvert}
\newcommand{\ceil}[1]{\left\lceil #1 \right\rceil}
\newcommand{\doublebracket}[1]{\llbracket #1 \rrbracket}
\newcommand{\floor}[1]{\left\lfloor #1 \right\rfloor}
\newcommand{\paren}[1]{\left( #1 \right)}

\newcommand{\seq}[1]{\cite[\href{http://oeis.org/#1}{#1}]{OEIS}}

\begin{document}

\title[Algebraic series and their automatic complexity]{Algebraic power series and their \\ automatic complexity modulo prime powers}

\author{Eric Rowland}
\address{
	Department of Mathematics \\
	Hofstra University \\
	Hempstead, NY 11549 \\
	USA
}
\author{Reem Yassawi}
\address{
	School of Mathematical Sciences \\
	Queen Mary University of London \\
	Mile End Road \\
	London E1 4NS \\
	UK
}

\date{June 25, 2026}

\subjclass[2020]{11A63, 11B85, 13F25}

\thanks{The second author was supported by the EPSRC, grant number EP/V007459/2.}

\begin{abstract}
Christol and, independently, Denef and Lipshitz showed that an algebraic sequence of $p$-adic integers (or integers) is $p$-automatic when reduced modulo~$p^\alpha$.
Previously, the best known bound on the minimal automaton size for such a sequence was doubly exponential in $\alpha$.
Under mild conditions, we improve this to a bound whose dominant factor is $p^{\alpha^3 h d / 3}$, where $h$ and $d$ are the height and degree of the minimal annihilating polynomial modulo~$p$.
We achieve this bound by showing that all states in the automaton are naturally represented in a new numeration system.
This significantly restricts the set of possible states.
Since our approach embeds algebraic sequences as diagonals of rational functions, we also obtain bounds more generally for diagonals of multivariate rational functions.
\end{abstract}

\maketitle

\section{Introduction}\label{section: introduction}

Christol's theorem~\cite{Christol, Christol--Kamae--Mendes France--Rauzy} states that a power series $F = \sum_{n \geq 0} a(n) x^n \in \field\doublebracket{x}$ with coefficients in a finite field is algebraic if and only if its sequence of coefficients $a(n)_{n \geq 0}$ is $q$-automatic.
Given a nonzero polynomial $P \in \field[x, y]$ such that $P(x, F) = 0$, the proof of Christol's theorem allows one to compute an automaton that outputs $a(n)$ when fed the standard base-$q$ representation of $n$.
Bridy~\cite{Bridy} gave an upper bound on the size of this automaton in terms of information about $P$, namely the \emph{height} $h \colonequal \deg_x P$, the \emph{degree} $d \colonequal \deg_y P$, and the genus of $P$.
Bounding the genus above by $(h - 1) (d - 1)$, Bridy obtains as a corollary that the number of states is in $(1 + o(1)) q^{h d}$ as $q$, $h$, or $d$ gets large.
In a previous article, Stipulanti and the authors~\cite{Rowland--Stipulanti--Yassawi} gave a new proof of this genus-free bound.

Christol~\cite{Christol Fonctions} and, using a different approach, Denef and Lipshitz~\cite{Denef--Lipshitz} proved a generalization of Christol's theorem for power series with coefficients in $\Z$ or, more generally, the set $\Z_p$ of $p$-adic integers where $p$ is a prime.
Namely, if $F = \sum_{n \geq 0} a(n) x^n \in \Z_p\doublebracket{x}$ is algebraic, then the sequence $(a(n) \bmod p^\alpha)_{n \geq 0}$ of elements in the ring $\ring \colonequal \Z/p^\alpha \Z$ is $p$-automatic for each $\alpha \geq 1$.

In this article, we bound the size of the minimal automaton generating this sequence.
Here and throughout this article, automata read the least significant digit first.
Our main results concern series defined as follows.

\begin{definition*}
Let $P \in \Z_p[x, y]$ such that $P(0, 0) = 0$ and $\frac{\partial P}{\partial y}(0, 0) \nequiv 0 \mod p$, that is, $\frac{\partial P}{\partial y}(0, 0)$ is a unit.
The \emph{Furstenberg series} associated with $P$ is the unique element $F \in \Z_p\doublebracket{x}$ satisfying $F(0) = 0$ and $P(x, F) = 0$.
\end{definition*}

From one of the constructions of Denef and Lipshitz~\cite[Remark~6.6]{Denef--Lipshitz}, for Furstenberg series one obtains that the number of states is at most $p^{\alpha (p^{\alpha - 1} \max(h, d) + 1)^2}$~\cite[Remark~2.2]{Rowland--Yassawi}.
This bound is doubly exponential in $\alpha$.
The techniques from \cite{Rowland--Stipulanti--Yassawi} can be generalized to give a bound in this setting, but this bound is also doubly exponential.
In this article, we significantly reduce the bound to roughly $p^{\alpha^3 h d / 3}$.
This reduction is achieved primarily by identifying new structure in the representations of the states of the automaton.

In the definition of a Furstenberg series, the condition $\frac{\partial P}{\partial y}(0, 0) \nequiv 0 \mod p$ says that the coefficient of $x^0 y^1$ in $P$ is nonzero modulo~$p$, and it implies that $\deg_y(P \bmod p) \geq 1$.
If $\deg_x P = 0$, then $F$ is the trivial $0$ series; we exclude it since otherwise the statement of the following theorem does not hold when $h = 0$.

\begin{theorem}\label{kernel size asymptotic bound - ring}
Let $p$ be a prime, let $\alpha \geq 1$, and let $F = \sum_{n \geq 0} a(n) x^n \in \Z_p\doublebracket{x} \setminus \{0\}$ be the Furstenberg series associated with a polynomial $P \in \Z_p[x, y]$.
Let $h \colonequal \deg_x(P \bmod p)$ and $d \colonequal \deg_y(P \bmod p)$, and assume $h = \deg_x P$ and $d = \deg_y P$.
Then the size of the minimal $p$-automaton generating $(a(n) \bmod p^\alpha)_{n \geq 0}$ is in
\[
	(1 + o(1)) \, p^{\frac{1}{6} \alpha (\alpha + 1) ((2 h d - 1) \alpha + h d + 1)}
\]
as any of $p$, $\alpha$, $h$, or $d$ tends to infinity and the others remain constant.
\end{theorem}

Theorem~\ref{kernel size asymptotic bound - ring} follows from a finer result, whose statement involves the following functions.
Define $\partitions(n)$ to be the set of all integer partitions of $n$.
The \emph{Landau function} $\Landau(n)$ is the maximum value of $\lcm(\sigma)$ over all integer partitions $\sigma \in \partitions(n)$~\seq{A000793}.
For example, $\Landau(5)$ is the maximum value among $\lcm(5)$, $\lcm(4, 1)$, $\lcm(3, 2)$, $\lcm(3, 1, 1)$, $\lcm(2, 2, 1)$, $\lcm(2, 1, 1, 1)$, and $\lcm(1, 1, 1, 1, 1)$, so $\Landau(5) = 6$.
Finally, define
\[
	\Landaulcm(l, m, n) \colonequal
	\max_{\substack{
		\vphantom{\sigma_1 \in \partitions(i)} 1 \leq i \leq l \\
		\vphantom{\sigma_2 \in \partitions(j)} 1 \leq j \leq m \\
		\vphantom{\sigma_3 \in \partitions(k)} 1 \leq k \leq n
	}}
	\max_{\substack{
		\sigma_1 \in \partitions(i) \\
		\sigma_2 \in \partitions(j) \\
		\sigma_3 \in \partitions(k)
	}}
	\lcm(\lcm(\sigma_1), \lcm(\sigma_2), \lcm(\sigma_3)).
\]

In the following theorem, we remove the mild conditions $h= \deg_x P$ and $d= \deg_y P$.
They are needed in Theorem~\ref{kernel size asymptotic bound - ring} for the asymptotic result, since otherwise a family of polynomials $P$ can be chosen such that $\deg_x P = p^{p^h}$ or $\deg_y P = p^{p^d}$, and for such families the value of $p^u$ in Theorem~\ref{kernel size upper bound - ring} grows too quickly as $h$ or $d$ gets large.
Note that the conditions of Theorem~\ref{kernel size asymptotic bound - ring} imply $u = 1$.
The maps $\pi_{x, i}$ and $\pi_{y, j}$ project bivariate Laurent polynomials to univariate Laurent polynomials; for the definitions, see Section~\ref{section: structure - ring}.

\begin{theorem}\label{kernel size upper bound - ring}
Let $p$ be a prime, let $\alpha \geq 1$, and let $F = \sum_{n \geq 0} a(n) x^n \in \Z_p\doublebracket{x}$ be the Furstenberg series associated with a polynomial $P \in \Z_p[x, y]$ such that $h \colonequal \deg_x(P \bmod p) \geq 1$.
Let $d = \deg_y(P \bmod p)$,
\[
	N = \tfrac{1}{6} \alpha (\alpha + 1) ((2 h d - 1) \alpha + h d + 1),
\]
and
\[
	u = \floor{\log_p \max\!\paren{\alpha (\deg_x(P \bmod p^\alpha) - h), \alpha (\deg_y(P \bmod p^\alpha) - d) + 1}} + 1.
\]
Let $Q \in \ring[x, y, y^{-1}]$ be a lift of $P/y \bmod p$ which has the same monomial support as $P/y \bmod p$, and let
\begin{align*}
	u\subl &= \floor{\log_p \max\!\paren{p^{\alpha - 1} (d - 1 - \deg \pi_{x, 0}(Q)), 1}} + 1 \\
	u\subr &= \floor{\log_p \max\!\paren {p^{\alpha - 1} (d - 1 - \deg \pi_{x, h}(Q)), 1}} + 1 \\
	u\subt &= \floor{\log_p \max\!\paren {p^{\alpha - 1} (h - \deg \pi_{y, d - 1}(Q)), 1}} + 1.
\end{align*}
Then the minimal $p$-automaton that generates $(a(n) \bmod p^\alpha)_{n \geq 0}$ has size at most
\begin{multline*}
	p^N + p^{N - \alpha (\alpha + 1) (h + d - 1) / 2} \Landaulcm(h, d, d) \\
	+ \max(u\subl, u\subr, u\subt)
	+ \floor{\log_p \max(h, d)}
	+ 2 \alpha - 1
	+ \tfrac{p^u - 1}{p - 1}.
\end{multline*}
\end{theorem}

For a random assignment of the coefficients of $P \bmod p$, the value of $\max(u\subl, u\subr, u\subt)$ is typically $1$ when $p$ is large.
For all polynomials $P$, using $\max(u\subl, u\subr, u\subt) \leq \floor{\log_p \max(h, d)} + \alpha$, we obtain the simpler bound
\[
	p^N + p^{N - \alpha (\alpha + 1) (h + d - 1) / 2} \Landaulcm(h, d, d)
	+ 2 \floor{\log_p \max(h, d)}
	+ 3 \alpha - 1
	+ \tfrac{p^u - 1}{p - 1}.
\]

We recover the bound for a sequence of elements in the field $\F_p$~\cite[Theorem~1]{Rowland--Stipulanti--Yassawi} as follows.
Let $\alpha = 1$, and let $P\in \F_p[x,y]$ such that $\deg_x P = h$ and $\deg_y P = d$, so that $u = 1$.
Then the bound in Theorem~\ref{kernel size upper bound - ring} is at most
\[
	p^{h d} + p^{(h - 1) (d - 1)} \Landaulcm(h, d, d)
	+ \floor{\log_p \max(h, d - 1)} + 1
	+ \floor{\log_p \max(h, d)} + 1
	+ 1.
\]
When we use the further specificity of working over a field, one obtains the bound
\begin{equation}\label{fields bound}
	p^{h d} + p^{(h - 1) (d - 1)} \Landaulcm(h, d, d)
	+ \floor{\log_p h}
	+ \floor{\log_p \max(h, d)}
	+ 3
\end{equation}
in~\cite{Rowland--Stipulanti--Yassawi}.

Explicit computations suggest that the bound \eqref{fields bound} for sequences of elements in $\F_p$ is asymptotically sharp but that the bound in Theorem~\ref{kernel size upper bound - ring} is not.
The appendix of this article contains the results of searches for large automata for comparison.
Bounds on the automaton size have implications for the time complexity of any algorithm that computes an automaton generating $(a(n) \bmod p^\alpha)_{n \geq 0}$.
In particular, the algorithm we describe in Section~\ref{section: vector space - ring} has been used to systematically answer number theoretic questions about sequences arising in combinatorics and number theory, such as the Catalan numbers and Ap\'ery numbers~\cite{Rowland--Yassawi}.
We mention that another approach to this same question, using sequences represented as constant terms of powers of Laurent polynomials rather than diagonals of rational functions, has also been quite successful~\cite{Rowland--Zeilberger, Henningsen--Straub, Straub, Beukers}.

The main innovation in this article is a new numeration system for a family of bivariate Laurent polynomials.
This numeration system behaves in many ways like base-$p$ representations of integers, except that the digits are Laurent polynomials whose degrees can increase from one digit to the next.
States of the automaton for $(a(n) \bmod p^\alpha)_{n \geq 0}$ are identified with bivariate Laurent polynomials, as in Section~\ref{section: vector space - ring}, and we show that each state has a unique base-$\frac{p}{Q}$ representation, as defined in Section~\ref{section: image - ring}, consisting of $\alpha$ digits.
Most Laurent polynomials are not representable in this numeration system and therefore do not represent states.
This allows us to avoid a doubly exponential upper bound on the number of states.
By bounding the degrees of the digits, we obtain the bound in Theorem~\ref{kernel size upper bound - ring}.

There are other benefits of base-$\frac{p}{Q}$ representations as well.
Representations of states in this numeration system are much more compact than their full expansions as Laurent polynomials, so using base-$\frac{p}{Q}$ representations greatly reduces the amount of memory and time required to compute an automaton.
For the sequence of Catalan numbers $C(n)_{n \geq 0}$, previously we were able to compute an automaton for $(C(n) \bmod 2^9)_{n \geq 0}$, which has $2403$ states, but not for larger powers of $2$~\cite{Rowland--Yassawi}.
With base-$\frac{p}{Q}$ representations, we are able to compute an automaton for $(C(n) \bmod 2^{14})_{n \geq 0}$; it has $174037$ states and required $34$GB of memory.
From this automaton, one computes that only $\frac{2990}{2^{14}} \approx 18.2\%$ of the residues modulo~$2^{14}$ are attained by the Catalan numbers, and only $\frac{2037}{2^{14}} \approx 12.4\%$ of residues are attained infinitely many times.
This agrees with results of Straub~\cite[Table~2]{Straub}, who previously computed the residues attained by $C(n)$ modulo~$2^{14}$ by representing $C(n)$ as a constant term and using scaling $p$-schemes rather than $p$-automata.
In particular, these computations establish an upper bound of $0.124$ on the density of the set $\{C(n) : n \geq 0\}$ in $\Z_2$.
It is not known whether this density is positive.

Unrelated to our main results, in Section~\ref{section: inverse limit} we give another application of the base-$\frac{p}{Q}$ numeration system.
For a given algebraic power series $F$, the automaton for $F \bmod p^\alpha$ projects naturally to the automaton for $F \bmod p^\beta$ for each $\beta \leq \alpha$, as was shown by the authors~\cite{Rowland--Yassawi profinite}.
Consequently, the inverse limit as $\alpha \to \infty$ of these automata exists.
However, no explicit description of its states was known.
We show that the base-$\frac{p}{Q}$ representations of the states of the automaton for $F \bmod p^\beta$ are simply truncations of those for $F \bmod p^\alpha$.
This gives explicit descriptions of the states of the inverse limit automaton as inverse limits of finite base-$\frac{p}{Q}$ representations.

Theorem~\ref{kernel size asymptotic bound - ring} requires $F$, annihilated by $P$, to be a Furstenberg series.
In Section~\ref{section: non-Furstenberg}, we consider the possibility that $F$ is not a Furstenberg series.
A standard modification of $F$ gives an algebraic series $G$ whose annihilating polynomial has a nonzero derivative at the origin.
If, in addition, this derivative is nonzero modulo~$p$, then we obtain an asymptotic bound in Theorem~\ref{kernel size upper bound - non-Furstenberg};
this bound is roughly $p^{\frac{2}{3} \alpha^3 h d^3}$ where $h = \deg_x P$ and $d = \deg_y P$, so it is slightly larger than the bound in Theorem~\ref{kernel size asymptotic bound - ring}.
However, the relevant derivative may be zero modulo~$p$.
For example, the polynomial $P = (x + 1) (3 x - 1) y^2 + 1$ annihilates the generating series of the central trinomial coefficients, and for $p=2$ the technique does not work~\cite[Section~4.2]{Rowland--Yassawi}.

We prove Theorem~\ref{kernel size upper bound - ring} in three steps.
The first step is carried out in Section~\ref{section: image - ring}, where we describe the base-$\frac{p}{Q}$ numeration system.
This first step is not present in the proof of the bound \eqref{fields bound}, since when $\alpha = 1$ the base-$\frac{p}{Q}$ representation consists only of a single digit and does not place any restrictions on Laurent polynomials that represent states.

The second step is to obtain a preliminary upper bound on the automaton size.
This is the subject of Section~\ref{section: first bounds}.
We define a vector space $\mathcal W$ and show that $\mathcal W$ contains the base-$\frac{p}{Q}$ representations of most of the automaton states.
To do this, we bound the degrees of the digits of the base-$\frac{p}{Q}$ representations of states.
The space $\mathcal W$ has dimension $N = \frac{1}{6} \alpha (\alpha + 1) ((2 h d - 1) \alpha + h d + 1)$, and this gives the main term in Theorem~\ref{kernel size upper bound - ring}.
The only states whose base-$\frac{p}{Q}$ representations do not belong to $\mathcal W$ are those reachable from the initial state by either a fixed small number of transitions or reading a sequence of $0$s.
The former are counted easily, and the latter correspond to the states in the orbit of the initial state under a certain linear transformation $\lambda_{0, 0}$.

The third and longest step in the proof of Theorem~\ref{kernel size upper bound - ring} is to bound the orbit size of the initial state under $\lambda_{0, 0}$.
We do this by studying the structure of $\lambda_{0, 0}$ as a linear transformation on bivariate Laurent polynomials.
We essentially decompose the space containing the automaton states into four subspaces --- three ``borders'' and an ``interior''.
On the interior, we have no control over the behavior of $\lambda_{0, 0}$ except what we get from base-$\frac{p}{Q}$ representations.
On the three borders, however, we have significant control.
In Section~\ref{section: structure - ring}, we show that $\lambda_{0, 0}$, when restricted to each of the borders, behaves like a linear transformation $\lambda_0$ on univariate Laurent polynomials.
This is analogous to part of the proof of the bound \eqref{fields bound}, but here we must also show that the base-$\frac{p}{Q}$ representations are compatible.
Next, we must bound orbit sizes under $\lambda_0$.
We do this by first bounding the period length of the coefficient sequence of a univariate rational power series of the form $\frac{1}{R^{p^{\alpha - 1}}} \bmod p^\alpha$ in Section~\ref{section: period lengths}.
We use results of Engstrom~\cite{Engstrom} that bound period lengths of coefficient sequences of rational series modulo~$p$ and modulo~$p^\alpha$.
Then, in Section~\ref{section: orbit size univariate - ring}, we transfer the bound on the period length of $\frac{1}{R^{p^{\alpha - 1}}} \bmod p^\alpha$ to a bound on the orbit size of a univariate Laurent polynomial under $\lambda_0$.
In Section~\ref{section: orbit size - ring}, we complete the proof of Theorem~\ref{kernel size upper bound - ring} by showing how the orbit sizes under $\lambda_0$ contribute to the orbit size under $\lambda_{0,0}$.
Theorem~\ref{kernel size asymptotic bound - ring} follows relatively easily.

Bounding the period lengths of coefficients of rational power series in Section~\ref{section: period lengths} and the orbit size under $\lambda_0$ in Section~\ref{section: orbit size univariate - ring} are considerably more involved over $\ring$ than in the proof of the bound \eqref{fields bound}, which relied on facts about finite fields.
One difficulty is that $\ring[z]$ does not have unique factorization.
We get around this by showing that the first bound essentially depends on the period length of the coefficient sequence modulo~$p$.
We use a lifting-the-exponent lemma throughout to show that the second bound also essentially depends on information modulo~$p$.

Finally, in Section~\ref{section: diagonals}, we generalize Theorem~\ref{kernel size upper bound - ring} from algebraic series to diagonals of rational functions in several variables.
The proof of Theorem~\ref{kernel size upper bound - ring} begins by converting the algebraic power series $F$ to the diagonal of a rational function (in $2$ variables).
By starting directly with the latter, the same approach allows us to bound the automaton size for the diagonal of a rational function in $m$ variables modulo~$p^\alpha$, where the diagonal operator $\mathcal{D}$ is defined on multivariate series analogously to bivariate series.

\begin{theorem}\label{kernel size upper bound - multivariate diagonals}
Let $p$ be a prime, let $\alpha \geq 1$, and let
\[
	F \colonequal \mathcal{D}\!\paren{\frac{P(x_1, \dots, x_m)}{Q(x_1, \dots, x_m)}}
\]
where $P(x_1, \dots, x_m)$ and $Q(x_1, \dots, x_m)$ are polynomials in $\Z_p[x_1, \dots, x_m]$ such that $Q(0, \dots, 0) \nequiv 0 \mod p$ and $m\geq 2$.
Write $F = \sum_{n \geq 0} a(n) x^n$.
Let $h_i = \max(\deg_{x_i}(P \bmod p), \deg_{x_i}(Q \bmod p))$, and assume that $h_i \geq 1$ for each $i$.
Let $M = \sum_{k = 0}^{\alpha - 1} \prod_{i = 1}^m ((k + 1) h_i + 1)$.
Then the minimal $p$-automaton that generates $(a(n) \bmod p^\alpha)_{n \geq 0}$ has size at most $p^M$.
\end{theorem}

The dominant factor in the previous bound is $p^{\alpha^{m + 1} h_1 \cdots h_m}$.
For polynomials $P, Q$ with integer coefficients, a result of Beukers~\cite[Corollary~4.2]{Beukers} implies that if $\deg_{x_i} P \leq \deg_{x_i} Q \equalcolon h_i$ and $\alpha - 1 \leq \rho - \ceil{\rho/p}$ for some integer $\rho$ then there is an automaton generating $(a(n) \bmod p^\alpha)_{n \geq 0}$ with size at most $p^{\alpha (\rho h_1 + 1) \cdots (\rho h_m + 1)}$.
In particular, if $\alpha \leq p$ then we can take $\rho = \alpha$ and obtain the bound $p^{\alpha (\alpha h_1 + 1) \cdots (\alpha h_m + 1)}$, which has the same dominant factor as the bound in Theorem~\ref{kernel size upper bound - multivariate diagonals} if $\deg_{x_i} Q = \deg_{x_i}(Q \bmod p)$ and $\deg_{x_i} P = \deg_{x_i}(P \bmod p)$, but which is asymptotically larger.
Adamczewski, Bostan, and Caruso~\cite[Corollary~1.4]{Adamczewski--Bostan--Caruso} \cite{Adamczewski--Bostan--Caruso sharper} investigated the automatic complexity of multivariate algebraic series $F$ over a perfect field of positive characteristic.
Taking the field to be $\F_p$ so that $\alpha = 1$, they obtained the same dimension $M = \prod_{i = 1}^m (h_i + 1)$ as in Theorem~\ref{kernel size upper bound - multivariate diagonals}, with a further refinement involving the total height of the annihilating polynomial of $F$.

For $m = 2$, in Theorem~\ref{kernel size upper bound - diagonals} we improve Theorem~\ref{kernel size upper bound - multivariate diagonals} to obtain the asymptotic bound $(1 + o(1)) p^{\alpha (\alpha + 1) (2 \alpha + 1) h d / 6}$.
New techniques would be required to extend this approach to $m \geq 3$.

\section{The module of possible states}\label{section: vector space - ring}

In this section, we recall the construction of the automaton generating $(a(n) \bmod p^\alpha)_{n \geq 0}$, where $p$ is a prime and $\alpha \geq 1$.
We assume the reader is familiar with deterministic finite automata with output;
see~\cite{Allouche--Shallit} for a comprehensive treatment and \cite{Rowland} for a short introduction.
An automaton with input alphabet $\{0, 1, \dots, p - 1\}$ generates the $p$-automatic sequence whose $n$th term is the output of the automaton when fed the standard base-$p$ representation of $n$, starting with the least significant digit.

Theorems~\ref{kernel size asymptotic bound - ring} and \ref{kernel size upper bound - ring} are concerned with automatic sequences with elements in $\ring$.
These sequences arise in the following result of Denef and Lipshitz~\cite[Lemma~6.3]{Denef--Lipshitz}, which extends Furstenberg's theorem~\cite{Furstenberg} to an integral domain.
We state it for the $p$-adic integers, along with an automaticity result proved by Denef and Lipshitz~\cite[Remark~6.6]{Denef--Lipshitz}; see also \cite[Theorem~2.1]{Rowland--Yassawi}.
The \emph{diagonal operator} $\mathcal{D}$, acting on bivariate power series, is defined by
\[
	\mathcal{D}\!\paren{\sum_{m \geq 0} \sum_{n \geq 0} a(m, n) x^m y^n}
	= \sum_{n \geq 0} a(n, n) x^n.
\]

\begin{theorem}[Denef and Lipshitz]\label{Furstenberg - ring}
Let $p$ be a prime.
Let $P \in \Z_p[x, y]$ such that $P(0, 0) = 0$ and $\frac{\partial P}{\partial y}(0, 0) \nequiv 0 \mod p$.
Let $F$ be the Furstenberg series associated with $P$.
Then
\[
	F = \mathcal{D}\!\paren{\frac{
		y \frac{\partial P}{\partial y}(x y, y)
	}{
		P(x y, y) / y
	}}.
\]
Moreover, for each $\alpha\geq 1$, the coefficient sequence of $F \bmod p^\alpha$ is $p$-automatic.
\end{theorem}

Since we will be working modulo~$p^\alpha$, we are mostly interested in sequences $a(n)_{n \geq 0}$ with entries from the set $\ring$.
We establish a correspondence between states of an automaton and Laurent polynomials in $\ring[x, y, y^{-1}]$.
We do this by identifying states first with sequences and then with power series.
Finally, Theorem~\ref{Furstenberg - ring} will allow us to identify automaton states with Laurent polynomials.
To do this we will use the \emph{$p$-kernel} of $a(n)_{n \geq 0}$, defined as
\[
	\ker_p(a(n)_{n \geq 0}) \colonequal \{a(p^e n + r)_{n \geq 0} : \text{$e \geq 0$ and $0 \leq r \leq p^e - 1$}\}.
\]
The smallest automaton that generates $a(n)_{n \geq 0}$ and that is not affected by leading $0$s is its \emph{minimal automaton}.
Eilenberg's theorem gives a bijection between the states of the minimal automaton and the elements of the $p$-kernel.

We represent kernel sequences $a(p^e n + r)_{n \geq 0}$ by their generating series $\sum_{n \geq 0} a(p^e n + r) x^n$.
Elements of the $p$-kernel can be accessed by applying the following operators.

\begin{definition*}
Let $n \in \Z$.
For each $r \in \{0, 1, \dots, p - 1\}$, define the \emph{Cartier operator} $\Lambda_r$ on the monomial $x^n$ by
\[
	\Lambda_r(x^n) =
	\begin{cases}
		x^\frac{n - r}{p}	& \text{if $n \equiv r \mod p$} \\
		0			& \text{otherwise}.
	\end{cases}
\]
Then extend $\Lambda_r$ linearly to polynomials, Laurent polynomials, and Laurent series in $x$ with coefficients in $\ring$.
In particular, for polynomials we have
\[
	\Lambda_r\!\paren{\sum_{n = 0}^N a(n) x^n}
	= \sum_{n = 0}^{\floor{N/p}} a(p n + r) x^n.
\]
Similarly, for $m, n \in \Z$ and $r, s \in \{0, 1, \dots, p - 1\}$, define the bivariate Cartier operator
\[
	\Lambda_{r, s}(x^m y^n) =
	\begin{cases}
		x^\frac{m - r}{p} y^\frac{n - s}{p}	& \text{if $m \equiv r \mod p$ and $n \equiv s \mod p$} \\
		0						& \text{otherwise},
	\end{cases}
\]
and extend $\Lambda_{r, s}$ linearly to bivariate polynomials, Laurent polynomials, and Laurent series.
\end{definition*}

The operator $\Lambda_r$ has the following useful property.

\begin{proposition}\label{Cartier - ring}
Let $p$ be a prime, let $\alpha \geq 1$, and let $r, s \in \{0, 1, \dots, p - 1\}$.
For all univariate Laurent series $F$ and $G$ with coefficients in $\ring$, we have $\Lambda_r(G F^{p^\alpha}) = \Lambda_r(G) F^{p^{\alpha - 1}}$.
Similarly, for all bivariate Laurent series $F$ and $G$ with coefficients in $\ring$, we have $\Lambda_{r, s}(G F^{p^\alpha}) = \Lambda_{r, s}(G) F^{p^{\alpha - 1}}$.
\end{proposition}

A proof of Proposition~\ref{Cartier - ring} can be found in \cite[Proposition~1.9]{Rowland--Yassawi}.
It uses the following lifting-the-exponent lemma for (Laurent) polynomials, whose proof is analogous to the proof for integers.
We will use Lemma~\ref{lifting-the-exponent} throughout this article.

\begin{lemma}\label{lifting-the-exponent}
If $Q, \bar{Q} \in \ring[x_1, x_1^{-1}, \dots, x_m, x_m^{-1}]$ and $Q \equiv \bar{Q} \mod p$, then $Q^{p^{\alpha - 1}} = \bar{Q}^{p^{\alpha - 1}}$.
\end{lemma}

The final step is to establish a connection between a power series $\sum_{n \geq 0} a(p^e n + r) x^n$ corresponding to a kernel sequence and a (not necessarily unique) Laurent polynomial.
This Laurent polynomial will be the numerator of a rational function whose diagonal is the desired power series.
We do this using Theorem~\ref{Furstenberg - ring}.
We shear bivariate series by replacing $x$ with $x y^{-1}$.
When we do this, the diagonal operator is replaced by the \emph{center row operator} $\mathcal{C}$, defined by
\[
	\mathcal{C}\!\paren{\sum_{m \geq 0} \sum_{n \in \Z} a(m, n) x^m y^n}
	= \sum_{m \geq 0} a(m, 0) x^m.
\]
We have
\[
	\Lambda_r \mathcal{C}\!\paren{\frac{S}{Q^{p^{\alpha - 1}}}}
	= \mathcal{C} \Lambda_{r, 0}\!\paren{\frac{S}{Q^{p^{\alpha - 1}}}}
	= \mathcal{C} \Lambda_{r, 0}\!\paren{\frac{S Q^{p^\alpha - p^{\alpha - 1}}}{Q^{p^\alpha}}} \\
	= \mathcal{C}\!\paren{\frac{\Lambda_{r, 0}\!\paren{S Q^{p^\alpha - p^{\alpha - 1}}}}{Q^{p^{\alpha - 1}}}},
\]
where the last equality follows from Proposition~\ref{Cartier - ring}.
Note that the initial and final series in this equation have the same denominator $Q^{p^{\alpha - 1}}$.
Therefore the map $S \mapsto \Lambda_{r, 0}\!\paren{S Q^{p^\alpha - p^{\alpha - 1}}}$ on $\ring[x, y, y^{-1}]$ emulates the Cartier operator $\Lambda_r$ on $\ring\doublebracket{x}$.
We will represent states of the automaton by Laurent polynomials $S \in \ring[x, y, y^{-1}]$.
The main reason for shearing is that it separates the contributions of $h$ and $d$; otherwise, in Proposition~\ref{digit bounds under Cartier} below, the $y$-degree would be bounded by a function of $h + d$ rather than just of $d$.

This final step of converting a power series to a Laurent polynomial is not bijective.
This is because different rational functions can have the same diagonal.
Therefore the automaton we will construct is not necessarily minimal.

We introduce notation for the initial state of the automaton and the emulating map as follows.

\begin{notation*}
Let $p$ be prime and $\alpha\geq 1$.
Let $P \in \Z_p[x, y]$ be a polynomial such that $P(0, 0) = 0$ and $\frac{\partial P}{\partial y}(0, 0) \nequiv 0 \mod p$.
Define $h \colonequal \deg_x(P \bmod p)$ and $d \colonequal \deg_y(P \bmod p)$.
The Furstenberg series $F$ given by Theorem~\ref{Furstenberg - ring} has coefficients in $\Z_p$.
We define a Laurent polynomial $Q$ as follows.
Take $P/y \bmod p$, and let $Q \in \ring[x, y, y^{-1}]$ be a lift of $P/y \bmod p$ which has the same monomial support as $P/y \bmod p$; in particular $\deg_x Q = h$ and $\deg_y Q = d - 1$.
For example, an element of $\mathcal R_p[x, y, y^{-1}]$ can be lifted in a standard way to $\ring[x, y, y^{-1}]$ by identifying $\mathcal R_p$ with $D = \{0, 1, \dots, p - 1\}$.
This somewhat convoluted definition of $Q$ allows us obtain optimal bounds even in the situation that the $x$- or $y$-degree of $P/y \bmod p$ is less than that of $P/y \bmod p^\alpha$; if neither degree drops, then one can simply take $Q$ to be $P/y \bmod p^\alpha$.
We define an automaton whose initial state is
\[
	S_0 \colonequal \paren{y \tfrac{\partial P}{\partial y} \, (P/y)^{p^{\alpha - 1} - 1} \bmod p^\alpha}.
\]
For each $r \in \{0, 1, \dots, p - 1\}$, define $\lambda_{r, 0} \colon \ring[x, y, y^{-1}] \to \ring[x, y, y^{-1}]$ by
\begin{equation}\label{little lambda definition - ring}
	\lambda_{r, 0}(S)
	\colonequal
	\Lambda_{r, 0}\!\paren{S Q^{p^\alpha - p^{\alpha - 1}}}.
\end{equation}
Define $\mathcal M_{p^\alpha}$ to be the smallest subset of $\ring[x, y, y^{-1}]$ that contains $S_0$ and is closed under the operators $\lambda_{r,0}$.
The operators $\lambda_{r,0}$ define the transitions between states.
Finally, the output associated with the state $S \in \mathcal M_{p^\alpha}$ is the constant term of $S$ divided by the constant term of $Q^{p^{\alpha - 1}}$.
\end{notation*}

\begin{remark}\label{h and d positive}
The condition $\frac{\partial P}{\partial y}(0, 0) \nequiv 0 \mod p$ implies $d \geq 1$.
Furthermore, if $F \bmod p^\alpha$ is not a polynomial then $h \geq 1$.
This is because $h = 0$ implies that the coefficient of each monomial $x^i y^j$ in $P$ with $i \geq 1$ is $0$ modulo~$p$.
Therefore the denominator $P(x y, y)/y$ in Theorem~\ref{Furstenberg - ring} also has this property.
Using the geometric series formula to expand the rational expression, we obtain a bivariate series that contains only finitely many nonzero diagonal monomials modulo~$p^\alpha$.
\end{remark}

Note that, a priori, the operator $\lambda_{r, 0}$ in Equation~\eqref{little lambda definition - ring} should be defined with $P/y\bmod p^\alpha$ in place of $Q$.
However, Lemma~\ref{lifting-the-exponent} gives us the following.

\begin{proposition}\label{lifting-the-exponent lambda}
For all $S \in \ring[x, y, y^{-1}]$ and for each $r \in \{0, 1, \dots, p - 1\}$, we have
\[
	\Lambda_{r, 0}\!\paren{S (P/y \bmod p^\alpha)^{p^\alpha - p^{\alpha - 1}}}
	= \Lambda_{r, 0}\!\paren{S Q^{p^\alpha - p^{\alpha - 1}}}
	= \lambda_{r, 0}(S).
\]
\end{proposition}

We will use the following elementary lemma.

\begin{lemma}\label{Minkowski}
The set of bivariate Laurent polynomials with the property that every nonzero monomial $c \, x^I y^J$ satisfies $J\geq -I$ is closed under addition, multiplication, and each Cartier operator $\Lambda_{r,0}$.
\end{lemma}

\begin{proposition}
The constructed automaton is not sensitive to leading $0$s.
\end{proposition}

\begin{proof}
Let $S \in \mathcal M_{p^\alpha}$, let $x^i y^j$ be a stripped monomial in $S$ (that is, a monomial without its coefficient), and let $x^I y^J$ be a stripped monomial in $Q^{p^\alpha - p^{\alpha - 1}}$;
we are interested in products $x^i y^j \cdot x^I y^J$ which equal $1$, since these contribute to the constant term and therefore the output value.
This implies $I = -i$ and $J = -j$.
Since $i \geq 0$ and $I \geq 0$, we have $i = 0 = -I$.
We consider the three cases $j < 0$, $j > 0$, and $j = 0$.
If $j > 0$, then $J < 0$; this implies $I > 0$ by the condition $P(0, 0) \neq 0$, which contradicts $I = 0$.
If $j < 0$, then $-i \leq j < 0$ by Lemma~\ref{Minkowski}, which implies $i < 0$ and contradicts $i = 0$.
Therefore $j = 0 = J$, so the constant term of $S Q^{p^\alpha - p^{\alpha - 1}}$ is the product of the constant term of $S$ and the constant term of $Q^{p^\alpha - p^{\alpha - 1}}$.
The assumption $\frac{\partial P}{\partial y}(0, 0) \nequiv 0 \mod p$ implies that the constant term $c \, x^0 y^0$ of $Q$ is nonzero modulo~$p$.
It follows that the constant term of $Q^{p^\alpha - p^{\alpha - 1}}$ is $c^{p^\alpha - p^{\alpha - 1}} = 1$, since the Euler totient function satisfies $\phi(p^\alpha) = p^\alpha - p^{\alpha - 1}$.
Therefore the constant term of $\lambda_{0, 0}(S)$ divided by the constant term of $Q$, which is the output assigned to the state $\lambda_{0, 0}(S)$, is the constant term of $S$ divided by the constant term of $Q$.
\end{proof}

\section{A numeration system for the automaton states}\label{section: image - ring}

In this section, we define a numeration system for a certain set of Laurent polynomials.
We then show in Theorems~\ref{closure under Cartier} and \ref{state base-p/Q representation} that all states of the automaton $\mathcal M_{p^\alpha}$ have a representation in this numeration system.

We continue to assume that $p$ is a prime number and $\alpha \geq 1$.
We also continue to use $Q$ as defined in the previous section, although the only property we need for a numeration system is that the constant term of $Q$ is nonzero modulo~$p$.

\begin{definition*}
Let $D = \{0, 1, \dots, p - 1\}$; we view $D \subseteq \ring$.
We say that $S \in \ring[x, y, y^{-1}]$ has a \emph{base-$\frac{p}{Q}$ representation} if there are Laurent polynomials $T_0, T_1, \dots, T_{\alpha - 1}$ such that $T_k\in D[x, y, y^{-1}]$ for each $k$ and
\begin{equation}\label{series form}
	S =
	\left(T_0
	+ T_1 \, \tfrac{p}{Q}
	+ T_2 \, (\tfrac{p}{Q})^2
	+ \dots
	+ T_{\alpha - 1} \, (\tfrac{p}{Q})^{\alpha - 1}\right)
	Q^{p^{\alpha - 1} - 1}.
\end{equation}
We refer to $T_0, T_1, \dots, T_{\alpha - 1}$ as \emph{digits}.
Since $p^{\alpha - 1} - 1 \geq \alpha - 1$, the right side of Equation~\eqref{series form} is a Laurent polynomial.
\end{definition*}

Because of the factor $Q^{p^{\alpha - 1} - 1}$, this numeration system is not exactly analogous to classical base-$p$ representations of integers.
In addition, the set of possible digits is currently infinite; later we will restrict it.
We start by showing that, like classical numeration systems, base-$\frac{p}{Q}$ representations have two desirable properties.

\begin{proposition}\label{uniqueness of representation}
If the Laurent polynomial $S\in \ring[x, y, y^{-1}]$ has a base-$\frac{p}{Q}$ representation, then this representation is unique.
\end{proposition}

\begin{proof}
Suppose that
\begin{multline}\label{unique representation}
	\paren{T_0
	+ T_1 \, \tfrac{p}{Q}
	+ \dots
	+ T_{\alpha - 1} \, (\tfrac{p}{Q})^{\alpha - 1}}
	Q^{p^{\alpha - 1} - 1} \\
	\equiv
	\paren{U_0
	+ U_1 \, \tfrac{p}{Q}
	+ \dots
	+ U_{\alpha - 1} \, (\tfrac{p}{Q})^{\alpha - 1}}
	Q^{p^{\alpha - 1} - 1}
	\mod p^\alpha,
\end{multline}
where the digits $T_k$ and $U_k$ belong to $D[x, y, y^{-1}]$.
Then $T_0 Q^{p^{\alpha - 1} - 1} \equiv U_0 Q^{p^{\alpha - 1} - 1} \mod p$.
Since $Q$ has a constant term which is nonzero modulo~$p$, it is invertible, so $T_0 \equiv U_0 \mod p$, which implies $T_0 = U_0$.
Inductively, suppose that $T_m=U_m$ for $0\leq m \leq k-1$.
Reducing Equation~\eqref{unique representation} modulo~$p^{k+1}$ and expanding gives us
\begin{multline*}
	T_0 Q^{p^{\alpha - 1} - 1} + \dots + T_k p^k Q^{p^{\alpha - 1} - k-1} \\
	\equiv U_0 Q^{p^{\alpha - 1} - 1} + \dots + U_k p^k Q^{p^{\alpha - 1} - k-1} \mod p^{k+1}.
\end{multline*}
Subtracting the first $k$ terms from both sides and dividing by $Q^{p^{\alpha - 1} - k-1}$, we conclude that $T_k = U_k$.
\end{proof}

The next proposition allows us to perform carries and therefore normalize base-$\frac{p}{Q}$ representations where the digit coefficients are not in $D$.
This implies that the set of Laurent polynomials with base-$\frac{p}{Q}$ representations is closed under addition and scalar multiplication; we will use this vector space structure in later sections.

\begin{proposition}\label{carries}
Suppose $S \in \ring[x, y, y^{-1}]$ is of the form
\[
	\left(T_0'
	+ T_1' \, \tfrac{p}{Q}
	+ T_2' \, (\tfrac{p}{Q})^2
	+ \dots
	+ T_{\alpha - 1}' \, (\tfrac{p}{Q})^{\alpha - 1}\right)
	Q^{p^{\alpha - 1} - 1}
\]
where $T_k' \in \ring[x, y, y^{-1}]$ for each $k \in \{0, 1, \dots, \alpha - 1\}$.
Then $S$ has a base-$\frac{p}{Q}$ representation.
\end{proposition}

\begin{proof}
To put $S$ into the desired form, so that it has a representation with digits $T_k \in D[x, y, y^{-1}]$, it is sufficient to show how to perform a carry from one digit to the next.
We begin the procedure by setting $T_0''=T_0'$.
Given $T_k''$, we perform division to write it as $T_k''=pU_k+R_k$ with $R_k \in D[x, y, y^{-1}]$.
Then
\begin{align*}
	S
	&= \left(
		T_0 + \dots
		+ (R_k + p U_k) \, (\tfrac{p}{Q})^k
		+ T_{k + 1}' \, (\tfrac{p}{Q})^{k + 1} + \dots
		+ T_{\alpha - 1}' \, (\tfrac{p}{Q})^{\alpha - 1}
	\right)
	Q^{p^{\alpha - 1} - 1} \\
	&= \left(
		T_0 +
		\dots
		+ R_k \, (\tfrac{p}{Q})^k
		+ (U_k Q + T_{k + 1}') \, (\tfrac{p}{Q})^{k + 1}
		+ \dots
		+ T_{\alpha - 1}' \, (\tfrac{p}{Q})^{\alpha - 1}
	\right)
	Q^{p^{\alpha - 1} - 1},
\end{align*}
so that we can set $T_k = R_k$.
For $k = \alpha - 1$, the quotient $U_{\alpha - 1}$ plays no role since $U_{\alpha - 1} Q \, (\tfrac{p}{Q})^\alpha \equiv 0 \mod p^\alpha$.
\end{proof}

\begin{notation*}
If $S$ has a base-$\frac{p}{Q}$ representation, then $S$ can be written uniquely as $S = \left(\sum_{k = 0}^{\alpha - 1} T_k \, (\frac{p}{Q})^k\right) Q^{p^{\alpha - 1} - 1}$ by Proposition~\ref{uniqueness of representation}.
Define
\begin{equation}\label{rep definition}
	\rep_{p/Q}(S) \colonequal (T_{\alpha - 1}, \dots, T_1, T_0).
\end{equation}
Let $\dig_k(S)$ denote the $k$th digit $T_k$ in $\rep_{p/Q}(S)$.
Finally, define
\[
	\val_{p/Q}((T_{\alpha - 1}, \dots, T_1, T_0)) \colonequal \left(\sum_{k = 0}^{\alpha - 1} T_k \, (\tfrac{p}{Q})^k\right) Q^{p^{\alpha - 1} - 1}.
\]
\end{notation*}

\begin{example}\label{state base-p/Q representation example}
Let $P = x y^2 + (x + 1) y + x$.
The coefficient sequence $a(n)_{n \geq 0}$ of the Furstenberg series $F$ satisfying $P(x, F) = 0$ is $0, -1, 1, -2, 4, -9, 21, -51, \dots$ and is a signed, shifted variant of the sequence of Motzkin numbers~\seq{A001006}.
Let $p = 2$ and $\alpha = 3$.
Consider the initial state $S_0 \in \mathcal{R}_{8}[x, y, y^{-1}]$ of the automaton generating $(a(n) \bmod 8)_{n \geq 0}$.
By definition,
\begin{multline*}
	S_0
	= \paren{y \tfrac{\partial P}{\partial y} \, (P/y)^3 \bmod 8}
	=
		2 x^4 y^5
		+ \left(7 x^4 + 7 x^3\right) y^4
		+ \left(7 x^4 + 2 x^3 + x^2\right) y^3 \\
		+ \left(4 x^4 + 6 x^3 + 7 x^2 + 5 x\right) y^2
		+ \left(3 x^4 + 4 x^3 + 2 x^2 + 4 x + 1\right) y \\
		+ \left(4 x^4 + 2 x^3 + x^2 + 3 x\right)
		+ \left(5 x^4 + 6 x^3 + 3 x^2\right) y^{-1}
		+ \left(x^4 + x^3\right) y^{-2}.
\end{multline*}
By Theorem~\ref{state base-p/Q representation}, which we state below, the initial state has a base-$\frac{p}{Q}$ representation, where $Q = x y + (x + 1) + x y^{-1} \in \mathcal{R}_{8}[x, y, y^{-1}]$, namely
\[
	S_0 = \left((x + 1) y + \left(x^2 y^3 + (x^2 + x) y^2 + x^2 y\right) \tfrac{2}{Q} + 0 \cdot \tfrac{4}{Q^2}\right) Q^3,
\]
so that
\[
	\rep_{2/Q}(S_0) = \left(0, \, x^2 y^3 + (x^2 + x) y^2 + x^2 y, \, (x + 1) y\right).
\]
\end{example}

\begin{theorem}\label{closure under Cartier}
The set of Laurent polynomials in $\ring[x, y, y^{-1}]$ with a base-$\frac{p}{Q}$ representation is closed under $\lambda_{r,0}$ for $0\leq r \leq p-1$.
\end{theorem}

The proof uses Lemma~\ref{polynomial modulo prime power} below.
For a Laurent polynomial $S \in \ring[x, y, y^{-1}]$, let $\mindeg_y S$ be the smallest $y$-degree of a monomial in $S$ with a nonzero coefficient.
If $\beta = 1$, then the statement of Lemma~\ref{polynomial modulo prime power} follows from Proposition~\ref{Cartier - ring}, since, modulo~$p$, we have $\Lambda_{r, 0}(\frac{T}{Q^{k + 1}}) Q^{k + 1} \equiv \Lambda_{r, 0}(\frac{T Q^{p (k + 1)}}{Q^{k + 1}}) = \Lambda_{r, 0}(T Q^{(p - 1) (k + 1)})$, and the latter is a Laurent polynomial; it can be verified that it satisfies the claimed degree bounds.

\begin{lemma}\label{polynomial modulo prime power}
Let $\beta \in \{1, 2, \dots, \alpha\}$, $T \in \mathcal R_{p^\beta}[x, y, y^{-1}]$, and $r \in \{0, 1, \dots, p - 1\}$.
For all $k \geq 0$, the Laurent series $S\colonequal\Lambda_{r, 0}(\frac{T}{(Q \bmod p^\beta)^{k + 1}}) (Q \bmod p^{\beta})^{k + \beta}$ is a Laurent polynomial.
Moreover, its degrees satisfy
\begin{align*}
	\deg_x S &\leq (k + \beta) h + \floor{\tfrac{\deg_x T - (k + 1) h - r}{p}} \\
	\deg_y S &\leq (k + \beta) (d-1) + \floor{\tfrac{\deg_y T - (k + 1) (d-1)}{p}} \\
	\mindeg_y S &\geq -(k + \beta) + \ceil{\tfrac{\mindeg_y T + (k + 1)}{p}}.
\end{align*}
Finally, if each nonzero monomial $c \, x^I y^J$ in $T$ satisfies $J \geq -I$, then so does each nonzero monomial in $S$.
\end{lemma}

\begin{proof}
For ease of notation, in this proof we use $Q$ to refer to $Q\bmod p^\beta$.
Since $Q(x^p, y^p) \equiv Q(x, y)^p \mod p$, we have $\Delta \colonequal Q(x^p, y^p) - Q(x, y)^p \equiv 0 \mod p$.
Therefore
\[
	\frac{Q(x^p, y^p)^{k + \beta}}{Q^{k + 1}}
	= \frac{(Q^p + \Delta)^{k + \beta}}{Q^{k + 1}}
	= \sum_{m = 0}^{k + \beta} \binom{k + \beta}{m} Q^{p m - k - 1} \Delta^{k + \beta - m}.
\]
For $m \in \{0, 1, \dots, k\}$, the summand is $0$ since $\Delta^\beta = 0$.
Therefore
\begin{equation}\label{auxiliary polynomial}
	\frac{Q(x^p, y^p)^{k + \beta}}{Q^{k + 1}}
	= \sum_{m = k + 1}^{k + \beta} \binom{k + \beta}{m} Q^{p m - k - 1} \Delta^{k + \beta - m}
	\equalcolon Z.
\end{equation}
This Laurent series $Z$ is a Laurent polynomial since, for all $m \geq k + 1$, we have $p m - k - 1 \geq p (k + 1) - k - 1 \geq 0$.
Let $z \in \{x, y\}$.
Since $\deg_z \Delta \leq p \deg_z Q$, the degree of $Q^{p m - k - 1} \Delta^{k + \beta - m}$ is at most $((k + \beta) p - k - 1) \deg_z Q$, so this is also an upper bound on $\deg_z Z$.
Analogously, $\mindeg_y Z\geq ((k + \beta) p - k - 1) \mindeg_y Q \geq - ((k + \beta) p - k - 1)$.

Multiplying both sides by $T$, we have $\frac{TQ(x^p,y^p)^{k + \beta}}{Q^{k + 1}} = TZ$.
Since the Laurent polynomial $Q(x^p, y^p)^{k + \beta}$ consists of terms with exponents that are multiples of $p$, applying $\Lambda_{r, 0}$ to both sides gives $\Lambda_{r, 0}(\frac{T}{Q^{k + 1}}) Q^{k + \beta} = \Lambda_{r, 0}(TZ)$.
In particular, $\Lambda_{r, 0}(\frac{T}{Q^{k + 1}}) Q^{k + \beta}$ is a Laurent polynomial.
Moreover, its $z$-degree and $y$-min-degree are as claimed, using the bounds in the previous paragraph on $\deg_z Z$ and $\mindeg_y Z$.
Finally, since the coefficient of $x^0 y^{-1}$ in $Q$ is $0$, it follows from Lemma~\ref{Minkowski} that each nonzero monomial $c \, x^I y^J$ in $\Delta$ and $Z$ satisfies $J \geq -I$.
Therefore if each nonzero monomial in $T$ also satisfies this constraint, then, again by Lemma~\ref{Minkowski}, so does each nonzero monomial in $S$.
\end{proof}

We can now establish closure under $\lambda_{r,0}$.

\begin{proof}[Proof of Theorem~\ref{closure under Cartier}]
Let $S$ be a Laurent polynomial with a base-$\frac{p}{Q}$ representation, and let $\rep_{p/Q}(S) = (T_{\alpha - 1}, \dots, T_1, T_0)$, so that
\[
	S =
	\left(T_0
	+ T_1 \, \tfrac{p}{Q}
	+ T_2 \, (\tfrac{p}{Q})^2
	+ \dots
	+ T_{\alpha - 1} \, (\tfrac{p}{Q})^{\alpha - 1}\right)
	Q^{p^{\alpha - 1} - 1}.
\]
We show that $\lambda_{r, 0}(S)$ also has this form.
To do this, for each $k$, we construct the base-$\frac{p}{Q}$ representation of $\Lambda_{r, 0}\!\left(T_k Q^{-k - 1}\right) Q^\alpha$.
The digits of these representations are Laurent polynomials $U_{k,j}\in D[x, y, y^{-1}]$, and we will recombine these Laurent polynomials to obtain the base-$\frac{p}{Q}$ representation of $\lambda_{r, 0}(S)$.

Let $k \in \{0, 1, \dots, \alpha - 1\}$.
We define Laurent polynomials $U_{k, j}$ as follows.
Lemma~\ref{polynomial modulo prime power} with $\beta = j+1$ implies that $\Lambda_{r, 0}\!\paren{T_k Q^{-k - 1}} Q^{k+j+1} \bmod p^{j+1}$ is a Laurent polynomial for all $j \geq 0$.
For each $j \in \{0,1, \dots, \alpha -k-1\}$, define $U_{k, j} \in D[x, y, y^{-1}]$ to be the Laurent polynomial satisfying
\begin{align*}
	\Lambda_{r, 0}\!\left(T_k Q^{-k-1}\right) Q^{k+1} &\equiv U_{k,0} \mod p \\
	\Lambda_{r, 0}\!\left(T_k Q^{-k-1}\right) Q^{k+2} &\equiv U_{k,0} Q + p U_{k,1} \mod p^2 \\
	\Lambda_{r, 0}\!\left(T_k Q^{-k-1}\right) Q^{k+3} &\equiv U_{k,0} Q^2 + p U_{k,1} Q + p^2 U_{k,2} \mod p^3 \\
	&\;\;\vdots \\
	\Lambda_{r, 0}\!\left(T_k Q^{-k - 1}\right) Q^\alpha &\equiv U_{k,0} Q^{\alpha - k - 1} + p U_{k,1} Q^{\alpha - k - 2} + \dots + p^{\alpha - k - 1} U_{k, \alpha - k - 1} \mod p^{\alpha - k}.
\end{align*}
The Laurent polynomials $U_{k,j}$ can be recursively computed using the proof of Lemma~\ref{polynomial modulo prime power}.
Dividing the previous congruence by $Q^{\alpha - 1}$, we obtain the Laurent series congruence
\begin{equation}\label{digit series}
	\Lambda_{r, 0}\!\left(T_k Q^{-k - 1}\right) Q \equiv \sum_{j = 0}^{\alpha - k - 1} \frac{U_{k, j}}{Q^{k + j}} \, p^j \mod p^{\alpha - k}.
\end{equation}

Using Proposition~\ref{Cartier - ring}, we have
\begin{align*}
	\lambda_{r, 0}(S)
	&= \lambda_{r, 0}\!\left(
		\paren{\sum_{k = 0}^{\alpha - 1} T_k \, (\tfrac{p}{Q})^k} Q^{p^{\alpha - 1} - 1}
	\right) \\
	&= \Lambda_{r, 0}\!\left(
		\paren{\sum_{k = 0}^{\alpha - 1} T_k \, (\tfrac{p}{Q})^k} Q^{-1} Q^{p^\alpha}
	\right) \\
	&= \Lambda_{r, 0}\!\left(
		\paren{\sum_{k = 0}^{\alpha - 1} T_k \, (\tfrac{p}{Q})^k} Q^{-1}
	\right) Q^{p^{\alpha - 1}} \\
	&= \paren{\sum_{k = 0}^{\alpha - 1} \Lambda_{r, 0}\!\left(T_k Q^{-k - 1}\right) p^k Q} Q^{p^{\alpha - 1} - 1}.
\end{align*}
By Equation~\eqref{digit series}, we have
\begin{align*}
	\lambda_{r, 0}(S)
	&= \paren{\sum_{k = 0}^{\alpha - 1} \sum_{j = 0}^{\alpha - k - 1} U_{k, j} \, \tfrac{p^{k + j}}{Q^{k + j}}} Q^{p^{\alpha - 1} - 1} \\
	&= \paren{\sum_{k = 0}^{\alpha - 1} \sum_{m = k}^{\alpha - 1} U_{k, m - k} \, (\tfrac{p}{Q})^m} Q^{p^{\alpha - 1} - 1}
\end{align*}
after substituting $j = m - k$.
Switching the order of summation gives
\[
	\lambda_{r, 0}(S)
	= \paren{\sum_{m = 0}^{\alpha - 1} \paren{\sum_{k = 0}^m U_{k, m - k}} (\tfrac{p}{Q})^m} Q^{p^{\alpha - 1} - 1}.
\]
For each $m \in \{0, 1, \dots, \alpha - 1\}$, set $T_m' = \sum_{k = 0}^m U_{k, m - k}$.
The coefficients in the Laurent polynomial $T_m'$ do not necessarily belong to $D$.
To obtain the base-$\frac{p}{Q}$ representation of $\lambda_{r, 0}(S)$, we perform carries as in Proposition~\ref{carries}.
\end{proof}

\begin{example}\label{closure under Cartier example}
As in Example~\ref{state base-p/Q representation example}, let $P = x y^2 + (x + 1) y + x$, $p = 2$, $\alpha = 3$, and $Q = x y + (x + 1) + x y^{-1} \in \mathcal{R}_{8}[x, y, y^{-1}]$; we saw that $\rep_{2/Q}(S_0)$ is
\[
	(T_2, T_1, T_0) \colonequal \left(0, \, x^2 y^3 + (x^2 + x) y^2 + x^2 y, \, (x + 1) y\right).
\]
We follow the proof of Theorem~\ref{closure under Cartier} to compute $\rep_{2/Q}(\lambda_{0, 0}(S_0))$.
The first step is to compute $U_{0, 0}, U_{0, 1}, U_{0, 2}, U_{1, 0}, U_{1, 1}, U_{2, 0} \in D[x, y, y^{-1}]$.
For $U_{0, 0}$, we can use Proposition~\ref{Cartier - ring} to obtain
\[
	U_{0, 0}
	\equiv \Lambda_{0, 0}\!\left(T_0 Q^{-1}\right) Q
	\equiv \Lambda_{0, 0}(T_0 Q)
	\equiv x y + x
	\mod 2.
\]
Therefore $U_{0, 0} = x y + x$.
For the others, we use Lemma~\ref{polynomial modulo prime power}.
Set
\[
	\Delta = Q(x^2, y^2) - Q(x, y)^2
	= \left(6 x^2 + 6 x\right) y + \left(6 x^2 + 6 x\right) + \left(6 x^2 + 6 x\right) y^{-1}.
\]
The Laurent polynomial $U_{0,1}$ is defined by $\Lambda_{0, 0}\!\left(T_0 Q^{-1}\right) Q^2 \equiv U_{0,0} Q + 2 U_{0,1} \mod 4$.
To compute $\Lambda_{r, 0}\!\left(T_0 Q^{-1}\right) Q^2$, we use the proof of Lemma~\ref{polynomial modulo prime power} with $k = 0$ and $\beta = 2$.
Let
\[
	Z \colonequal \sum_{m = 1}^{2} \binom{2}{m} Q^{2 m - 1} \Delta^{2 - m} = 2 Q \Delta + Q^3.
\]
Then
\[
	\Lambda_{0, 0}\!\left(T_0 Q^{-1}\right) Q^2
	\equiv \Lambda_{0, 0}\!\left(T_0 Z\right)
	\equiv x^2 y^2 + \left(2 x^2 + x\right) y + \left(2 x^2 + x\right) + x^2 y^{-1}
	\mod 4.
\]
This implies $2 U_{0,1} \equiv \Lambda_{0, 0}\!\left(T_0 Q^{-1}\right) Q^2 - U_{0,0} Q \equiv 0 \mod 4$, so $U_{0,1} = 0$.
The remaining $U_{k, j}$ are computed analogously, and we find
\begin{align*}
	U_{0, 0} &= x y + x						& U_{1, 0} &= x^2 y^2 + \left(x^2 + x\right) y + x^2 \\
	U_{0, 1} &= 0							& U_{1, 1} &= x^3 y^3 + \left(x^3 + x\right) y + x^3 y^{-1} \\
	U_{0, 2} &= x^2 y^2 + x^2 y + x^2 + x^2 y^{-1}	& U_{2, 0} &= 0.
\end{align*}
The second step is to compute the new digits $T_m' = \sum_{k = 0}^m U_{k, m - k}$.
We obtain
\begin{align*}
	T_0' &= U_{0, 0} = x y + x \\
	T_1' &= U_{0, 1} + U_{1, 0} = x^2 y^2 + \left(x^2 + x\right) y + x^2 \\
	T_2' &= U_{0, 2} + U_{1, 1} + U_{2, 0} = x^3 y^3 + x^2 y^2 + \left(x^3 + x^2 + x\right) y + x^2 + \left(x^3 + x^2\right) y^{-1}.
\end{align*}
The third step is to perform carries to normalize the base-$\frac{2}{Q}$ representation.
In this example, it happens that no carries are necessary.
Therefore $\rep_{2/Q}(\lambda_{0, 0}(S_0)) = (T_2', T_1', T_0')$.
\end{example}

Theorem~\ref{closure under Cartier} gets us most of the way to the following result.
Recall that $\mathcal M_{p^\alpha}$ is the set of states of an (unminimized) automaton that generates the sequence $(a(n) \bmod p^\alpha)_{n \geq 0}$.

\begin{theorem}\label{state base-p/Q representation}
Every state in $\mathcal M_{p^\alpha}$ has a unique base-$\frac{p}{Q}$ representation.
\end{theorem}

To prove Theorem~\ref{state base-p/Q representation}, and later Proposition~\ref{transient - ring}, we will use the following lemma.

\begin{lemma}\label{polynomial quotient}
For each $k \in \{0, 1, \dots, \alpha - 1\}$, the Laurent series $S \colonequal \left(\frac{Q^{k+1}}{P/y} \bmod p^{k+1}\right)$ is a Laurent polynomial.
Moreover, its degrees satisfy
\begin{align*}
	\deg_x S &\leq k \deg_x(P \bmod p^{k + 1}) \\
	\deg_y S &\leq k \left(\deg_y(P \bmod p^{k + 1}) - 1\right) \\
	\mindeg_y S &\geq -k.
\end{align*}
\end{lemma}

\begin{proof}
Fix $k$.
Since $Q \equiv P/y \mod p$, we have $\Delta \colonequal (Q - P/y \bmod p^{k + 1}) \equiv 0 \mod p$.
Therefore
\[
	\frac{Q^{k + 1}}{P/y}
	\equiv \frac{(P/y + \Delta)^{k + 1}}{P/y}
	\equiv \sum_{m = 0}^{k + 1} \binom{k + 1}{m} (P/y)^{m - 1} \Delta^{k + 1 - m}
	\mod p^{k + 1}.
\]
For $m = 0$, the summand is $0$ modulo~$p^{k + 1}$ since $\Delta$ is divisible by $p$.
Therefore $\frac{Q^{k+1}}{P/y} \bmod p^{k+1}$ is a Laurent polynomial.
Finally, the degree bounds follow from the fact that $\deg_z \Delta \leq \deg_z(P/y \bmod p^{k + 1})$ for $z \in \{x, y\}$, and $\mindeg_y \Delta \geq \mindeg_y(P/y \bmod p^{k + 1})$.
\end{proof}

We now use Lemma~\ref{polynomial quotient} to prove Theorem~\ref{state base-p/Q representation}.

\begin{proof}[Proof of Theorem~\ref{state base-p/Q representation}]
We show that the initial state $S_0 = \paren{y \frac{\partial P}{\partial y} \, (P/y)^{p^{\alpha - 1} - 1} \bmod p^\alpha}$ has a base-$\frac{p}{Q}$ representation; the result then follows from Theorem~\ref{closure under Cartier} and Proposition~\ref{uniqueness of representation}.
We use the convention that whenever we multiply or divide a polynomial in $\mathcal R_{p^k}[x, y, y^{-1}]$ by a polynomial in $\Z_p[x, y, y^{-1}]$, we project the latter to $\mathcal R_{p^k}[x, y, y^{-1}]$.
Conversely, we will define the digits $T_k$ of $S_0$ to be in $\mathcal R_p[x, y, y^{-1}]$, but then we identify $\mathcal R_p$ with $D$ so that we can do arithmetic in $\ring$ as in Equation~\eqref{series form}.
We will use $(P/y)^{p^{\alpha - 1}} \equiv Q^{p^{\alpha - 1}} \mod p^\alpha$, which follows from Lemma~\ref{lifting-the-exponent}.
Set $T_0 = \paren{y \frac{\partial P}{\partial y} \bmod p}$; then $S_0 \equiv y \frac{\partial P}{\partial y} \, (P/y)^{p^{\alpha - 1} - 1} \equiv T_0 Q^{p^{\alpha - 1} - 1} \mod p$.
To define the digit $T_1$, we use $0 \equiv y \frac{\partial P}{\partial y} - T_0 \equiv y \frac{\partial P}{\partial y} \cdot \frac{Q}{P/y} - T_0 \mod p$, and set
\[
	T_1 = \paren{\frac{y \tfrac{\partial P}{\partial y} \cdot \frac{Q}{P/y} - T_0}{p} Q \bmod p}.
\]
We have $T_1 \in \mathcal R_p[x, y, y^{-1}]$ since $\frac{Q^2}{P/y} \in \mathcal R_{p^2}[x, y, y^{-1}]$ by Lemma~\ref{polynomial quotient}.
Then
\[
	y \tfrac{\partial P}{\partial y} \cdot \tfrac{Q^{p^{\alpha - 1}}}{P/y} \equiv \left(T_0 + T_1 \tfrac{p}{Q}\right) Q^{p^{\alpha - 1}-1} \mod p^2,
\]
so we have $S_0 \equiv y \frac{\partial P}{\partial y} \, (P/y)^{p^{\alpha - 1} - 1} \equiv \left(T_0 + T_1 \frac{p}{Q}\right) Q^{p^{\alpha - 1} - 1} \mod p^2$.
Recursively, for each $k \in \{2, \dots, \alpha - 1\}$, define the Laurent polynomial
\[
	T_k = \paren{\frac{y \tfrac{\partial P}{\partial y} \cdot \frac{Q}{P/y} - T_0 - T_1 \, (\tfrac{p}{Q}) - \dots - T_{k - 1} \, (\tfrac{p}{Q})^{k - 1}}{p^k} Q^k \bmod p}.
\]
By Lemma~\ref{polynomial quotient}, we have $T_k \in \mathcal R_p[x, y, y^{-1}]$.
Then
$S_0 \equiv
	\left(T_0
	+ T_1 \frac{p}{Q}
	+ \dots + T_k \, (\frac{p}{Q})^k\right)
	Q^{p^{\alpha - 1} - 1} \mod p^{k + 1}
$ for each $k$.
In particular, $S_0$ has a base-$\frac{p}{Q}$ representation.
\end{proof}

\section{Compatibility of automata}\label{section: inverse limit}

Let $\mathcal M_{p^\alpha}$ be the automaton in Theorem~\ref{state base-p/Q representation} that generates the coefficients of $F \bmod p^\alpha$.
In this section, we show that the base-$\frac{p}{Q}$ representations of states in $\mathcal M_{p^{\alpha+1}}$ project onto those of states in $\mathcal M_{p^\alpha}$.
This implies that, when we vary $\alpha$, the automata $\mathcal M_{p^\alpha}$ support an inverse limit.
In a previous article~\cite{Rowland--Yassawi profinite}, the authors proved this for a sequence of automata reading most significant digit first instead, but no explicit description of its states was known.
The advantage of the construction below is that it gives a description of the states as sequences of base-$\frac{p}{Q}$ digits.
The results of this section are not needed for the remainder of the article.
We begin with an example.

\begin{example}
Let $P$ be the polynomial in Examples~\ref{state base-p/Q representation example} and \ref{closure under Cartier example} with $p = 2$, but now let $\alpha = 2$.
We consider the set of states $\mathcal M_{4}$ in the automaton generating the sequence $(a(n) \bmod 4)_{n \geq 0}$.
Let $S_0 = \paren{y \frac{\partial P}{\partial y} \, (P/y) \bmod 4}$.
The orbit of $S_0$ under $\lambda_{0, 0}$ begins
\begin{align*}
	\rep_{p/Q}(S_0) &= \left(x^2 y^3 + (x^2 + x) y^2 + x^2 y, \, (x + 1) y\right) \\
	\rep_{p/Q}(\lambda_{0, 0}(S_0)) &= \left(x^2 y^2 + (x^2 + x) y + x^2, \, x y + x\right).
\end{align*}
For each of these two states, the two digits $T_0$ and $T_1$ are the same as for the corresponding states modulo~$8$ in Examples~\ref{state base-p/Q representation example} and \ref{closure under Cartier example}.
\end{example}

The following notation allows us to simultaneously work with Cartier operators defined on different rings.

\begin{notation*}
Define $\hat{Q} \in \Z_p[x, y, y^{-1}]$ to be a lift of $P/y \bmod p$ which has the same monomial support as $P/y \bmod p$.
For all $\alpha \geq 1$ and $S \in \ring[x, y, y^{-1}]$, define
\[
	\lambda_{r, 0}^{(\alpha)}(S)
	\colonequal
	\Lambda_{r, 0}\!\paren{S \, (\hat{Q} \bmod p^\alpha)^{p^\alpha - p^{\alpha - 1}}}.
\]
\end{notation*}

For the remainder of this section, for $\alpha\geq 1$, we use the convention of writing $\rep_{p/\hat{Q}}(\lambda_{r, 0}^{(\alpha)}(S^{(\alpha)}))$ as shorthand for $\rep_{p/(\hat{Q} \bmod p^\alpha)}(\lambda_{r, 0}^{(\alpha)}(S^{(\alpha)}))$.

\begin{theorem}\label{digit compatibility}
Let $\beta \in \{1, 2, \dots, \alpha\}$.
Let $S^{(\beta)} \in \mathcal M_{p^\beta}$ and $S^{(\alpha)} \in \mathcal M_{p^\alpha}$ such that the first $\beta$ digits of $\rep_{p/\hat{Q}}(S^{(\alpha)})$ agree with
$\rep_{p/\hat{Q}}(S^{(\beta)})$.
Then the first $\beta$ digits of $\rep_{p/\hat{Q}}(\lambda_{r, 0}^{(\alpha)}(S^{(\alpha)}))$ agree with
$\rep_{p/\hat{Q}}(\lambda_{r, 0}^{(\beta)}(S^{(\beta)}))$.
\end{theorem}

\begin{proof}
Let $\rep_{p/\hat{Q}}(S^{(\alpha)}) = (T_{\alpha - 1}, \dots, T_1, T_0)$;
by assumption,
\[
	S^{(\beta)}
	= \left(
	\left(T_0
	+ T_1 \, \tfrac{p}{\hat{Q}}
	+ \dots
	+ T_{\beta - 1} \, (\tfrac{p}{\hat{Q}})^{\beta - 1}\right)
	\hat{Q}^{p^{\beta - 1} - 1}
	\bmod p^\beta\right).
\]
By Theorem~\ref{state base-p/Q representation}, $\lambda_{r, 0}^{(\alpha)}(S^{(\alpha)})$ has a base-$\frac{p}{\hat{Q}}$ representation.
Let $T_m'$ be the $m$th base-$\frac{p}{\hat{Q}}$ digit of $\lambda_{r, 0}^{(\alpha)}(S^{(\alpha)})$.
We show that $T_0', T_1', \dots, T_{\beta - 1}'$ are defined from $T_0, T_1, \dots, T_{\beta - 1}$ and $(\hat{Q} \bmod p^\beta)$ in the same way as the digits of $\lambda_{r, 0}^{(\beta)}(S^{(\beta)})$.
This implies the statement of the theorem.

We use the proof of Theorem~\ref{closure under Cartier}.
Define $U_{k, j}$ as in Equation~\eqref{digit series}.
From the proof of Theorem~\ref{closure under Cartier}, the $m$th unnormalized digit of $\lambda_{r, 0}^{(\alpha)}(S^{(\alpha)})$ is $\sum_{k = 0}^m U_{k, m - k}$.
If we can show that the $m$th unnormalized digits of $\lambda_{r, 0}^{(\alpha)}(S^{(\alpha)})$ and $\lambda_{r, 0}^{(\beta)}(S^{(\beta)})$ agree, then their first $\beta$ normalized digits also agree.

For $m = 0$, we have $T_0' = U_{0, 0} \equiv \Lambda_{r, 0}\!\left(T_k \hat{Q}^{-1}\right) \hat{Q} \mod p$.
Therefore the $0$th digits satisfy $\dig_0(\lambda_{r, 0}^{(\alpha)}(S^{(\alpha)})) = \dig_0(\lambda_{r, 0}^{(\beta)}(S^{(\beta)}))$.

Inductively, let $m \in \{1, 2, \dots, \beta - 2\}$ and suppose that each $U_{k, j}$ for $k + j < m$ depends only on $T_0, T_1, \dots, T_{m - 1}$ and $(\hat{Q} \bmod p^m)$.
This implies that $\dig_k(\lambda_{r, 0}^{(\alpha)}(S^{(\alpha)})) = \dig_k(\lambda_{r, 0}^{(\beta)}(S^{(\beta)}))$ for $k \in \{0, 1, \dots, m - 1\}$.
We will show that each $U_{k, j}$ for $k + j = m$ depends only $T_0, T_1, \dots, T_m$ and $(\hat{Q} \bmod p^{m + 1})$;
this will imply $\dig_m(\lambda_{r, 0}^{(\alpha)}(S^{(\alpha)})) = \dig_m(\lambda_{r, 0}^{(\beta)}(S^{(\beta)}))$.
Assume $k + j = m$.
By the construction of $U_{k, j}$ in the proof of Theorem~\ref{closure under Cartier}, we have
\[
	\Lambda_{r, 0}\!\left(T_k \hat{Q}^{-k - 1}\right) \hat{Q}^{k + j + 1}
	\equiv U_{k,0} \hat{Q}^j + \dots + p^{j - 1} U_{k, j - 1} \hat{Q} + p^j U_{k, j} \mod p^{j + 1}.
\]
Therefore
\[
	\frac{\Lambda_{r, 0}\!\left(T_k \hat{Q}^{-k - 1}\right) \hat{Q}^{k + j + 1} - \left(U_{k,0} \hat{Q}^j + \dots + p^{j - 1} U_{k, j - 1} \hat{Q}\right)}{p^j}
	\equiv U_{k, j} \mod p.
\]
The left side depends only on $T_0, T_1, \dots, T_m$ and $(\hat{Q} \bmod p^{m + 1})$.
\end{proof}

\begin{corollary}
Let $p$ be a prime.
The inverse limit of the automata $\mathcal M_{p^\alpha}$ exists, and each state in the inverse limit is identified with an infinite sequence of base-$\frac{p}{\hat{Q}}$ digits.
\end{corollary}

\begin{proof}
Let $\alpha\geq 1$, and let $\beta \in \{1, 2, \dots, \alpha\}$.
Denote the initial states in $\mathcal M_{p^\beta}$ and $\mathcal M_{p^\alpha}$ by $S_0^{(\beta)}$ and $S_0^{(\alpha)}$.
We will show that for each $k \in \{0, 1, \dots, \beta - 1\}$, we have $\dig_k(S_0^{(\alpha)}) = \dig_k(S_0^{(\beta)})$.
Then Theorem~\ref{digit compatibility} tells us that for $n \geq 0$ and $r_1, r_2, \dots, r_n \in \{0, 1, \dots, p - 1\}$, the first $\beta$ digits of $\rep_{p/\hat{Q}}((\lambda_{r_n, 0}^{(\alpha)} \circ \dots \circ \lambda_{r_2, 0}^{(\alpha)} \circ \lambda_{r_1, 0}^{(\alpha)})(S_0^{(\alpha)}))$ agree with
$\rep_{p/\hat{Q}}((\lambda_{r_n, 0}^{(\beta)} \circ \dots \circ \lambda_{r_2, 0}^{(\beta)} \circ \lambda_{r_1, 0}^{(\beta)})(S_0^{(\beta)}))$.
This allows us to define $\pi_{\alpha,\beta}\colon \mathcal M_{p^\alpha} \rightarrow \mathcal M_{p^\beta}$ as follows.
Identifying a state $S$ with $\rep_{p/\hat{Q}}(S)$, define $\pi_{\alpha,\beta}((T_{\alpha - 1}, \dots, T_1, T_0)) = (T_{\beta - 1}, \dots, T_1, T_0)$.
It is clear that if there is a state transition from $S$ to $S'$ in $\mathcal M_{p^\alpha}$ then there is a transition labeled $r$ from $\pi_{\alpha,\beta}(S)$ to $\pi_{\alpha,\beta}(S')$ in $\mathcal M_{p^\beta}$.
Finally, if $\gamma \leq \beta \leq \alpha$, we have $\pi_{\alpha,\gamma} = \pi_{\beta,\gamma} \circ \pi_{\alpha,\beta}$.
This means that we have a well defined inverse limit $\varprojlim \mathcal M_{p^\alpha}$ whose states have the claimed structure.

It remains to show that for each $k \in \{0, 1, \dots, \beta - 1\}$, we have $\dig_k(S_0^{(\alpha)}) = \dig_k(S_0^{(\beta)})$.
From the proof of Theorem~\ref{state base-p/Q representation}, the $0$th digit of $S_0^{(\alpha)}$ is $T_0 = \paren{y \frac{\partial P}{\partial y} \bmod p}$, and this is independent of $\alpha$.
Recursively, for $k\geq 1$, the $k$th digit of $S_0^{(\alpha)}$ is
\[
	\frac{y \tfrac{\partial P}{\partial y} \cdot \frac{\hat{Q} \bmod p^\alpha}{P/y} - T_0 - T_1 \, (\tfrac{p}{\hat{Q} \bmod p^\alpha}) - \dots - T_{k - 1} \, (\tfrac{p}{\hat{Q} \bmod p^\alpha})^{k - 1}}{p^k} (\hat{Q} \bmod p^\alpha)^k \bmod p,
\]
similarly for the $k$th digit of $S_0^{(\beta)}$.
By definition of $\hat{Q}$, the numerators of these expressions are congruent to each other modulo~$p^{k+1}$.
Dividing by $p^k$, they are congruent modulo~$p$.
Therefore these expressions are congruent modulo~$p$, so $\dig_k(S_0^{(\alpha)}) = \dig_k(S_0^{(\beta)})$.
\end{proof}

\section{First bounds on the size of the automaton}\label{section: first bounds}

By Theorem~\ref{state base-p/Q representation}, every state in $\mathcal M_{p^\alpha}$ has a base-$\frac{p}{Q}$ representation.
In this section, we bound the digits of these representations, to give preliminary bounds on the size of $\mathcal M_{p^\alpha}$ in Corollaries~\ref{kernel initial upper bound - ring} and \ref{kernel preliminary upper bound - ring}.
To do this, we will show that the base-$\frac{p}{Q}$ representations of most states live in the spaces $\mathcal W$ or $\mathcal V$ defined below in Equations~\eqref{W definition - ring} and \eqref{V definition - ring}.

\begin{notation*}
Recall $D = \{0, 1, \dots, p - 1\}$.
For $k \in \{0, 1, \dots, \alpha - 1\}$, let
\[
	W_k \colonequal
	\left\{
		\sum_{i = 0}^{(k + 1) h - 1} \sum_{j = \max(-k, -i)}^{(k + 1) (d - 1)} c_{i, j} x^i y^j
		: \text{$c_{i, j} \in D$ for each $i, j$}
	\right\}
\]
and
\[
	V_k \colonequal \left\{
		\sum_{i = 0}^{(k + 1) h} \sum_{j = \max(-k, -i)}^{(k + 1)(d - 1)} c_{i, j} x^i y^j
		: \text{$c_{i, j} \in D$ for each $i, j$}
	\right\}.
\]
Define
\begin{align}\label{W definition - ring}
	\mathcal W &\colonequal \{
		(T_{\alpha - 1}, \dots, T_1, T_0)
		: \text{$T_k \in W_k$ for each $k \in \{0, 1, \dots, \alpha - 1\}$}
	\} \\
	\label{V definition - ring}
	\mathcal V &\colonequal \{
		(T_{\alpha - 1}, \dots, T_1, T_0)
		: \text{$T_k \in V_k$ for each $k \in \{0, 1, \dots, \alpha - 1\}$}
	\}.
\end{align}
\end{notation*}

The proof of Proposition~\ref{carries} implies that $\mathcal W$ is a $D$-vector space, where the vector space addition operation is addition of the corresponding Laurent polynomials.
The set of all tuples of the form $(0, \dots, 0, x^i y^j, 0, \dots, 0)$ with the appropriate bounds on $i$ and $j$ forms a basis of $\mathcal W$, in that every element of $\mathcal W$ is a unique $D$-linear combination of this basis.
Figure~\ref{monomials} depicts the geometry of the monomials in $\mathcal W$ for a particular choice of values for $p$, $\alpha$, $h$, and $d$.

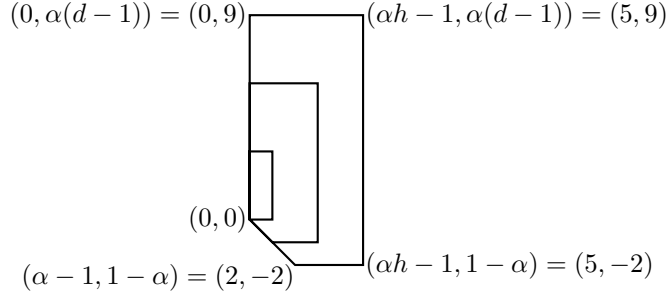
\begin{figure}
\begin{tikzpicture}
\draw[thick]
(0, 0) --
(.3, 0) --
(.3, .9) --
(0, .9) --
cycle
(0, 0) --
(.3, -.3) --
(.9, -.3) --
(.9, 1.8) --
(0, 1.8) --
cycle
node[xshift = -1.2em] {$(0, 0)$}(0, 0) --
node[xshift = -4.3em, yshift = -1.3em] {$(\alpha - 1, 1 - \alpha) = (2, -2)$}(.6, -.6) --
node[xshift = 6.8em] {$(\alpha h - 1, 1 - \alpha) = (5, -2)$}(1.5, -.6) --
node[xshift = 5.8em, yshift = 4.7em] {$(\alpha h - 1, \alpha (d - 1)) = (5, 9)$}(1.5, 2.7) --
node[xshift = -6.7em] {$(0, \alpha (d - 1)) = (0, 9)$}(0, 2.7) --
cycle;
\end{tikzpicture}
\caption{Nested polygons for $k\in \{0,1,2\}$ containing pairs of exponents $(i, j)$ corresponding to monomials $x^i y^j$ in the basis of $W_k$, with $(p, \alpha, h, d) = (3,3,2,4)$.}
\label{monomials}
\end{figure}

We compute the size $N$ of the basis of $\mathcal W$ by partitioning the Newton polygon of each $T_k$ into a rectangle and a trapezoid:
\begin{align}\label{W size}
	N
	&\colonequal \sum_{k = 0}^{\alpha - 1} \paren{
		(k + 1) h \cdot \paren{(k + 1) (d - 1) + 1}
		+ \sum_{j = -k}^{-1} \paren{(k + 1) h + j}
	} \nonumber \\
	&= \tfrac{1}{6} \alpha (\alpha + 1) ((2 h d - 1) \alpha + h d + 1).
\end{align}
Similarly, the dimension of $\mathcal V$ is
\begin{multline}\label{V size}
	\sum_{k = 0}^{\alpha - 1} \paren{
		((k + 1) h + 1) \cdot \paren{(k + 1) (d - 1) + 1}
		+ \sum_{j = -k}^{-1} \paren{(k + 1) h + 1 + j}
	} \\
	= \tfrac{1}{6} \alpha (\alpha + 1) ((2 h d - 1) \alpha + (h + 3) d + 1).
\end{multline}

We work with states whose base-$\frac{p}{Q}$ representation belongs to $\mathcal V$.
For these states, one checks the following elementary result.

\begin{lemma}\label{degree bounds}
If $S \in \val_{p/Q}(\mathcal V)$, then $\deg_x S \leq p^{\alpha - 1} h$, $\deg_y S \leq p^{\alpha - 1} (d - 1)$, and $1 - p^{\alpha - 1} \leq \mindeg_y S$.
\end{lemma}

We will see in Proposition~\ref{transient - ring} that most states belong to $\mathcal W$.
Furthermore, in Corollary~\ref{module - ring} we show that $\val_{p/Q}(\mathcal W)$ and $\val_{p/Q}(\mathcal V)$ are invariant under $\lambda_{r, 0}$ for each $r$.
To do this we use the following proposition.

\begin{proposition}\label{digit bounds under Cartier}
Let $r \in \{0, 1, \dots, p - 1\}$, and suppose that $S$ has a base-$\frac{p}{Q}$ representation $\rep_{p/Q}(S)= (T_{\alpha - 1}, \dots, T_1, T_0)$.
Then the $m$th base-$\frac{p}{Q}$ digit $T_m'$ of $\lambda_{r, 0}(S)$ satisfies
\begin{align*}
	\deg_x T_m' &\leq \max_{0\leq k \leq m} \left(\floor{\tfrac{\deg_x T_k - (k + 1) h - r}{p}}\right) + (m + 1) h \\
	\deg_y T_m' &\leq \max_{0\leq k \leq m} \left(\floor{\tfrac{\deg_y T_k - (k + 1) (d - 1)}{p}}\right) + (m + 1) (d-1)\\
	\mindeg_y T_m' &\geq \min_{0\leq k \leq m} \left(\ceil{\tfrac{\mindeg_y T_k + (k + 1)}{p}}\right) - (m+1).
\end{align*}
\end{proposition}

\begin{proof}
We prove the first inequality; the others follow similarly, the only difference being that $r$ is replaced by 0.
Define Laurent polynomials $U_{k, m}$ as in the proof of Theorem~\ref{closure under Cartier}, so that Equation~\eqref{digit series} is satisfied modulo~$p^{j + 1}$, namely
\begin{equation}\label{rephrased digit series}
	\Lambda_{r, 0}\!\left(T_k Q^{-k - 1}\right) Q^{k + j + 1} \equiv \sum_{m = 0}^j U_{k, m} Q^{j - m} p^m \mod p^{j + 1}.
\end{equation}
We claim that
\begin{equation}\label{degree bound}
	\deg_x U_{k, j} \leq \floor{\tfrac{\deg_x T_k - (k + 1) h - r}{p}} + (k + j + 1) h
\end{equation}
for all $j \in \{0, 1, \dots, \alpha - k - 1\}$.
By Lemma~\ref{polynomial modulo prime power} (or indeed just by Proposition~\ref{Cartier - ring}), the Laurent polynomial $U_{k, 0}$ satisfies
\[
	\deg_x U_{k, 0} \leq \floor{\tfrac{\deg_x T_k - (k + 1) h - r}{p}} + (k + 1) h.
\]
Inductively, assume the claim is true for $U_{k, 0}, \dots, U_{k, j - 1}$.
For each $m \in \{0, 1, \dots, j - 1\}$, this implies
\[
	\deg_x(U_{k, m} Q^{j - m}) \leq \floor{\tfrac{\deg_x T_k - (k + 1) h - r}{p}} + (k+j+1) h,
\]
which is independent of $m$.
By Lemma~\ref{polynomial modulo prime power} with $\beta = j + 1$, we have
\[
	\deg_x(\Lambda_{r, 0}\!\left(T_k Q^{-k-1}\right) Q^{k+j+1}) \leq \floor{\tfrac{\deg_x T_k - (k + 1) h - r}{p}} + (k + j + 1) h.
\]
The claim follows by combining the previous two sets of inequalities using Equation~\eqref{rephrased digit series}.
Finally, for each $m$, by Lemma~\ref{Minkowski}, Lemma~\ref{polynomial modulo prime power}, and the conditions on $T_k$, we have that each nonzero monomial $c \, x^I y^J$ in $U_{k,m}$ satisfies $J\geq -I$.

It remains to transfer the bounds on the Laurent polynomials $U_{k, j}$ to the digits $T_m' = \sum_{k = 0}^m U_{k, m - k}$.
Since the condition~\eqref{degree bound} holds for each $k \in \{0, 1, \dots, \alpha - 1\}$, we have
\[
	\deg_x T_m' \leq \max_{0\leq k \leq m} \left(\floor{\tfrac{\deg_x T_k - (k + 1) h - r}{p}} + (m + 1) h\right).
\]
Also, for each $m$, by Lemma~\ref{Minkowski}, each nonzero monomial $c \, x^I y^J$ in $T_m'$ satisfies $J\geq -I$.

The carried digit from $T_m'$, when multiplied by $Q$, has $x$-degree at most $\deg_x T_m' + h$;
this is already the $x$-degree bound on $T_{m+1}'$.
Therefore the digits of the state $\lambda_{r, 0}(S)$ satisfy the desired degree bounds.
\end{proof}

\begin{corollary}\label{module - ring}
For each $r \in \{0, 1, \dots, p - 1\}$, we have $\lambda_{r, 0}(\val_{p/Q}(\mathcal W)) \subseteq \val_{p/Q}(\mathcal W)$ and $\lambda_{r, 0}(\val_{p/Q}(\mathcal V)) \subseteq \val_{p/Q}(\mathcal V)$.
Furthermore if $r\neq 0$ then $\lambda_{r, 0}(\val_{p/Q}(\mathcal V)) \subseteq \val_{p/Q}(\mathcal W)$.
\end{corollary}

\begin{proof}First let $S \in \ring[x, y, y^{-1}]$ such that $(T_{\alpha - 1}, \dots, T_1, T_0) \colonequal \rep_{p/Q}(S) \in \mathcal W$.
By assumption, we have
\begin{align*}
	\deg_x T_k &\leq (k + 1) h - 1 \\
	\deg_y T_k &\leq (k + 1) (d-1) \\
	\mindeg_y T_k &\geq - k.
\end{align*}
Proposition~\ref{digit bounds under Cartier} now implies that $\lambda_{r, 0}(S) \in \mathcal W$.
The proof that $\val_{p/Q}(\mathcal V)$ is closed under $\lambda_{r, 0}$ is similar.

If $r\neq 0$, $\rep_{p/Q}(S) \in \mathcal V$, and $(T_{\alpha - 1}', \dots, T_1', T_0') \colonequal \rep_{p/Q}(\lambda_{r,0}(S))$, then Proposition~\ref{digit bounds under Cartier} tells us that
\begin{align*}
	\deg_x T_m' &\leq \max_{0\leq k \leq m} \left(\floor{\tfrac{\deg_x T_k - (k + 1) h - r}{p}}\right) + (m + 1) h \\
	&\leq \max_{0\leq k \leq m} \left(\floor{\tfrac{-r}{p}}\right) + (m + 1) h = (m+1)h -1,
\end{align*}
Therefore $\rep_{p/Q}(\lambda_{r,0}(S)) \in \mathcal W$.
\end{proof}

Proposition~\ref{transient - ring} below tells us that the representations of most states belong to $\mathcal W$.
For this we need the following lemma.
Define $h_k = \deg_x(P \bmod p^{k + 1})$ and $d_k = \deg_y(P \bmod p^{k + 1})$; we have $h_0 = h$ and $d_0=d$.

\begin{lemma}\label{initial state}
The base-$\frac{p}{Q}$ digits $T_k$ of the initial state $S_0$ satisfy
\begin{align*}
	\deg_x T_k &\leq (k + 1) h_k \nonumber \\
	\deg_y T_k &\leq (k + 1) (d_k - 1) + 1 \\
	\mindeg_y T_k &\geq -k, \nonumber
\end{align*}
and every nonzero monomial $c \, x^I y^J$ that appears in $T_k$ satisfies $J\geq -I$.
\end{lemma}

\begin{proof}
Recall that $S_0 = \paren{y \frac{\partial P}{\partial y} \, (P/y)^{p^{\alpha - 1} - 1} \bmod p^\alpha}$.
Define the digits $T_k$ as in the proof of Theorem~\ref{state base-p/Q representation}, namely
\[
	T_k = \paren{\paren{y \tfrac{\partial P}{\partial y} \cdot \tfrac{Q}{P/y} - T_0 - T_1 \, (\tfrac{p}{Q}) - \dots - T_{k - 1} \, (\tfrac{p}{Q})^{k - 1}} \paren{\tfrac{Q}{p}}^k \bmod p}.
\]

First we consider the $x$-degree.
For $T_0 = \paren{y \frac{\partial P}{\partial y} \bmod p}$, we have $\deg_x T_0 \leq h_0$.
For $T_1$, we use Lemma~\ref{polynomial quotient} to obtain $\deg_x\!\left(\frac{Q^2}{P/y} \bmod p^2\right) \leq h_1$, which gives
\[
	\deg_x T_1 \leq \max(h_1 + h_1, h_0 + h_1) = 2 h_1.
\]
Inductively, for $k \geq 1$ we have
\begin{align*}
	\deg_x T_k &\leq \max(
		h_k + k h_k,
		h_0 + k h_k,
		2 h_1 + (k - 1) h_k,
		\dots,
		k h_{k - 1} + h_k
	) \\
	&= (k + 1) h_k.
\end{align*}

We now consider the $y$-degree.
For $T_0$, we have $\deg_y T_0 \leq d_0$.
For $T_1$, we use Lemma~\ref{polynomial quotient} to obtain $\deg_y\!\left(\frac{Q^2}{P/y} \bmod p^2\right) \leq d_1 -1$, which gives
\[
	\deg_y T_1 \leq \max(d_1 + d_1 - 1, d_0 + d_1 - 1) = 2 d_1 - 1.
\]
Inductively, for $k \geq 1$ we have
\begin{align*}
	\deg_y T_k &\leq \max(
		d_k + k (d_k - 1),
		d_0+k(d_k-1),
		\dots,
		k (d_{k-1} - 1) + 1 + d_k-1
	) \\
	&= (k + 1) (d_k - 1) + 1.
\end{align*}

Finally, we consider the $y$-min-degree.
We have $\mindeg_y T_0 \geq 0$.
Inductively, $\mindeg_y T_k \geq -k$ by Lemma~\ref{polynomial quotient}.

Lemma~\ref{Minkowski} gives the final condition.
\end{proof}

\begin{proposition}\label{transient - ring}
Let $S_0$ be the initial state, and let $(T_{\alpha - 1}, \dots, T_1, T_0) \colonequal \rep_{p/Q}(S_0)$.
Let
\[
	u = \floor{\log_p \max\!\paren{\alpha (h_{\alpha - 1} - h), \alpha (d_{\alpha - 1} - d) + 1}} + 1.
\]
Then, for all $r_1, r_2, \dots, r_u \in \{0, 1, \dots, p - 1\}$, we have $\rep_{p/Q}((\lambda_{r_u, 0} \circ \dots \circ \lambda_{r_2, 0} \circ \lambda_{r_1, 0})(S_0)) \in \mathcal V$.
\end{proposition}

\begin{proof}
Let $S$ denote a state, $(T_{\alpha - 1}, \dots, T_1, T_0) = \rep_{p/Q}(S)$, $r\in \{0, \dots, p-1\}$, and $(T'_{\alpha - 1}, \dots, T'_1, T'_0) = \rep_{p/Q}(\lambda_{r, 0}(S))$.
First we consider the $y$-min-degrees of states.
By Proposition~\ref{digit bounds under Cartier}, if $\mindeg_y T_k\geq -k$ for $k\in \{0, 1, \dots, \alpha - 1\}$, then $\mindeg_y T_m'\geq -m$ for $m\in \{0, 1, \dots, \alpha - 1\}$.
From Lemma~\ref{initial state}, the digits of the initial state satisfy these constraints.
Hence all states satisfy these constraints.

Next we consider the $x$-degree and $y$-degree of states.
Assume that there is an $n\geq 0$ such that, for each $k \in \{0, 1, \dots, \alpha - 1\}$, the $k$th digit $T_k$ of $S$ satisfies
\begin{align}\label{degree condition}
	\deg_x T_k &\leq \floor{\frac{(k + 1) h_k - (k + 1) h}{p^n}} + (k + 1) h \\
	\deg_y T_k &\leq \floor{\frac{(k + 1) d_k +1 - (k + 1) d}{p^n}} + (k + 1) (d-1). \nonumber
\end{align}
Let $m\in \{0, 1, \dots, \alpha - 1\}$.
Proposition~\ref{digit bounds under Cartier} gives
\begin{align*}
	\deg_x T_m'
	&\leq \max_{0\leq k \leq m} \left(\floor{\frac{\deg_x T_k - (k + 1) h - r}{p}}\right) + (m + 1) h \\
	&\leq \max_{0\leq k \leq m} \left(\floor{\frac{\floor{\frac{(k + 1) h_k - (k + 1) h}{p^n}} + (k + 1) h - (k + 1) h}{p}}\right) + (m + 1) h \\
	&\leq \max_{0\leq k \leq m} \left(\floor{\frac{(k + 1) h_k - (k + 1) h}{p^{n + 1}}}\right) + (m + 1) h.
\end{align*}
Since $h_k \leq h_m$ for each $k \in \{0, 1, \dots, m\}$, this implies
\[
	\deg_x T_m'
	\leq \floor{\frac{(m + 1) h_m - (m + 1) h}{p^{n + 1}}} + (m + 1) h.
\]
Similarly,
\[
	\deg_y T_m' \leq \floor{\frac{(m + 1) d_m +1 - (m + 1) d}{p^{n + 1}}} + (m + 1) (d-1).
\]
It follows that, if $\floor{\frac{(m + 1) h_m - (m + 1) h}{p^{n + 1}}}=0$ and $\floor{\frac{(m + 1) d_m +1 - (m + 1) d}{p^{n + 1}}}=0$ for each $m$, then $\rep_{p/Q}(\lambda_{r, 0}(S)) \in \mathcal V$.

We have shown that iterating $\lambda_{r, 0}$ reduces the floor expression in \eqref{degree condition}.
The initial state $S_0$ satisfies \eqref{degree condition} with $n=0$ since, by Lemma~\ref{initial state}, the digits $T_k$ of $S_0$ satisfy
\begin{align*}
	\deg_x T_k &\leq (k + 1) h_k \\
	\deg_y T_k &\leq (k + 1) (d_k - 1) + 1.
\end{align*}
It follows from the definition of $u$ and the fact that $h_m \leq h_{\alpha - 1}$ and $d_m \leq d_{\alpha - 1}$ that $\floor{\frac{(m + 1) h_m - (m + 1) h}{p^u}}=0$ and $\floor{\frac{(m + 1) d_m +1 - (m + 1) d}{p^u}}=0$ for each $m\in \{0, 1, \dots, \alpha - 1\}$.
Therefore, for all $r_1, r_2, \dots, r_u \in \{0, 1, \dots, p - 1\}$, we have $\rep_{p/Q}((\lambda_{r_u, 0} \circ \dots \circ \lambda_{r_2, 0} \circ \lambda_{r_1, 0})(S_0)) \in \mathcal V$.
\end{proof}

An immediate corollary of Proposition~\ref{transient - ring} is the following, where $\size{\mathcal V} = p^{\dim \mathcal V}$, since each coefficient in the digit $T_k$ belongs to $D = \{0, 1, \dots, p - 1\}$, and where $\dim \mathcal V$ is given by Equation~\eqref{V size}.

\begin{corollary}\label{kernel initial upper bound - ring}
Let $p$ be a prime, and let $\alpha \geq 1$.
Let $F = \sum_{n \geq 0} a(n) x^n \in \Z_p\doublebracket{x}$ be the Furstenberg series associated with a polynomial $P \in \Z_p[x, y]$ such that $h \colonequal \deg_x(P \bmod p) \geq 1$ and $d \colonequal \deg_y(P \bmod p) \geq 1$.
Let
\begin{equation}\label{u transient}
	u = \floor{\log_p \max\!\paren{\alpha (\deg_x(P \bmod p^\alpha) - h), \alpha (\deg_y(P \bmod p^\alpha) - d) + 1}} + 1.
\end{equation}
Then
\[
	\size{\ker_p((a(n) \bmod p^\alpha)_{n \geq 0})}
	\leq p^{\frac{1}{6} \alpha (\alpha + 1) ((2 h d - 1) \alpha + (h + 3) d + 1)} + \tfrac{p^u -1}{p-1}.
\]
\end{corollary}

In particular, for $\alpha = 1$ we have $\size{\ker_p((a(n) \bmod p)_{n \geq 0})} \leq p^{(h + 1) d} + 1$.
This is because Corollary~\ref{kernel initial upper bound - ring} is also true with the $u$ defined as in Proposition~\ref{transient - ring}, and for $\alpha = 1$ we have $u = 1$ for this value.

By beginning to consider the orbit under $\lambda_{0, 0}$, we obtain the following refinement of Corollary~\ref{kernel initial upper bound - ring}.
For a function $f \colon X \to X$, define the \emph{orbit} of $S \in X$ under $f$ to be the sequence $S, f(S), f^2(S), \dots$, and let $\size{\orb_f(S)}$ be the number of distinct terms in the orbit.

\begin{corollary}\label{kernel preliminary upper bound - ring}
Let $p$ be a prime, and let $\alpha \geq 1$.
Let $F = \sum_{n \geq 0} a(n) x^n \in \Z_p\doublebracket{x}$ be the Furstenberg series associated with a polynomial $P \in \Z_p[x, y]$ such that $h \colonequal \deg_x(P \bmod p) \geq 1$ and $d \colonequal \deg_y(P \bmod p) \geq 1$.
Define $u$ as in Equation~\eqref{u transient}.
Then
\[
	\size{\ker_p((a(n) \bmod p^\alpha)_{n \geq 0})}
	\leq p^{\frac{1}{6} \alpha (\alpha + 1) ((2 h d - 1) \alpha + h d + 1)} + \size{\orb_{\Lambda_0}(F \bmod p^\alpha)} + \tfrac{p^u -1}{p-1}-u.
\]
\end{corollary}

\begin{proof}
By Proposition~\ref{transient - ring}, there are at most $\frac{p^u -1}{p-1}$ states that are not in $\mathcal V$.
By Corollary~\ref{module - ring}, if $r\neq 0$ then $\lambda_{r, 0}(\val_{p/Q}(\mathcal V)) \subseteq \val_{p/Q}(\mathcal W)$, and this is where the first term comes from, since the dimension of $\mathcal W$ is given by Equation~\eqref{W size}.
Applying $\lambda_{0, 0}$ iteratively to $S_0$ produces $\size{\orb_{\lambda_{0, 0}}(S_0)}$ states, and $\size{\orb_{\lambda_{0, 0}}(S_0)} = \size{\orb_{\Lambda_0}(F \bmod p^\alpha)}$ by definition.
The first $u$ elements of $\orb_{\Lambda_0}(F \bmod p^\alpha)$ are already counted in the term $\frac{p^u -1}{p-1}$.
The result follows.
\end{proof}

The structure of states given by Theorem~\ref{state base-p/Q representation} also provides a faster algorithm for computing the automaton by representing states as $\alpha$-tuples of base-$\frac{p}{Q}$ digits rather than as Laurent polynomials.
To make it efficient, we use tricks analogous to those we used for polynomial states~\cite{Rowland--Yassawi}.
The algorithm is as follows.
First, for each $k \in \{0, 1, \dots, \alpha - 1\}$ and each $j \in \{0, 1, \dots, \alpha - k - 1\}$, compute the Laurent polynomial
\[
	Z_{j, k} \colonequal \left(\sum_{m = k + 1}^{k + j + 1} \binom{k + j + 1}{m} Q^{p m - k - 1} \Delta^{k + j + 1 - m} \bmod p^{j + 1}\right)
\]
from Equation~\eqref{auxiliary polynomial} in the proof of Lemma~\ref{polynomial modulo prime power} (where $\beta = j + 1$), since these Laurent polynomials are used repeatedly and do not depend on the state $S$ whose images $\lambda_{r, 0}(S)$ we are computing at a given step.
Then bin the monomials in each $Z_{j, k}$ according to their exponents modulo~$p$.
This allows us to compute, for each digit $T_k$ that arises, the $p$ images $\Lambda_{r, 0}(T_k Z_{j, k})$ for $r \in \{0, 1, \dots, p - 1\}$ in one pass and without discarding any monomials.

\section{Structure of the linear transformation $\lambda_{0, 0}$}\label{section: structure - ring}

By Corollary~\ref{kernel preliminary upper bound - ring}, it remains to bound $\size{\orb_{\Lambda_0}(F)}$.
In this section, we take the first step toward this goal by identifying univariate operators $\lambda_0$ that emulate $\lambda_{0, 0}$ on three subspaces.
The main result is Corollary~\ref{univariate emulation - ring - tuple}.

\begin{notation*}
Define the following coefficient-extraction maps.
For a Laurent polynomial $S = \sum_{i, j} c_{i,j} x^i y^j$, define $\pi_{x, i}(S) = \sum_j c_{i,j} y^j$ and $\pi_{y, j}(S) = \sum_i c_{i,j} x^i$.
The maps $\pi_{x, i}$ and $\pi_{y, j}$ allow us to focus on univariate (Laurent) polynomials by defining the following operator.
Let $R \in \ring[z, z^{-1}]$.
Define $\lambda_0 \colon \ring[z,z^{-1}] \to \ring[z,z^{-1}]$ by
\[
	\lambda_0(S) = \Lambda_0\!\paren{S R^{p^\alpha - p^{\alpha - 1}}}.
\]
\end{notation*}

Let $\tilde P=yQ$.
That is, $\tilde P$ is a lift of $P \bmod p$ which has the same monomial support as $P \bmod p$.
Write
\[
	\tilde P(x, y)= \sum_{i \geq 0} x^i A_i(y) = \sum_{j \geq 0} B_j(x) y^j.
\]
The univariate Laurent polynomials that will be used to define the various operators $\lambda_0$ are $R = \pi_{x, 0}(Q) = A_0/y$, $R = \pi_{x, h}(Q) = A_h/y$, and $R = \pi_{y, d - 1}(Q) = B_d$.

Proposition~\ref{univariate emulation - ring} below is analogous to \cite[Proposition~13]{Rowland--Stipulanti--Yassawi}.

The following result is essentially a commutation relation.
It shows that the left, right, and top borders of $\lambda_{0, 0}(S)$ depend only on the respective borders of $S$, and therefore the behavior of $\lambda_{0, 0}$ on these components is emulated by the respective operators $\lambda_0$.
One checks that the conditions on $S$ are satisfied by every state whose base-$\frac{p}{Q}$ representation belongs to $\mathcal V$, defined in Equation~\eqref{V definition - ring}.

\begin{proposition}\label{univariate emulation - ring}
We have the following.
\begin{enumerate}
\item
Let $R = \pi_{x, 0}(Q)$.
For all $S \in \ring[x, y, y^{-1}]$,
\[
	\pi_{x, 0}(\lambda_{0, 0}(S))
	= \lambda_0(\pi_{x, 0}(S)).
\]
\item\label{right border - ring}
Let $R = \pi_{x, h}(Q)$.
For all $S \in \ring[x, y, y^{-1}]$ with height at most $p^{\alpha - 1} h$,
\[
	\pi_{x, p^{\alpha - 1} h}(\lambda_{0, 0}(S))
	= \lambda_0(\pi_{x, p^{\alpha - 1} h}(S)).
\]
\item
Let $R = \pi_{y, d - 1}(Q)$.
For all $S \in \ring[x, y, y^{-1}]$ with degree at most $p^{\alpha - 1} (d-1)$,
\[
	\pi_{y, p^{\alpha - 1} (d - 1)}(\lambda_{0, 0}(S))
	= \lambda_0(\pi_{y, p^{\alpha - 1} (d - 1)}(S)).
\]
\end{enumerate}
\end{proposition}

\begin{proof}
We prove the second statement; the proofs of the others are analogous.
Let $\pi\subr = \pi_{x, p^{\alpha - 1} h}$, and let $\deg_x S \leq p^{\alpha - 1} h$.
Since $\pi\subr$ projects onto polynomials in $y$, we are interested in monomials with $x$-degree $p^{\alpha - 1} h$ in $\lambda_{0, 0}(S) = \Lambda_{0, 0}\!\paren{S Q^{p^\alpha - p^{\alpha - 1}}}$.
These come from monomials with $x$-degree $p^\alpha h$ in $S Q^{p^\alpha - p^{\alpha - 1}}$.
Since $\deg_x S \leq p^{\alpha - 1} h$ and $\deg_x Q = h$, each monomial $c \, x^{p^\alpha h} y^J$ in $S Q^{p^\alpha - p^{\alpha - 1}}$ arises only from the product of a monomial in $x^{p^{\alpha - 1}h} \pi\subr(S)$ together with a product of $p^\alpha - p^{\alpha - 1}$ monomials in $x^h \pi_{x, h}(Q)$, namely, monomials in $x^h A_h/y$.
Therefore
\begin{align*}
	\pi\subr(\lambda_{0, 0}(S))
	&= \pi\subr(\lambda_{0, 0}(x^{p^{\alpha - 1} h} \pi\subr(S))) \\
	&= \pi\subr\!\left(\Lambda_{0, 0}\!\paren{x^{p^{\alpha - 1} h} \pi\subr(S) \cdot (x^h A_h/y)^{p^\alpha - p^{\alpha - 1}}}\right) \\
	&= \pi\subr\!\left(x^{p^{\alpha - 1} h} \Lambda_{0, 0}\!\paren{\pi\subr(S) (A_h/y)^{p^\alpha - p^{\alpha - 1}}}\right) \\
	&= \Lambda_0\!\paren{\pi\subr(S) (A_h/y)^{p^\alpha - p^{\alpha - 1}}} \\
	&= \lambda_0(\pi\subr(S)),
\end{align*}
where in the third equality we use Proposition~\ref{Cartier - ring} to rewrite $\Lambda_{0, 0}(G x^{p^\alpha h}) = x^{p^{\alpha - 1} h} \Lambda_{0, 0}(G)$.
\end{proof}

We introduce the following projection maps on base-$\frac{p}{Q}$ representations.

\begin{notation*}
For $z \in \{x, y\}$ and $i\geq 0$, define
\[
	\pr_{z, i}((T_{\alpha - 1}, \dots, T_0)) = (\pi_{z, \alpha i}(T_{\alpha - 1}), \dots, \pi_{z, i}(T_0)).
\]
\end{notation*}

The next result tells us that the operation of converting to digit representations commutes with projecting onto one of the borders.
We need the following notation.

\begin{notation*}
Given a univariate Laurent polynomial $R \in \ring[z, z^{-1}]$, we define $\rep_{p/R}$ analogously to $\rep_{p/Q}$ in Equation~\eqref{rep definition}.
Namely, $\rep_{p/R}\!\left(\sum_{k = 0}^{\alpha - 1} T_k \, (\frac{p}{R})^k R^{p^{\alpha - 1} - 1}\right) \colonequal (T_{\alpha - 1}, \dots, T_1, T_0)$.
\end{notation*}

\begin{theorem}\label{state base-p/Q representation - univariate}
Let $S$ be a state in $\val_{p/Q}(\mathcal V)$.
\begin{enumerate}
\item
Let $R = \pi_{x, 0}(Q)$.
Then $\pr_{x, 0}(\rep_{p/Q}(S)) = \rep_{p/R}(\pi_{x, 0}(S))$.
\item
Let $R = \pi_{x, h}(Q)$.
Then $\pr_{x, h}(\rep_{p/Q}(S)) = \rep_{p/R}(\pi_{x, p^{\alpha - 1} h}(S))$.
\item
Let $R = \pi_{y, d - 1}(Q)$.
Then $\pr_{y, d-1}(\rep_{p/Q}(S)) = \rep_{p/R}(\pi_{y, p^{\alpha - 1} (d - 1)}(S))$.
\end{enumerate}
\end{theorem}

\begin{proof}
Let $(T_{\alpha - 1}, \dots, T_1, T_0) = \rep_{p/Q}(S)$, so that
\[
	S = \left(
		\sum_{k = 0}^{\alpha - 1} T_k \, (\tfrac{p}{Q})^k
	\right)
	Q^{p^{\alpha - 1} - 1}.
\]
We prove the second statement; the others are analogous.
We have
\[
	\pi_{x, p^{\alpha - 1} h}(S)
	= \sum_{k = 0}^{\alpha - 1} \pi_{x, p^{\alpha - 1} h}\!\left(T_k p^k Q^{p^{\alpha - 1} - 1 - k}\right).
\]
Since $\rep_{p/Q}(S) \in \mathcal V$, the digit $T_k$ has $x$-degree at most $(k + 1) h$.
The only way to get a monomial in $T_k Q^{p^{\alpha - 1} - 1 - k}$ with $x$-degree $p^{\alpha - 1} h$ is to multiply a monomial in $T_k$ with $x$-degree $(k + 1) h$ by a monomial in $Q^{p^{\alpha - 1} - 1 - k}$ with $x$-degree $(p^{\alpha - 1} - 1 - k) h$.
Therefore
\begin{align*}
	\pi_{x, p^{\alpha - 1} h}(S)
	&= \sum_{k = 0}^{\alpha - 1} \pi_{x, (k + 1) h}(T_k) \, p^k \pi_{x, (p^{\alpha - 1} - 1 - k) h}\!\left(Q^{p^{\alpha - 1} - 1 - k}\right) \\
	&= \sum_{k = 0}^{\alpha - 1} \pi_{x, (k + 1) h}(T_k) \, p^k (\pi_{x, h}(Q))^{p^{\alpha - 1} - 1 - k} \\
	&= \sum_{k = 0}^{\alpha - 1} \pi_{x, (k + 1) h}(T_k) \, p^k R^{p^{\alpha - 1} - 1 - k} \\
	&= \left(
		\sum_{k = 0}^{\alpha - 1} \pi_{x, (k + 1) h}(T_k) \, (\tfrac{p}{R})^k
	\right)
	R^{p^{\alpha - 1} - 1}.
\end{align*}
This implies $\rep_{p/R}(\pi_{x, p^{\alpha - 1} h}(S)) = \pr_{x, h}(\rep_{p/Q}(S))$, as claimed.
\end{proof}

Proposition~\ref{univariate emulation - ring} and Theorem~\ref{state base-p/Q representation - univariate}, which are both commutation statements involving projection, culminate in the following.

\begin{corollary}\label{univariate emulation - ring - tuple}
Let $S$ be a state in $\val_{p/Q}(\mathcal V)$.
\begin{enumerate}
\item
Let $R = \pi_{x, 0}(Q)$.
Then $\pr_{x, 0}(\rep_{p/Q}(\lambda_{0, 0}(S))) = \rep_{p/R}(\lambda_0(\pi_{x, 0}(S)))$.
\item
Let $R = \pi_{x, h}(Q)$.
Then $\pr_{x, h}(\rep_{p/Q}(\lambda_{0, 0}(S))) = \rep_{p/R}(\lambda_0(\pi_{x, p^{\alpha - 1} h}(S)))$.
\item
Let $R = \pi_{y, d - 1}(Q)$.
Then $\pr_{y, d-1}(\rep_{p/Q}(\lambda_{0, 0}(S))) = \rep_{p/R}(\lambda_0(\pi_{y, p^{\alpha - 1} (d - 1)}(S)))$.
\end{enumerate}
\end{corollary}

\begin{proof}
We prove the second statement; the other proofs are similar.
Since $S\in\val_{p/Q}(\mathcal V)$, then by Corollary~\ref{module - ring}, $\lambda_{0, 0}(S)\in \val_{p/Q}(\mathcal V)$.
Therefore we can apply Theorem~\ref{state base-p/Q representation - univariate} to give $\pr_{x, h}(\rep_{p/Q}(\lambda_{0, 0}(S))) = \rep_{p/R}(\pi_{x, p^{\alpha - 1} h}(\lambda_{0, 0}(S)))$.
Again because $S\in\val_{p/Q}(\mathcal V)$, it satisfies the conditions of Proposition~\ref{univariate emulation - ring}, so that
$
	\pi_{x, p^{\alpha - 1} h}(\lambda_{0, 0}(S))
	= \lambda_0(\pi_{x, p^{\alpha - 1} h}(S))
$.
Putting these two equations together, we get the second statement.
\end{proof}

\section{Period lengths of series expansions modulo $p$ and modulo $p^\alpha$}\label{section: period lengths}

In this section, we state theorems of Engstrom~\cite{Engstrom} that bound period lengths of coefficient sequences of rational series modulo~$p$ and modulo~$p^\alpha$.
We then apply these theorems to bound the period length of the coefficient sequence of $\frac{1}{R^{p^{\alpha - 1}}}$ in Corollary~\ref{period length comparison}.

The following is a strengthening of Engstrom~\cite[Theorems~2 and 3]{Engstrom}, who bounds the period length of the coefficient sequence of a rational series $\frac{1}{R}$.
In Theorem~\ref{constant-recursive modulo prime} we reduce his bound by a factor of $p$ when $e$ is a power of $p$.

\begin{theorem}\label{constant-recursive modulo prime}
Let $R \in \F_p[z]$ be a polynomial with $R(0) \neq 0$ and $\deg R \geq 1$.
Factor $R = c R_1^{e_1} \cdots R_k^{e_k}$ into monic irreducibles.
Let $e = \max_{1 \leq i \leq k} e_i$ and $L = \lcm_{1 \leq i \leq k} (p^{\deg R_i} - 1)$.
Then the coefficient sequence of $\frac{1}{R}$ is periodic with period length dividing $p^{\lceil \log_p e \rceil} L$.
\end{theorem}

\begin{proof}
Lift $R$ and $R_1, \dots, R_k$ to $\Z_p[z]$, and write $\frac{1}{R} = \sum_{n \geq 0} a(n) z^n\in \Z_p\doublebracket{z}$.
We follow Engstrom~\cite[Section~3.1]{Engstrom}.
For $i \in \{1, 2, \dots, k\}$ and $j \in \{1, 2, \dots, \deg R_i\}$, let $\rho_{i,j}$ be the $j$th root of $R_i(z) = 0$.
Then, for all $n \geq 0$, we have
\[
	a(n) = \sum_{i, j} \left(c_{i, j, 0} \binom{n}{0} + c_{i, j, 1} \binom{n}{1} + \dots + c_{i, j, e_i - 1} \binom{n}{e_i - 1}\right) \rho_{i, j}^n
\]
for some elements $c_{i, j, j'}$ in the splitting field of $R$ over $\Q_p$.
Engstrom~\cite[Lemma~4]{Engstrom} implies that each $c_{i, j, j'}$ is a $p$-adic algebraic integer.
The appropriate generalization of Fermat's little theorem implies that $\rho_{i, j}^{p^{\deg R_i} - 1} \equiv 1 \mod p$.
Therefore the period length of $(\rho_{i, j}^n \bmod p)_{n \geq 0}$ divides $L$.
It remains to bound the period lengths modulo~$p$ of $\binom{n}{0}, \dots, \binom{n}{e_i - 1}$.
A result of Z\k{a}bek~\cite[Th\'eor\`eme~3]{Zabek} implies that the period length of $(\binom{n}{j'} \bmod p)_{n \geq 0}$ is $p^{\lfloor \log_p j' \rfloor + 1} \leq p^{\lfloor \log_p(e_i - 1) \rfloor + 1} \leq p^{\lfloor \log_p(e - 1) \rfloor + 1} = p^{\lceil \log_p e \rceil}$.
The result follows.
\end{proof}

\begin{theorem}[{Engstrom~\cite[Theorem~8]{Engstrom}}]\label{constant-recursive modulo prime power}
Let $T \in \ring[z]$ be a polynomial with $t \colonequal \deg T \geq 1$ such that the coefficients of $z^0$ and $z^t$ in $T$ are nonzero modulo~$p$.
Write $\frac{1}{T} = \sum_{n \geq 0} a(n) z^n$, and let $m$ be the period length of $(a(n) \bmod p)_{n \geq 0}$.
Then $a(n)_{n \geq 0}$ is periodic with period length dividing $p^{\alpha - 1} m$.
\end{theorem}

\begin{example}\label{Fibonacci mod 4}
Let $p = 2$, $\alpha = 2$, and $R = -z^2 - z + 1 \in \mathcal{R}_4[z]$.
Let $T \colonequal R^{p^{\alpha - 1}} = (-z^2 - z + 1)^2$.
Since $R \bmod 2$ is irreducible, the quantities in Theorem~\ref{constant-recursive modulo prime} are $e = 2$ and $L = 2^2 - 1 = 3$.
Theorem~\ref{constant-recursive modulo prime} implies that the coefficient sequence of $\frac{1}{T} \bmod p$ is periodic with period length $m$ dividing $p^{\lceil \log_p e \rceil} L = 2^1 \cdot 3 = 6$, and in fact the period length is $m = 6$.
Theorem~\ref{constant-recursive modulo prime power} then implies that the coefficient sequence of $\frac{1}{T}$ is periodic with period length dividing $p^{\alpha - 1} m = 12$, and in fact the period length is $12$.
\end{example}

We will also use the following lemma, which is interesting in its own right.
It establishes that the period of the coefficient sequence of the series $\frac{1}{T}$ ends with zeros.
For example, consider $T = -z^2 - z + 1 \in \Q[z]$, whose coefficient sequence of $\frac{1}{T}$ is the shifted Fibonacci sequence $1, 1, 2, 3, 5, 8, \dots$;
the following lemma implies that its period modulo~$p$ ends with exactly $1$ zero.

\begin{lemma}\label{periodic series zeros}
Let $T \in \ring[z]$ be a polynomial with $t \colonequal \deg T \geq 1$.
If the coefficients of $z^0$ and $z^t$ in $T$ are nonzero modulo~$p$, then the coefficient sequence of $\frac{1}{T}$ is periodic, and its period ends with exactly $t - 1$ zeros.
\end{lemma}

\begin{proof}
Write $T = \sum_{i = 0}^t c_i z^i$ and $\frac{1}{T} = \sum_{n \geq 0} a(n) z^n$.
The sequence $a(n)_{n \geq 0}$ is periodic by a standard argument using the invertibility of $c_0$ and $t\geq 1$.
In what follows, we use periodicity to relate the end of the period to the beginning of the sequence.
We iterate an argument which gives one zero at each step.

Comparing coefficients on both sides of the equation $T \sum_{n \geq 0} a(n) z^n = 1$, we obtain $c_0 a(0) = 1$ and
\begin{equation}\label{initial coefficients}
	c_{t - i} a(0) + c_{t - i - 1} a(1) + \dots + c_0 a(t - i) = 0
\end{equation}
for all $i$ in the range $1 \leq i \leq t - 1$.
Moreover, for all $n \geq 0$ we have the recurrence
\begin{equation}\label{recurrence}
	c_t a(n) + c_{t - 1} a(n + 1) + c_{t - 2} a(n + 2) + \dots + c_0 a(n + t) = 0.
\end{equation}
Let $m$ be the period length of $a(n)_{n \geq 0}$.

As $m \geq 1$, we can set $n = m - 1$ in Equation~\eqref{recurrence}.
The periodicity of $a(n)_{n \geq 0}$ and Equation~\eqref{initial coefficients} with $i = 1$ gives
\begin{align*}
	c_t a(m - 1)
	&= -c_{t - 1} a(m) - c_{t - 2} a(m + 1) - \dots - c_0 a(m - 1 + t) \\
	&= -c_{t - 1} a(0) - c_{t - 2} a(1) - \dots - c_0 a(t - 1) \\
	&= 0.
\end{align*}
Since $c_t$ is invertible, we have $a(m - 1) = 0$.
As $a(0) \neq 0$, this implies $m \geq 2$.

Therefore, we can set $n = m - 2$ in Equation~\eqref{recurrence}.
Periodicity, $a(m - 1) = 0$, and Equation~\eqref{initial coefficients} with $i = 2$ gives
\begin{align*}
	c_t a(m - 2)
	&= -c_{t - 1} a(m - 1) - c_{t - 2} a(m) - \dots - c_0 a(m - 2 + t) \\
	&= 0 - c_{t - 2} a(0) - \dots - c_0 a(t - 2) \\
	&= 0.
\end{align*}
Iterating this argument for $i \in \{3, 4, \dots, t - 2\}$, we obtain $a(n) = 0$ for all $n$ satisfying $m - (t - 2) \leq n \leq m - 1$ and $m \geq t - 1$.

Setting $n = m - (t - 1)$ and $i = t - 1$ gives
\begin{align*}
	c_t a(m - t + 1)
	&= -c_{t - 1} a(m - t + 2) - \dots - c_2 a(m - 1) - c_1 a(m) - c_0 a(m + 1) \\
	&= 0 + \dots + 0 - c_1 a(0) - c_0 a(1) \\
	&= 0.
\end{align*}
Therefore $a(n) = 0$ for all $n$ satisfying $m - (t - 1) \leq n \leq m - 1$ and $m \geq t$.

Finally, setting $n = m - t$ gives
\begin{align*}
	c_t a(m - t)
	&= -c_{t - 1} a(m - t + 1) - \dots - c_1 a(m - 1) - c_0 a(m) \\
	&= 0 + \dots + 0 - c_0 a(0) \\
	&= -1,
\end{align*}
leaving exactly $t - 1$ zeros at the end of the period of $a(n)_{n \geq 0}$.
\end{proof}

The following corollary of Lemma~\ref{periodic series zeros} tells us that certain Laurent polynomials produce periodic series expansions.

\begin{corollary}\label{periodic series zeros - Laurent}
Let $T \in \ring[z]$ such that $t \colonequal \deg T \geq 1$ and the coefficients of $z^0$ and $z^t$ in $T$ are nonzero modulo~$p$.
For all $i$ in the range $0 \leq i \leq t - 1$, the coefficient sequence of $\frac{1}{z^{-i} T}$ is periodic, and its period begins with $i$ zeros and ends with $t - 1 - i$ zeros.
\end{corollary}

We next combine Corollary~\ref{periodic series zeros - Laurent} with Engstrom's bounds to get information on the period lengths of the coefficient sequences of $\frac{1}{R \bmod p}$ and $\frac{1}{R^{p^{\alpha - 1}}}$ in the case that $R \in z^{-1} \ring[z]$.
These period lengths depend on $R \bmod p$.
Its \emph{factorization into irreducibles} is $(R \bmod p) = c z^{e_0} R_1^{e_1} \cdots R_k^{e_k}$, where $z, R_1, \dots, R_k \in \F_p[z]$ are distinct, monic, irreducible polynomials, $c \in \F_p$, $e_0 \geq -1$, and $e_i \geq 1$ for all $i \in \{1, \dots, k\}$.
If $(R \bmod p) \neq 0$, define $\deg(R \bmod p)$ to be the largest exponent of $z$ with a nonzero coefficient in the expansion of $R \bmod p$ in the monomial basis.
Finally, let $\nu_p(m)$ denote the $p$-adic valuation of $m$.

\begin{corollary}\label{period length comparison}
Let $R \in z^{-1} \ring[z]$ be a nonzero Laurent polynomial.
Factor $(R \bmod p) = c z^{e_0} R_1^{e_1} \cdots R_k^{e_k}$ into irreducibles.
Suppose that $e_0 \in \{-1, 0\}$ and $\deg(R \bmod p) \geq 1$.
Then the coefficient sequence of $\frac{1}{R \bmod p}$ is periodic, and its period length $m$ satisfies $\nu_p(m) \leq \ceil{\log_p \max(e_1, \dots, e_k)}$.
Moreover, the coefficient sequence of $\frac{1}{R^{p^{\alpha - 1}}}$ is periodic, and its period length divides $p^{2 (\alpha - 1)} m$.
\end{corollary}

\begin{proof}
First we address the series $\frac{1}{R \bmod p}$.
If $e_0 = 0$, then Theorem~\ref{constant-recursive modulo prime} already tells us that its coefficient sequence is periodic and that its period length $m$ satisfies $\nu_p(m) \leq \ceil{\log_p \max(e_1, \dots, e_k)}$.
If $e_0 = -1$, we write $\frac{1}{R \bmod p} = \frac{1}{z^{-1}(z R \bmod p)}$ and apply Corollary~\ref{periodic series zeros - Laurent} to conclude that the coefficient sequence of $\frac{1}{R \bmod p}$ is periodic with the same period length as $\frac{1}{zR \bmod p}$;
since $zR$ is a polynomial, the period length $m$ satisfies $\nu_p(m) \leq \ceil{\log_p \max(e_1, \dots, e_k)}$ by Theorem~\ref{constant-recursive modulo prime}.

Next we show that the coefficient sequence of $\frac{1}{R^{p^{\alpha - 1}}}$ is periodic.
Since the coefficient of $z^{e_0}$ is nonzero in $R \bmod p$ and $e_0 \leq 0$, we have a power series expansion for $\frac{1}{R^{p^{\alpha - 1}}}$.
Let $r \colonequal \deg(R \bmod p)$.

If $e_0 = 0$, we apply Lemma~\ref{lifting-the-exponent} to see that $\frac{1}{R^{p^{\alpha - 1}}} \equiv \frac{1}{(R \bmod p)^{p^{\alpha - 1}}} \mod p^\alpha$.
Since the leading coefficient of $R\bmod p$ is nonzero modulo~$p$, the leading coefficient of $(R \bmod p)^{p^{\alpha - 1}}$ is nonzero modulo~$p^\alpha$.
Then, since $r \geq 1$, the coefficient sequence of $\frac{1}{(R \bmod p)^{p^{\alpha - 1}}}$ is periodic.
This implies that the coefficient sequence of $\frac{1}{R^{p^{\alpha - 1}}}$ is periodic.

If $e_0 = -1$, (using Lemma~\ref{lifting-the-exponent} again) we apply Lemma~\ref{periodic series zeros} to $T \colonequal (zR)^{p^{\alpha - 1}} = (z R \bmod p)^{p^{\alpha - 1}}$, whose degree is $p^{\alpha - 1} (r + 1)$.
Since $r \geq 1$, we deduce that the period of the coefficient sequence of $\frac{1}{T}$ ends with $p^{\alpha - 1} (r + 1) - 1 \geq p^{\alpha - 1}$ zeros.
As $R^{p^{\alpha - 1}} = z^{-p^{\alpha - 1}} T$, we can apply Corollary~\ref{periodic series zeros - Laurent} with $i = p^{\alpha - 1}$ to conclude that the coefficient sequence of $\frac{1}{R^{p^{\alpha - 1}}}$ is periodic.

Finally we bound the period length of $\frac{1}{R^{p^{\alpha - 1}}}$.
To do this, we first bound the period length of $\frac{1}{R^{p^{\alpha - 1}}} \bmod p$.
Since $\F_p$ has characteristic~$p$, we have $\frac{1}{R(z)^p} \equiv \frac{1}{R(z^p)} \mod p$, so that raising $\frac{1}{R}$ to the power $p$ causes its coefficient sequence modulo~$p$ to dilate by a factor of $p$.
Iterating, we get $\frac{1}{R(z)^{p^{\alpha - 1}}} \equiv \frac{1}{R(z^{p^{\alpha - 1}})} \mod p$.
Therefore the period length of $\frac{1}{R^{p^{\alpha - 1}}} \bmod p$ divides $p^{\alpha - 1} m$.

Now we use the period length of $\frac{1}{R^{p^{\alpha - 1}}} \bmod p$ to bound the period length of $\frac{1}{R^{p^{\alpha - 1}}}$ modulo~$p^\alpha$.
We apply Theorem~\ref{constant-recursive modulo prime power} with $T = R^{p^{\alpha - 1}}$ to see that the period length of the coefficient sequence of $\frac{1}{R^{p^{\alpha - 1}}}$ divides $p^{2 (\alpha - 1)} m$.
\end{proof}

\section{Orbit size of a univariate Laurent polynomial under $\lambda_0$}\label{section: orbit size univariate - ring}

In this section, our goal is to prove Corollary~\ref{combined univariate orbit size upper bound - ring}, which tells us that the orbit under a univariate Cartier operator $\lambda_0$ is finite, and which gives bounds on the transient and period length of this orbit.

We begin with the following result, which lets us restrict attention to Laurent polynomials $S$ with $\deg S \leq p^{\alpha - 1} \deg(R \bmod p)$ in Theorems~\ref{square-free orbit size upper bound - ring - positive degree}, \ref{univariate orbit size upper bound - ring}, and \ref{square-free orbit size upper bound - ring}.
The proof uses the fact that if $f(x) = \floor{\frac{x - p^{\alpha - 1} r}{p}} + p^{\alpha - 1} r$, then, for every $x \geq p^{\alpha - 1} r$ and $n \geq \floor{\log_p(x - p^{\alpha - 1} r)} + 1$, we have $f^n(x) = p^{\alpha - 1} r$.

\begin{lemma}\label{fixed point - ring}
Let $R \in z^{-1} \ring[z]$ be a Laurent polynomial, let $r = \deg(R \bmod p)$, and define $\lambda_0$ on $\ring[z,z^{-1}]$ by $\lambda_0(S) = \Lambda_0\!\paren{S R^{p^\alpha - p^{\alpha - 1}}}$.
Let $S \in \ring[z, z^{-1}]$, let $s = \deg S$, and suppose that $s>p^{\alpha - 1} r$.
If $n\geq \floor{\log_p(s - p^{\alpha - 1} r)} + 1$, then $\deg \lambda_0^n(S)\leq p^{\alpha - 1} r$.
\end{lemma}

Our strategy is to transfer periodicity of a series expansion to eventual periodicity of the orbit under $\lambda_0$.
We first illustrate with an example.

\begin{example}
As in Example~\ref{Fibonacci mod 4}, let $p = 2$, $\alpha = 2$, and $R = -z^2 - z + 1 \in \mathcal{R}_4[z]$.
Write $\frac{1}{R^{p^{\alpha - 1}}} = \frac{1}{R^2} = \sum_{n \geq 0} a(n) z^n \in \mathcal{R}_4\doublebracket{z}$.
The sequence $a(n)_{n \geq 0}$ is periodic with period length~$12$.
In light of Lemma~\ref{fixed point - ring}, we consider Laurent polynomials $S \in \mathcal{R}_4[z]$ such that $-1 = 1-p^{\alpha - 1}\leq \mindeg S$ and $\deg S \leq p^{\alpha - 1} \deg(R \bmod p) = 4$.
Let $j \in \{-1, 0, 1, 2, 3, 4\}$, so that each monomial in $S$ is of the form $c \, z^j$.
By Proposition~\ref{Cartier - ring},
\[
	\lambda_0(z^j) = \Lambda_0(z^j R^2) = \Lambda_0(\tfrac{z^j}{R^2} R^4) = \Lambda_0(\tfrac{z^j}{R^2}) R^2.
\]
If $j = -1$, it can be verified that $\lambda_0(z^{-1}) = 0$.
For $j \in \{0, 1, 2, 3, 4\}$, we show that $\Lambda_0^4(\frac{z^j}{R^2}) = \Lambda_0^2(\frac{z^j}{R^2})$, so that
\[
	\lambda_0^4(z^j)
	= \Lambda_0^4(\tfrac{z^j}{R^2}) R^2
	= \Lambda_0^2(\tfrac{z^j}{R^2}) R^2
	= \lambda_0^2(z^j).
\]
Since $\frac{z^j}{R^2} = \sum_{n \geq j} a(n - j) z^n$ and $\Lambda_0^2(z^n) = 0$ if $n \nequiv 0 \mod 4$, we have
\[
	\Lambda_0^2\!\paren{\frac{z^j}{R^2}}
	= \Lambda_0^2\!\paren{\sum_{n \geq \ceil{j/4}} a(4 n - j) z^{4 n}}
	= \sum_{n \geq \ceil{j/4}} a(4 n - j) z^n.
\]
Similarly,
\[
	\Lambda_0^4\!\paren{\frac{z^j}{R^2}}
	= \Lambda_0^2\!\paren{\sum_{n \geq \ceil{j/4}} a(4 n - j) z^n}
	= \sum_{n \geq \ceil{j/16}} a(16 n - j) z^n.
\]
Since $4 n - j \equiv 16 n - j \mod 12$ and $\ceil{j/4} = \ceil{j/16}$ for $j \in \{0, 1, 2,3,4\}$, it follows that $\Lambda_0^4(\frac{z^j}{R^2}) = \Lambda_0^2(\frac{z^j}{R^2})$, as desired.
By linearity, this implies $\lambda_0^4(S) = \lambda_0^2(S)$ for all $S \in \mathcal{R}_4[z]$ with $-1\leq \mindeg S$ and $\deg S \leq 4$.
\end{example}

Let $\ord_m(p)$ be the eventual period length of the sequence $(p^n \bmod m)_{n \geq 0}$.
That is, $\ord_m(p)$ is the smallest integer $k \geq 1$ such that $p^{n + k} \equiv p^n \mod m$ for all sufficiently large $n$.
When $p$ and $m$ are relatively prime, $\ord_m(p)$ is the usual multiplicative order of $p$ modulo~$m$.
When $p$ and $m$ are not relatively prime, then $\ord_m(p) = \ord_{m'}(p)$ where $m = p^{\nu_p(m)} m'$.

Engstrom's Theorems~\ref{constant-recursive modulo prime} and \ref{constant-recursive modulo prime power} let us bound the eventual period length $m$ of $\frac{1}{R\bmod p}$, and in Corollary~\ref{period length comparison} we showed that the period length of the expansion of $\frac{1}{R^{p^{\alpha - 1}}}$ divides $p^{2 (\alpha - 1)} m$.
We will see that the orbit size under $\lambda_0$ is related to $m$.
Since $\Lambda_0^k\!\paren{\sum_{n \geq 0} a(n) z^n} = \sum_{n \geq 0} a(p^k n) z^n$, if $a(n)_{n \geq 0}$ has period length $m$ then $\ord_m(p)$ determines the generic orbit size under $\Lambda_0$.
We will use the next lemma to evaluate it.

\begin{lemma}\label{order}
Let $p \geq 2$ be an integer, let $r_1, \dots, r_k$ be positive integers, and let $L = \lcm_{1 \leq i \leq k} (p^{r_i} - 1)$.
Then $\ord_L(p) = \lcm(r_1, \dots, r_k)$.
\end{lemma}

\begin{proof}
For each $i$, we have that $p^{r_i} - 1$ divides $p^{\lcm(r_1, \dots, r_k)} - 1$.
Therefore $L$ divides $p^{\lcm(r_1, \dots, r_k)} - 1$.
It follows that $p^{\lcm(r_1, \dots, r_k)} - 1 \equiv 0 \mod L$, and therefore $\ord_L(p)$ divides $\lcm(r_1, \dots, r_k)$.

It remains to show that $\lcm(r_1, \dots, r_k)$ divides $\ord_L(p)$.
For each $i$, the definition of $\ord_{p^{r_i} - 1}(p)$ implies $p^{\ord_{(p^{r_i} - 1)}(p)} \equiv 1 \mod p^{r_i} - 1$.
It follows that $p^{r_i} - 1$ divides $p^{\ord_{(p^{r_i} - 1)}(p)} - 1$.
Since $p^a - 1 \mid p^b - 1$ implies $a \mid b$, we have $r_i \mid \ord_{p^{r_i} - 1}(p)$.
By the definition of $L$, we have $p^{r_i} - 1 \mid L$, so this implies $r_i \mid \ord_L(p)$.
Every prime power dividing $\lcm(r_1, \dots, r_k)$ divides some $r_i$, so $\lcm(r_1, \dots, r_k)$ divides $\ord_L(p)$.
\end{proof}

The next several results establish the orbit size under $\lambda_0$ in various settings.
We start with Theorem~\ref{square-free orbit size upper bound - ring - positive degree}, where we assume that $\deg(R \bmod p)\geq 1$ and $\mindeg (R \bmod p)\in \{-1,0\}$.
Then we consider $\mindeg (R \bmod p)\geq 1$ in Theorem~\ref{univariate orbit size upper bound - ring}.
Finally in Theorems~\ref{square-free orbit size upper bound - ring} and \ref{univariate orbit size upper bound - ring - degree 0}, we deal with the remaining possibilities.
The statement of Corollary~\ref{combined univariate orbit size upper bound - ring} covers all cases.

\begin{theorem}\label{square-free orbit size upper bound - ring - positive degree}
Let $R \in z^{-1} \ring[z]$ be a nonzero Laurent polynomial.
Factor $(R \bmod p) = c z^{e_0} R_1^{e_1} \cdots R_k^{e_k}$ into irreducibles.
Suppose that $e_0 \in \{-1, 0\}$ and $\deg(R \bmod p) \geq 1$.
Let $\ell = \lcm(\deg R_1, \dots, \deg R_k)$.
Define $\lambda_0$ on $\ring[z, z^{-1}]$ by $\lambda_0(S) = \Lambda_0\!\paren{S R^{p^\alpha - p^{\alpha - 1}}}$.
If $S \in \ring[z, z^{-1}]$ with $1 - p^{\alpha - 1} \leq \mindeg S$ and $\deg S \leq p^{\alpha - 1} \deg(R \bmod p)$, then
\[
	\lambda_0^{n + \ell}(S) = \lambda_0^n(S)
\]
for all $n \geq \ceil{\log_p \max(e_1, \dots, e_k)} + 2 (\alpha - 1)$.
\end{theorem}

In particular, the eventual period length of the orbit is independent of $\alpha$.

\begin{proof}
Let $r \colonequal \deg(R \bmod p)$.
By Corollary~\ref{period length comparison}, the coefficient sequence of $\frac{1}{R \bmod p}$ is periodic, and its period length $m$ satisfies $\nu_p(m) \leq \ceil{\log_p \max(e_1, \dots, e_k)}$.
Let $t \colonequal \nu_p(m)+ 2 (\alpha - 1)$.
We will show that if $n\geq t$, then $\lambda_0^{n + \ord_m(p)}(S) = \lambda_0^n(S)$.
Then, by Theorem~\ref{constant-recursive modulo prime}, $m$ divides $p^{\lceil \log_p \max(e_1, \dots, e_k) \rceil} L$, where $L = \lcm_{1 \leq i \leq k} (p^{\deg R_i} - 1)$.
Since $\gcd(L, p) = 1$, we have $\ord_m(p) = \ord_L(p)$.
Applying Lemma~\ref{order} with $r_i = \deg R_i$ gives that $\ord_m(p) = \ell$, and the theorem statement will follow.

Corollary~\ref{period length comparison} also tells us that the coefficient sequence of $\frac{1}{R^{p^{\alpha - 1}}}$ is periodic with period length dividing $p^{2 (\alpha - 1)} m$.
Therefore $\Lambda_0^t\!\paren{\frac{1}{R^{p^{\alpha - 1}}}}$ has period length dividing $m$.
Write $\frac{1}{R^{p^{\alpha - 1}}} = \sum_{n \geq 0} a(n) z^n \in \ring\doublebracket{z}$.
The period of $a(n)_{n \geq 0}$ ends with exactly $p^{\alpha - 1} r - 1$ zeros; this follows from Corollary~\ref{periodic series zeros - Laurent} when $e_0 = 0$ and is proved in the proof of Corollary~\ref{period length comparison} when $e_0 = -1$.
In other words, we have
\begin{equation}\label{zeros}
	a(p^{2 (\alpha - 1)} m - i) = 0
\end{equation}
for all $i \in \{1, \dots, p^{\alpha - 1} r - 1\}$.

Let $j \in \{1 - p^{\alpha - 1}, \dots, p^{\alpha - 1} r\}$.
By Proposition~\ref{Cartier - ring},
\begin{equation}\label{equation: monomial image}
	\lambda_0(z^j)
	= \Lambda_0\!\left(z^j R^{p^\alpha - p^{\alpha - 1}}\right)
	= \Lambda_0\!\left(\frac{z^j}{R^{p^{\alpha - 1}}}\right) R^{p^{\alpha - 1}}.
\end{equation}
Therefore, by iterating, $\lambda_0^n(z^j) = \Lambda_0^n\!\left(\frac{z^j}{R^{p^{\alpha - 1}}}\right) R^{p^{\alpha - 1}}$ for all $n \geq 0$.
We show
\[
	\Lambda_0^{t + \ord_m(p)}\!\left(\frac{z^j}{R^{p^{\alpha - 1}}}\right)
	= \Lambda_0^t\!\left(\frac{z^j}{R^{p^{\alpha - 1}}}\right);
\]
this implies $\lambda_0^{t + \ord_m(p)}(z^j) = \lambda_0^t(z^j)$, and the statement will follow from the linearity of $\lambda_0$.
Since $\Lambda_0^k(z^n) = 0$ if $n \nequiv 0 \mod p^k$, we have
\begin{align*}
	\Lambda_0^t\!\paren{\frac{z^j}{R^{p^{\alpha - 1}}}}
	&= \Lambda_0^t\!\paren{\sum_{n \geq j} a(n - j) z^n}
	= \Lambda_0^t\!\paren{\sum_{n \geq \ceil{j/p^t}} a(p^t n - j) z^{p^t n}} \\
	&= \sum_{n \geq \ceil{j/p^t}} a(p^t n - j) z^n.
\end{align*}
Similarly,
\[
	\Lambda_0^{t + \ord_m(p)}\!\paren{\frac{z^j}{R^{p^{\alpha - 1}}}}
	= \sum_{n \geq \ceil{j/p^{t + \ord_m(p)}}} a(p^{t + \ord_m(p)} n - j) z^n.
\]

Recall that $a(n)_{n \geq 0}$ is periodic with period length dividing $p^{2 (\alpha - 1)} m$.
Since $p^{t + \ord_m(p)} \equiv p^t \mod p^{2 (\alpha - 1)} m$, we have $a(p^{t + \ord_m(p)} n - j) = a(p^t n - j)$ for all $n$ such that $p^t n - j\geq 0$, that is, when $n \geq \ceil{j/p^t}$.
Therefore
\begin{equation}\label{equation: Cartier difference}
	\Lambda_0^{t + \ord_m(p)}\!\left(\frac{z^j}{R^{p^{\alpha - 1}}}\right) - \Lambda_0^t\!\left(\frac{z^j}{R^{p^{\alpha - 1}}}\right)
	= \sum_{n = \ceil{j/p^{t + \ord_m(p)}}}^{\ceil{j/p^t} - 1} a(p^{t + \ord_m(p)} n - j) z^n.
\end{equation}
If $j \leq 1$, then this sum is empty and therefore $0$.
So assume $j \in \{2, \dots, p^{\alpha - 1} r\}$.
It remains to show that $a(p^{t + \ord_m(p)} n - j) = 0$ for all $n$ in the range of summation of the sum in Equation~\eqref{equation: Cartier difference}.

Take $n$ in the range of summation $\{\ceil{j/p^{t + \ord_m(p)}},\dots, \ceil{j/p^t} - 1\}$ in Equation~\eqref{equation: Cartier difference}.
We have $j - n \in \{j+1 - \ceil{j/p^t}, \dots, j- \ceil{j/p^{t + \ord_m(p)}}\}$.
Since $2 \leq j \leq p^{\alpha - 1} r$, it follows with generosity that $j - n \in \{1, \dots, p^{\alpha - 1} r - 1\}$.
The period length of $a(n)_{n \geq 0}$ divides $p^{2 (\alpha - 1)} m$, so $a(p^{2 (\alpha - 1)} m n + n - j) = a(p^{2 (\alpha - 1)} m - (j - n))$.
Since $a(p^{2 (\alpha - 1)} m - (j - n)) = 0$ by Equation~\eqref{zeros}, this implies $a(p^{2 (\alpha - 1)} m n + n - j) = 0$, as desired.
\end{proof}

Next we show that, if $R$ is a polynomial and is divisible by $z$, then elements sufficiently far out in orbits under $\lambda_0$ are also divisible by a certain power of $z$.

\begin{proposition}\label{power of y transient size upper bound - ring}
Let $R \in \ring[z]$ be a nonzero polynomial such that $(R \bmod p) = z^{e_0} G$ for some $G \in \F_p[z]$ where $e_0 \geq 1$ and $G$ is not divisible by $z$.
If $S \in \ring[z, z^{-1}]$ with $1 - p^{\alpha - 1} \leq \mindeg S$ and $\deg S \leq p^{\alpha - 1} \deg(R \bmod p)$, then $\lambda_0^n(S)$ is a polynomial for all $n \geq 1$.
Moreover, the polynomial $\lambda_0^n(S)$ is divisible by $z^{p^{\alpha - 1} e_0}$ for all $n \geq \floor{\log_p e_0} + \alpha$.
\end{proposition}

\begin{proof}
Let $s = \mindeg S$, and write $S = z^s T$ (so that $T \in \ring[z]$ is a polynomial that is not divisible by $z$).
By Proposition~\ref{lifting-the-exponent lambda}, we have
\begin{align*}
	\lambda_0(S)
	= \Lambda_0\!\paren{S R^{p^\alpha - p^{\alpha - 1}}}
	&= \Lambda_0\!\paren{S (R \bmod p)^{p^\alpha - p^{\alpha - 1}}} \\
	&= \Lambda_0\!\paren{z^{s + e_0 (p^\alpha - p^{\alpha - 1})} T G^{p^\alpha - p^{\alpha - 1}}}.
\end{align*}
Since $1 - p^{\alpha - 1} \leq s$ and $e_0 \geq 1$, this implies that $\lambda_0(S)$ is a polynomial.
Therefore $\lambda_0^n(S)$ is a polynomial for each $n \geq 1$, and $\lambda_0(S)$ is divisible by $z^{f(s)}$, where $f(x) = e_0 p^{\alpha - 1} + \ceil{\frac{x - e_0 p^{\alpha - 1}}{p}}$.
If $n \geq \floor{\log_p e_0} + \alpha$, then $f^n(x) \geq e_0 p^{\alpha - 1}$, so $\lambda_0^n(S)$ is divisible by $z^{e_0 p^{\alpha - 1}}$.
\end{proof}

Finally, we use Proposition~\ref{power of y transient size upper bound - ring} to remove the restriction in Theorem~\ref{square-free orbit size upper bound - ring - positive degree} that $e_0 \in \{-1, 0\}$.
We show that if $e_0 \geq 1$ then every application of $\lambda_0$ pushes the image into a smaller $\ring$-module until we are emulating the map $\lambda_0$ for a polynomial $G$ satisfying $\mindeg G = 0$.

\begin{theorem}\label{univariate orbit size upper bound - ring}
Let $R \in \ring[z]$ be a nonzero polynomial.
Factor $(R \bmod p) = c z^{e_0} R_1^{e_1} \cdots R_k^{e_k}$ into irreducibles, and define $G = c R_1^{e_1} \cdots R_k^{e_k}$.
Suppose that $e_0 \geq 1$ and $\deg G \geq 1$.
Let $\ell = \lcm(\deg R_1, \dots, \deg R_k)$ and let
\[
	t = \max\!\paren{\floor{\log_p e_0} + \alpha, \ceil{\log_p \max(e_1, \dots, e_k)} + 2 (\alpha - 1)}.
\]
Let $m$ be the period length of the series expansion of $\frac{1}{G}$.
For all $S \in \ring[z, z^{-1}]$ with $1 - p^{\alpha - 1} \leq \mindeg S$ and $\deg S \leq p^{\alpha - 1} \deg(R \bmod p)$, the orbit size of $S$ under $\lambda_0$ is at most $t + \ell$.
\end{theorem}

\begin{proof}
Since $t \geq \floor{\log_p e_0} + \alpha$, Proposition~\ref{power of y transient size upper bound - ring} tells us that $\lambda_0^t(S)$ is divisible by $z^{e_0 p^{\alpha - 1}}$.
Define $T \in \ring[z]$ by $\lambda_0^t(S) = z^{e_0 p^{\alpha - 1}} T$; we claim $\deg T \leq p^{\alpha - 1} \deg G$.
Since $\deg \lambda_0(S) \leq p^{\alpha - 1} \deg(R \bmod p)$, we have
\begin{align*}
	\deg T
	&= \deg \lambda_0^t(S) - e_0 p^{\alpha - 1} \\
	&\leq p^{\alpha - 1} \deg(R \bmod p) - e_0 p^{\alpha - 1} \\
	&= p^{\alpha - 1} \deg G,
\end{align*}
as claimed.

For all $T \in \ring[z]$ (and in particular for the $T$ satisfying $\lambda_0^t(S) = T z^{e_0 p^{\alpha - 1}}$),
\begin{align*}
	\lambda_0(T z^{e_0 p^{\alpha - 1}})
	&= \Lambda_0\!\paren{T z^{e_0 p^{\alpha - 1}} (R \bmod p)^{p^\alpha - p^{\alpha - 1}}} \\
	&= \Lambda_0\!\paren{T c^{p^\alpha - p^{\alpha - 1}} z^{e_0 p^{\alpha - 1} + e_0 (p^\alpha - p^{\alpha - 1})} G^{p^\alpha - p^{\alpha - 1}}} \\
	&= \Lambda_0\!\paren{T c^{p^\alpha - p^{\alpha - 1}} G^{p^\alpha - p^{\alpha - 1}} z^{e_0 p^\alpha}} \\
	&= \Lambda_0\!\paren{T \, G^{p^\alpha - p^{\alpha - 1}}} z^{e_0 p^{\alpha - 1}}.
\end{align*}
Accordingly, define $\kappa_0 \colon \ring[z] \to \ring[z]$ by $\kappa_0(T) = \Lambda_0(T \, G^{p^\alpha - p^{\alpha - 1}})$ (where here we can take any lift of $G$ to $\ring[z]$), so that $\lambda_0(T z^{e_0 p^{\alpha - 1}}) = \kappa_0(T) z^{e_0 p^{\alpha - 1}}$.
Iterating, we have $\lambda_0^\ell(T z^{e_0 p^{\alpha - 1}}) = \kappa_0^\ell(T) z^{e_0 p^{\alpha - 1}}$.
The polynomial $G$ satisfies the conditions of Theorem~\ref{square-free orbit size upper bound - ring - positive degree}.
Applying Theorem~\ref{square-free orbit size upper bound - ring - positive degree} to $\kappa_0$, we have $\kappa_0^{n + \ell}(T) = \kappa_0^n(T)$ for all $n \geq t \geq \ceil{\log_p \max(e_1, \dots, e_k)} + 2 (\alpha - 1)$ since $\deg T \leq p^{\alpha - 1} \deg G$.
Therefore
\[
	\lambda_0^{t + \ell}(S)
	= \lambda_0^\ell\!\paren{\lambda_0^t(S)}
	= \lambda_0^\ell\!\paren{T z^{e_0 p^{\alpha - 1}}}
	= \kappa_0^\ell(T) z^{e_0 p^{\alpha - 1}}
	= T z^{e_0 p^{\alpha - 1}}
	= \lambda_0^t(S),
\]
so the orbit of $S$ under $\lambda_0$ contains at most $t + \ell$ elements.
\end{proof}

Next we complement Theorem~\ref{square-free orbit size upper bound - ring - positive degree} to allow $\deg(R \bmod p) \geq -1$.

\begin{theorem}\label{square-free orbit size upper bound - ring}
Let $R \in z^{-1} \ring[z]$ be a nonzero Laurent polynomial such that $\deg(R \bmod p) \in \{-1, 0\}$.
Let $e_0 = \mindeg (R \bmod p)$.
Define $\lambda_0$ on $\ring[z, z^{-1}]$ by $\lambda_0(S) = \Lambda_0\!\paren{S R^{p^\alpha - p^{\alpha - 1}}}$.
Let $S \in \ring[z, z^{-1}]$ with $1 - p^{\alpha - 1} \leq \mindeg S$ and $\deg S \leq p^{\alpha - 1} \deg(R \bmod p)$.
\begin{itemize}
\item
If $e_0 = -1$, then $\lambda_0^{n + 1}(S) = \lambda_0^n(S)$ for all $n \geq 2 (\alpha - 1)$.
\item
If $e_0 = 0$, then $\lambda_0^{n+1}(S) = \lambda_0^n(S)$ for all $n \geq \alpha - 1$.
\end{itemize}
\end{theorem}

\begin{proof}
Let $r \colonequal \deg(R \bmod p)$.
The assumption that $r \in \{-1,0\}$ implies $e_0 \in \{-1,0\}$.
If $r = -1$, the only Laurent polynomial $S$ satisfying the constraints is $S = 0$, and the conclusion holds.

Now let $r = 0$, so that we have $(R \bmod p) = b z^{-1} + c$ for some $b, c \in \F_p$ with $c \neq 0$.
We consider two cases: $b\neq 0$ so that $e_0=-1$, and $b=0$ so that $e_0=0$.

Suppose first that $b\neq 0$, and
let $m$ be the eventual period length of the coefficient sequence of $\frac{1}{R \bmod p}$.
Then
\begin{align*}
	\frac{1}{R \bmod p}
	&= \frac{z}{b (1 - (-c/b) z)}
	= \frac{1}{b} \sum_{n \geq 0} (-c/b)^n z^{n+1}.
\end{align*}
In particular $m$ divides $p-1$.

We write $\frac{1}{R^{p^{\alpha - 1}}} = \sum_{n \geq 0} a(n) z^n \in \ring\doublebracket{z}$, and
we apply Lemma~\ref{periodic series zeros} to the polynomial $T \colonequal (zR)^{p^{\alpha - 1}} = (zR\bmod p)^{p^{\alpha - 1}}$, whose degree is $p^{\alpha - 1}$.
We deduce that the period of the coefficient sequence of $\frac{1}{T}$ ends with $p^{\alpha - 1} - 1$ zeros.
As $R^{p^{\alpha - 1}} = z^{-p^{\alpha - 1}} T$, we can apply Corollary~\ref{periodic series zeros - Laurent} with $i = p^{\alpha - 1}-1$ to conclude that $a(n)_{n \geq 0}$ is eventually periodic, with transient length $1$.
Moreover, by Corollary~\ref{period length comparison}, the eventual period length of $a(n)_{n \geq 0}$ divides $p^{2 (\alpha - 1)}m$.

Let $j \in \{1 - p^{\alpha - 1}, \dots, 0\}$, so that $k z^j$ is a monomial in $S$.
Equation~\eqref{equation: monomial image} still holds, so that $\lambda_0^n(z^j) = \Lambda_0^n\!\left(\frac{z^j}{R^{p^{\alpha - 1}}}\right) R^{p^{\alpha - 1}}$ for all $n \geq 0$.
First assume $j = 0$.
The series $\frac{z^j}{R^{p^{\alpha - 1}}}$ has transient length $1$.
Applying $2 (\alpha - 1)$ iterations of $\Lambda_0$ produces a series with transient length at most $1$ and eventual period length dividing $m$.
If $j \leq -1$ and $n\geq 2 (\alpha - 1)$, then $\Lambda_0^n\!\left(\frac{z^j}{R^{p^{\alpha - 1}}}\right)$ is periodic with period length dividing $m$.
Since $m$ divides $p-1$ in both cases, we have that the period length divides $p-1$, and therefore $p$ is congruent to $1$ modulo the period length.
Hence further applications of $\Lambda_0$ leave the series fixed.
This completes the case $b\neq 0$.

Finally let $b = 0$, so that $(R \bmod p) = c \in \F_p$.
Then $\lambda_0^n(S) = c^{n (p^\alpha - p^{\alpha - 1})} \Lambda_0^n(S) = \Lambda_0^n(S)$.
Using the conditions on $S$, one verifies that if $n \geq \alpha - 1$ then $\lambda_0^n(S)$ is the constant term of $S$.
\end{proof}

For the case when $\deg G = 0$, we have a similar result.
The proof follows that of Theorem~\ref{univariate orbit size upper bound - ring} but with an application of the second bullet point of Theorem~\ref{square-free orbit size upper bound - ring}.

\begin{theorem}\label{univariate orbit size upper bound - ring - degree 0}
Let $R \in \ring[z]$ be a nonzero polynomial.
Suppose that $(R \bmod p) = c z^{e_0}$ with $e_0 \geq 1$.
For all $S \in \ring[z, z^{-1}]$ with $1 - p^{\alpha - 1} \leq \mindeg S$ and $\deg S \leq p^{\alpha - 1} e_0$, the orbit size of $S$ under $\lambda_0$ is at most $\alpha - 1$.
\end{theorem}

We now have the following general result by combining Theorems~\ref{square-free orbit size upper bound - ring - positive degree}, \ref{univariate orbit size upper bound - ring}, \ref{square-free orbit size upper bound - ring}, and \ref{univariate orbit size upper bound - ring - degree 0}.
Namely, Theorem~\ref{square-free orbit size upper bound - ring - positive degree} covers the case $\deg(R \bmod p)\geq 1$ and $\mindeg (R \bmod p)\in \{-1,0\}$,
Theorem~\ref{univariate orbit size upper bound - ring} covers $\mindeg (R \bmod p)\geq 1$, and Theorems~\ref{square-free orbit size upper bound - ring} and \ref{univariate orbit size upper bound - ring - degree 0} cover the remaining cases.

\begin{corollary}\label{combined univariate orbit size upper bound - ring}
Let $R \in z^{-1} \ring[z]$ be a nonzero Laurent polynomial.
Factor $(R \bmod p) = c z^{e_0} R_1^{e_1} \cdots R_k^{e_k}$ into irreducibles.
Let
\begin{equation}\label{transient time - ring}
	t = \max\!\paren{\floor{\log_p \max(e_0, 1)} + \alpha, \ceil{\log_p \max(e_1, \dots, e_k, 1)} + 2 (\alpha - 1)}
\end{equation}
and $\ell = \lcm(\deg R_1, \dots, \deg R_k)$.
Define $\lambda_0$ on $\ring[z, z^{-1}]$ by $\lambda_0(S) = \Lambda_0\!\paren{S R^{p^\alpha - p^{\alpha - 1}}}$.
For all $S \in \ring[z, z^{-1}]$ with $1 - p^{\alpha - 1} \leq \mindeg S$ and $\deg S \leq p^{\alpha - 1} \deg(R \bmod p)$, the orbit size of $S$ under $\lambda_0$ is at most $t + \ell$.
\end{corollary}

\section{Orbit size under $\lambda_{0, 0}$}\label{section: orbit size - ring}

In this section, we prove Theorems~\ref{kernel size asymptotic bound - ring} and \ref{kernel size upper bound - ring}.
The last remaining piece is to determine, for each border, how many times we must apply $\lambda_{0, 0}$ before we can apply Corollary~\ref{combined univariate orbit size upper bound - ring}.

\begin{lemma}\label{right degree - ring}
Let
\begin{align*}
	u\subl &= \floor{\log_p \max\!\paren{p^{\alpha - 1} (d - 1 - \deg \pi_{x, 0}(Q)), 1}} + 1 \\
	u\subr &= \floor{\log_p \max\!\paren {p^{\alpha - 1} (d - 1 - \deg \pi_{x, h}(Q)), 1}} + 1 \\
	u\subt &= \floor{\log_p \max\!\paren {p^{\alpha - 1} (h - \deg \pi_{y, d - 1}(Q)), 1}} + 1.
\end{align*}
For all $S \in \val_{p/Q}(\mathcal V)$, we have
\begin{enumerate}
\item
$\deg \pi_{x, 0}(\lambda_{0, 0}^n(S)) \leq p^{\alpha - 1} \deg \pi_{x, 0}(Q)$ for all $n \geq u\subl$,
\item
$\deg \pi_{x, p^{\alpha - 1} h}(\lambda_{0, 0}^n(S)) \leq p^{\alpha - 1} \deg \pi_{x, h}(Q)$ for all $n \geq u\subr$, and
\item
$\deg \pi_{y, p^{\alpha - 1} (d - 1)}(\lambda_{0,0}^n(S)) \leq p^{\alpha - 1} \deg \pi_{y, d - 1}(Q)$ for all $n \geq u\subt$.
\end{enumerate}
\end{lemma}

\begin{proof}
We prove the second statement; the others are similar.
Let $R = \pi_{x, h}(Q)$.
Since $S \in \val_{p/Q}(\mathcal V)$, Lemma~\ref{degree bounds} tells us that $\deg_x S \leq p^{\alpha - 1} h$ and $\deg_ y S \leq p^{\alpha - 1} (d - 1)$.
In particular, $S$ satisfies $\deg \pi_{x, p^{\alpha - 1} h}(S)\leq p^{\alpha - 1}(d-1)$, which is $p^{\alpha - 1} (d - 1 - \deg \pi_{x, h}(Q))$ away from the target degree $p^{\alpha - 1} \deg \pi_{x, h}(Q)$.

Part~\ref{right border - ring} of Proposition~\ref{univariate emulation - ring} gives
\[
	\pi_{x, p^{\alpha - 1} h}(\lambda_{0, 0}(S))
	= \lambda_0(\pi_{x, p^{\alpha - 1} h}(S))
	= \Lambda_0\!\paren{\pi_{x, p^{\alpha - 1} h}(S) \cdot (\pi_{x, h}(Q))^{p^\alpha - p^{\alpha - 1}}}.
\]
This implies
\begin{align*}
	\deg \pi_{x, p^{\alpha - 1} h}(\lambda_{0, 0}(S))
	&\leq p^{\alpha - 2} (d -1) + \left(p^{\alpha - 1} - p^{\alpha - 2}\right) \deg \pi_{x, h}(Q) \\
	&= p^{\alpha - 2} (d - 1 - \deg \pi_{x, h}(Q)) + p^{\alpha - 1} \deg \pi_{x, h}(Q).
\end{align*}
Therefore $\deg \pi_{x, p^{\alpha - 1} h}(\lambda_{0, 0}(S))$ is at most $p^{\alpha - 2} (d - 1 - \deg \pi_{x, h}(Q))$ away from the target degree $p^{\alpha - 1} \deg \pi_{x, h}(Q)$.
Iterating $\lambda_{0,0}$ at most $u\subr$ times, and then applying $\pi_{x, p^{\alpha - 1} h}$, we obtain the target degree.
\end{proof}

We now prove Theorem~\ref{kernel size upper bound - ring}.
By Remark~\ref{h and d positive}, the assumption $h \geq 1$ is not restrictive.

\begin{main theorem - ring}
Let $p$ be a prime, let $\alpha \geq 1$, and let $F = \sum_{n \geq 0} a(n) x^n \in \Z_p\doublebracket{x}$ be the Furstenberg series associated with a polynomial $P \in \Z_p[x, y]$ such that $h \colonequal \deg_x(P \bmod p) \geq 1$.
Let $d = \deg_y(P \bmod p)$,
\[
	N = \tfrac{1}{6} \alpha (\alpha + 1) ((2 h d - 1) \alpha + h d + 1),
\]
and
\[
	u = \floor{\log_p \max\!\paren{\alpha (\deg_x(P \bmod p^\alpha) - h), \alpha (\deg_y(P \bmod p^\alpha) - d) + 1}} + 1.
\]
Let $Q \in \ring[x, y, y^{-1}]$ be a lift of $P/y \bmod p$ which has the same monomial support as $P/y \bmod p$, and define $u\subl$, $u\subr$, and $u\subt$ as in Lemma~\ref{right degree - ring}.
Then the size of $\ker_p((a(n) \bmod p^\alpha)_{n \geq 0})$ is at most
\begin{multline*}
	p^N + p^{N - \alpha (\alpha + 1) (h + d - 1) / 2} \Landaulcm(h, d, d) \\
	+ \max(u\subl, u\subr, u\subt)
	+ \floor{\log_p \max(h, d)}
	+ 2 \alpha - 1
	+ \tfrac{p^u - 1}{p - 1}.
\end{multline*}
\end{main theorem - ring}

\begin{proof}
By Corollary~\ref{kernel preliminary upper bound - ring}, we have
\[
	\size{\ker_p((a(n) \bmod p^\alpha)_{n \geq 0})}
	\leq p^N + \size{\orb_{\Lambda_0}(F \bmod p^\alpha)}+ \tfrac{p^u -1}{p-1}-u
\]
where $N$ is the dimension of the space $\mathcal W$ defined in Equation~\eqref{W definition - ring}.
Equivalently,
\[
	\size{\ker_p((a(n) \bmod p^\alpha)_{n \geq 0})}
	\leq p^N + \size{\orb_{\Lambda_0}(\Lambda_0^u(F \bmod p^\alpha)}+ \tfrac{p^u -1}{p-1}.
\]
It remains to bound $\size{\orb_{\Lambda_0}(\Lambda_0^u(F \bmod p^\alpha))} \leq \size{\orb_{\lambda_{0, 0}}(\lambda_{0, 0}^u(S_0))}$, where $S_0$ is the initial state.
Let $S_u \colonequal \lambda_{0,0}^u(S_0)$.
By Proposition~\ref{transient - ring}, we have $S_u \in \val_{p/Q}(\mathcal V)$.
We bound $\size{\orb_{\lambda_{0, 0}}(S_u)}$ by emulating $\lambda_{0, 0}$ with the appropriate univariate operators $\lambda_0$ on the left, right, and top borders of $\mathcal V$ and using a crude upper bound for the ``interior'' $\mathcal V^\circ$ defined below.
Corollary~\ref{univariate emulation - ring - tuple}, Lemma~\ref{right degree - ring}, and then Corollary~\ref{combined univariate orbit size upper bound - ring} will allow us to do this.

For $k \in \{0, 1, \dots, \alpha - 1\}$, let
\[
	V_k^\circ \colonequal
	\left\{
		\sum_{i = 1}^{(k + 1) h - 1} \sum_{j = \max(-k, -i)}^{(k + 1)(d - 1) - 1} c_{i, j} x^i y^j
		: \text{$c_{i, j} \in D$ for each $i, j$}
	\right\},
\]
and
\begin{equation}\label{V_k interior definition - ring}
	\mathcal V^\circ \colonequal \{
		(T_{\alpha - 1}, \dots, T_1, T_0)
		: \text{$T_k \in V_k^\circ$ for each $k \in \{0, 1, \dots, \alpha - 1\}$}
	\}.
\end{equation}
We have
\begin{align*}
	\dim \mathcal V^\circ
	&= \sum_{k = 0}^{\alpha - 1} \paren{
		((k + 1) h-1) \cdot (k + 1) (d - 1)
		+ \sum_{j = -k}^{-1} \paren{(k + 1) h + j}
	} \\
	&= \tfrac{1}{6} \alpha (\alpha + 1) ((2 h d - 1) \alpha + h d - 3 h - 3 d + 4) \\
	&= N - \tfrac{1}{2} \alpha (\alpha + 1) (h + d - 1).
\end{align*}

We use the following fact.
Let $U$ be a finite vector space with basis $\mathcal B$.
Let $(\mathcal B_1, \mathcal B_2)$ be a partition of $\mathcal B$, and let $U_1$ and $U_2$ be the subspaces generated by $\mathcal B_1$ and $\mathcal B_2$.
Let $\pr_{U_i}$ denote projection onto $U_i$.
If $f \colon U \to U$ and $\tilde{f} \colon U_1 \to U_1$ are linear transformations satisfying ${\pr_{U_1}} \circ f = \tilde f \circ \pr_{U_1}$, then
\[
	f(x) = \pr_{U_1}(f(x)) + \pr_{U_2}(f(x)) = \tilde{f}(\pr_{U_1}(x)) + \pr_{U_2}(f(x)),
\]
so that $\size{\orb_f(x)} \leq \size{U_2} \cdot \size{\orb_{\tilde f} (\pr_{U_1}(x))}$ for all $x \in U$.

We apply this fact to $U = \mathcal V$, $U_2 = \mathcal V^\circ$, and $f \colon \mathcal{V} \to \mathcal{V}$ defined by $f = {\rep_{p/Q}} \circ \lambda_{0, 0} \circ \val_{p/Q}$.
The space $U_1$ is generated by the union of the bases of the three borders.
Recall from Section~\ref{section: structure - ring} the operators $\lambda_0$, defined using one of three Laurent polynomials $R$ by $\lambda_0(S) = \Lambda_0\!\paren{S R^{p^\alpha - p^{\alpha - 1}}}$.
Define the linear transformation $\tilde{f} \colon U_1 \to U_1$ by
\begin{itemize}
\item
$\tilde{f} \circ {\pr_{x, 0}} = {\rep_{p/R}} \circ \lambda_0 \circ {\val_{p/R}} \circ {\pr_{x, 0}}$ where $R = \pi_{x, 0}(Q)$,
\item
$\tilde{f} \circ {\pr_{x, h}} = {\rep_{p/R}} \circ \lambda_0 \circ {\val_{p/R}} \circ {\pr_{x, h}}$ where $R = \pi_{x, h}(Q)$, and
\item
$\tilde{f} \circ {\pr_{y, d - 1}} = {\rep_{p/R}} \circ \lambda_0 \circ {\val_{p/R}} \circ {\pr_{x, d - 1}}$ where $R = \pi_{y, d - 1}(Q)$.
\end{itemize}
We have defined the images under $\tilde{f}$ of the basis elements $x^0 y^{d - 1}$ and $x^h y^{d - 1}$ twice, but $\tilde{f}$ is well defined by Theorem~\ref{state base-p/Q representation - univariate}.
We claim that ${\pr_{U_1}} \circ f = \tilde f \circ \pr_{U_1}$.
We consider each border.
For the left border, we show that ${\pr_{x, 0}} \circ f = \tilde{f} \circ \pr_{x, 0}$.
Corollary~\ref{univariate emulation - ring - tuple} gives
\begin{align*}
	{\pr_{x, 0}} \circ f
	&= {\rep_{p/R}} \circ \lambda_0 \circ \pi_{x, 0} \circ \val_{p/Q} \\
	&= {\rep_{p/R}} \circ \lambda_0 \circ {\val_{p/R}} \circ {\rep_{p/R}} \circ \pi_{x, 0} \circ \val_{p/Q}.
\end{align*}
By Theorem~\ref{state base-p/Q representation - univariate}, this implies
\begin{align*}
	{\pr_{x, 0}} \circ f
	&= {\rep_{p/R}} \circ \lambda_0 \circ {\val_{p/R}} \circ \pr_{x, 0} \\
	&= \tilde{f} \circ \pr_{x, 0}.
\end{align*}
Similarly for the other two borders.
The fact in the previous paragraph now implies $\size{\orb_{\lambda_{0, 0}}(S)} = \size{\orb_f(\rep_{p/Q}(S))} \leq \size{\mathcal V^\circ} \cdot \size{\orb_{\tilde f} (\pr_{U_1}(\rep_{p/Q}(S)))}$ for all $S \in \val_{p/Q}(\mathcal V)$.
That is,
\begin{equation}\label{intermediate - ring}
	\size{\orb_{\lambda_{0, 0}}(S)}
	\leq p^{N - \alpha (\alpha + 1) (h + d - 1) / 2} \cdot \size{\orb_{\tilde{f}} (\pr_{U_1}(\rep_{p/Q}(S)))}
\end{equation}
for all $S \in \val_{p/Q}(\mathcal V)$.
In particular, this is true for each state $\lambda_{0, 0}^n(S_u)$ where $n \geq 0$, since $\lambda_{0, 0}^n(S_u) \in \val_{p/Q}(\mathcal V)$ by Corollary~\ref{module - ring}.

It remains to bound the orbit size $\size{\orb_{\tilde{f}} (\pr_{U_1}(\rep_{p/Q}(S_u)))}$.
We do this by going from base-$\frac{p}{Q}$ representations back to Laurent polynomials and bounding the orbit sizes of the three projections $\pi_{x, 0}(S_u)$, $\pi_{x, p^{\alpha - 1} h}(S_u)$, and $\pi_{y, p^{\alpha - 1} (d - 1)}(S_u)$.
The definition of $\tilde{f}$ involves three functions $\tilde{f}\subl \colon \pr_{x, 0}(\mathcal V) \to \pr_{x, 0}(\mathcal V)$, $\tilde{f}\subr \colon \pr_{x, h}(\mathcal V) \to \pr_{x, h}(\mathcal V)$, and $\tilde{f}\subt \colon \pr_{y, d - 1}(\mathcal V) \to \pr_{y, d - 1}(\mathcal V)$.
By definition, an orbit under one of these functions is in bijection with the orbit of the corresponding Laurent polynomial under the respective $\lambda_0$ operator.

Corollary~\ref{combined univariate orbit size upper bound - ring} allows us to bound an orbit size under $\lambda_0$.
To satisfy the conditions in Corollary~\ref{combined univariate orbit size upper bound - ring}, we use Lemma~\ref{right degree - ring} to obtain
\begin{align*}
	\deg \pi_{x, 0}(S_{u + \max(u\subl, u\subr, u\subt)}) &\leq p^{\alpha - 1} \deg \pi_{x, 0}(Q) \\
	\deg \pi_{x, p^{\alpha - 1} h}(S_{u + \max(u\subl, u\subr, u\subt)}) &\leq p^{\alpha - 1} \deg \pi_{x, h}(Q) \\
	\deg \pi_{y, p^{\alpha - 1} (d - 1)}(S_{u + \max(u\subl, u\subr, u\subt)}) &\leq p^{\alpha - 1} \deg \pi_{y, d - 1}(Q).
\end{align*}
Therefore, the orbits under $\lambda_0$, of the three Laurent polynomials obtained by projecting $S_{u + \max(u\subl, u\subr, u\subt)}$ onto the three borders, are eventually periodic with transient lengths given by Equation~\eqref{transient time - ring}.
Define $t\subl$, $t\subr$, and $t\subt$ to be these transient lengths.
Recall that $\deg \pi_{x, 0}(Q) \leq d - 1$, $\deg \pi_{x, h}(Q) \leq d - 1$, and $\deg \pi_{y, d - 1}(Q) \leq h$.
We have
\begin{align*}
	t &\colonequal \max(u\subl, u\subr, u\subt)+\max(t\subl, t\subr, t\subt) \\
	&\leq\max(u\subl, u\subr, u\subt) + \max\!\left(
		\floor{\log_p \max(h, d - 1)} + \alpha,
		\ceil{\log_p \max(h, d)} + 2 (\alpha - 1)
	\right) \\
	&\leq \max(u\subl, u\subr, u\subt) + \max\!\left(
		\floor{\log_p \max(h, d)} + \alpha,
		\floor{\log_p \max(h, d)} + 2 \alpha - 1
	\right) \\
	&= \max(u\subl, u\subr, u\subt) + \floor{\log_p \max(h, d)} + 2 \alpha - 1.
\end{align*}
These terms are present in the claimed bound.
Set $S_{u + t}\colonequal \lambda_{0,0}^{u + t}(S_0)$.
We will bound the sizes of the periodic orbits
\begin{align*}
	\orb\subl(S_{u + t}) &\colonequal \{\lambda_0^n(\pi_{x, 0}(S_{u + t})) : n \geq 0\} \\
	\orb\subr(S_{u + t}) &\colonequal \{\lambda_0^n(\pi_{x, p^{\alpha - 1}h}(S_{u + t})) : n \geq 0\} \\
	\orb\subt(S_{u + t}) &\colonequal \{\lambda_0^n(\pi_{y, p^{\alpha - 1}(d-1)}(S_{u + t})) : n \geq 0\}
\end{align*}
where again the three operators $\lambda_0$ are defined with the respective $R$.
By Corollary~\ref{combined univariate orbit size upper bound - ring}, we have $\size{\orb\subl(S_{u + t})} = \lcm(\sigma)$ for some integer partition $\sigma \in \partitions(\deg \pi_{x, 0}(Q))$.
Similarly, for $\size{\orb\subr(S_{u + t})}$ and $\size{\orb\subt(S_{u + t})}$ we obtain two integer partitions in $\partitions(\deg \pi_{x, h}(Q))$ and $\partitions(\deg \pi_{y, d - 1}(Q))$.
We have
\[
	\size{\orb_{\lambda_{0, 0}}(S_{u + t})}
	\leq \lcm(\size{\orb\subl(S_{u + t})} \cdot \size{\orb\subr(S_{u + t})} \cdot \size{\orb\subt(S_{u + t})}).
\]
Therefore $\size{\orb_{\tilde{f}}(S)}\leq \Landaulcm(h, d, d)$.
Finally, Equation~\eqref{intermediate - ring} gives
\[
	\size{\orb_{\lambda_{0,0}}(S_{u + t})}
	\leq p^{N - \alpha (\alpha + 1) (h + d - 1) / 2} \cdot \Landaulcm(h, d, d)
\]
as desired.
\end{proof}

We can now prove Theorem~\ref{kernel size asymptotic bound - ring}.

\begin{main theorem - asymptotic - ring}
Let $p$ be a prime, let $\alpha \geq 1$, and let $F = \sum_{n \geq 0} a(n) x^n \in \Z_p\doublebracket{x} \setminus \{0\}$ be the Furstenberg series associated with a polynomial $P \in \Z_p[x, y]$.
Let $h \colonequal \deg_x(P \bmod p)$ and $d \colonequal \deg_y(P \bmod p)$, and assume $h = \deg_x P$ and $d = \deg_y P$.
Then $\size{\ker_p((a(n) \bmod p^\alpha)_{n \geq 0})}$ is in
\[
	(1 + o(1)) \, p^{\frac{1}{6} \alpha (\alpha + 1) ((2 h d - 1) \alpha + h d + 1)}
\]
as any of $p$, $\alpha$, $h$, or $d$ tends to infinity and the others remain constant.
\end{main theorem - asymptotic - ring}

\begin{proof}
We use the upper bound from Theorem~\ref{kernel size upper bound - ring}.
To bound $\Landaulcm(h, d, d)$, recall the Landau function $\Landau(n)$ from Section~\ref{section: introduction}.
The set of triples of integer partitions of $h,d,d$ gives rise to a subset of integer partitions of $h + 2 d$.
Therefore $\Landaulcm(h, d, d) \leq \Landau(h + 2 d)$.

By assumption, $\deg_x P = h$ and $\deg_y P = d$.
This simplifies the value of $u$ defined in Theorem~\ref{kernel size upper bound - ring} to $u = 1$.
Therefore, by Theorem~\ref{kernel size upper bound - ring}, the size of $\ker_p((a(n) \bmod p^\alpha)_{n \geq 0})$ is at most
\begin{multline*}
	p^N + p^{N - \alpha (\alpha + 1) (h + d - 1) / 2} \Landau(h + 2 d) \\
	+ \max(u\subl, u\subr, u\subt)
	+ \floor{\log_p \max(h, d)}
	+ 2 \alpha - 1
	+ 1.
\end{multline*}
The expression $\max(u\subl, u\subr, u\subt) + \floor{\log_p \max(h, d)} + 2 \alpha$ is clearly in $o(1) p^N$ as $p$, $\alpha$, $h$, or $d$ tends to infinity.
It remains to show that $p^{N - \alpha (\alpha + 1) (h + d - 1) / 2} \Landau(h + 2 d)$ is also in $o(1) p^N$.
Landau~\cite{Landau} proved that $\log \Landau(n) \sim \sqrt{n \log n}$, that is, $\Landau(n) = e^{(1 + \epsilon(n)) \sqrt{n \log n}}$, where $\epsilon(n) \to 0$ as $n \to \infty$.
It follows that
\[
	\frac{p^{N - \alpha (\alpha + 1) (h + d - 1) / 2} \Landau(h + 2 d)}{p^N}
	= \frac{\Landau(h + 2 d)}{p^{\alpha (\alpha + 1) (h + d - 1) / 2}}
	= \frac{e^{(1+\epsilon(h + 2 d)) \sqrt{(h + 2 d) \log (h + 2 d)}}}{p^{\alpha (\alpha + 1) (h + d - 1) / 2}},
\]
and this tends to $0$ as $p$, $\alpha$, $h$, or $d$ tends to infinity and the others remain constant.
\end{proof}

\section{Non-Furstenberg series}\label{section: non-Furstenberg}

In this section, we consider algebraic series annihilated by a polynomial $P \in \Z_p[x, y]$ that does not have an associated Furstenberg series.
We use a standard technique, given in Lemma~\ref{separation}, to obtain an integer $e\geq 0$ and a series $G$ with $F = \sum_{n=0}^e a(n) x^n + x^e G$, where $G(0) = 0$ and $G$ is annihilated by a polynomial whose derivative at the origin is nonzero.
However, we cannot guarantee that its derivative at the origin is nonzero modulo~$p$, i.e., that $G$ is its Furstenberg series.
For example, the polynomial $P = (x + 1) (3 x - 1) y^2 + 1$, which annihilates the generating series of the central trinomial coefficients, suffers from this defect for $p=2$ \cite[Section~4.2]{Rowland--Yassawi}, and we cannot overcome this obstacle with this technique.

If $G$ is indeed a Furstenberg series, then the size of the minimal automaton generating $(a(n) \bmod p^\alpha)_{n \geq 0}$ satisfies Theorem~\ref{kernel size upper bound - non-Furstenberg} below, which is the main result of this section.
It uses the following notation.

\begin{notation*}
Let $p$ be prime and $\alpha\geq 1$.
Let $P \in \Z_p[x, y]$ such that $h \colonequal \deg_x P \geq 1$ and $d \colonequal \deg_y P \geq 1$.
(Note that these definitions of $h$ and $d$ are different than in previous sections.)
Let $\res_y(P, \frac{\partial P}{\partial y})$ denote the resultant, with respect to $y$, of $P$ and $\frac{\partial P}{\partial y}$.
We assume that $P$ is square-free so that this resultant is not $0$.
Let $e = \mindeg_x(\res_y(P, \frac{\partial P}{\partial y}))$.
Since the resultant is the determinant of the Sylvester matrix of $P$ and $\frac{\partial P}{\partial y}$, we have
\[
	e \leq h (2 d - 1).
\]
Let $F = \sum_{n\geq0} a(n) x^n \in \Z_p\doublebracket{x}$ such that $P(x, F) = 0$.
Define
\begin{align*}
	V &\colonequal \sum_{n=0}^e a(n) x^n \\
	G &\colonequal \tfrac{1}{x^e}(F - V) = \sum_{n \geq 1} a(n + e) x^n \in x \Z_p\doublebracket{x} \\
	P_e &\colonequal P(x, V + x^e z) \in \Z_p[x, z].
\end{align*}
We have $G(0)=0$ and $P_e(x,G)=0$.
\end{notation*}

\begin{example}\label{Catalan non-Furstenberg series}
Let $p = 2$ and $P = x y^2 - y + 1$.
The generating series $F = \sum_{n \geq 0} C(n) x^n$ of the sequence of Catalan numbers satisfies $P(x, F) = 0$.
We have $\frac{\partial P}{\partial y} = 2 x y - 1$ and $\res_y(P, \frac{\partial P}{\partial y}) = x (4 x - 1)$.
Therefore $e = 1$, $V = 1 + x$, and
\[
	P_e = x^3 z^2 + (2 x^3 + 2 x^2 - x) z + x^3 + 2 x^2.
\]
Since $\frac{\partial P_e}{\partial z}(0, 0) = 0$, the polynomial $P_e$ does not have an associated Furstenberg series.
\end{example}

The following lemma, whose proof appears below, shows that dividing $P_e$ by a power of $x$ gives a polynomial whose derivative at the origin is nonzero.
This alone does not guarantee that $G$ is a Furstenberg series, however.
Recall that $\mathcal{C}$ is the center row operator, defined in Section~\ref{section: vector space - ring}.

\begin{lemma}\label{separation}
There exists an integer $s \in \{e, e+1, \dots, 2 e\}$ such that $P_e / x^s$ is a polynomial satisfying $\frac{\partial (P_e / x^s)}{\partial z}(0,0)\neq 0$, so that
\[
	G
	= \mathcal{C}\!\left(\frac{z \frac{\partial P_e}{\partial z}}{P_e / z}\right).
\]
If, moreover, $\frac{\partial (P_e / x^s)}{\partial z}(0, 0) \nequiv 0 \mod p$, then $G$ is the Furstenberg series associated with $P_e(x, z) / x^s$.
\end{lemma}

\begin{example}\label{Catalan Furstenberg series}
Continuing Example~\ref{Catalan non-Furstenberg series}, the polynomial $P_e / x$ has constant term $0$, and its derivative at the origin is $-1 \nequiv 0 \mod 2$.
Therefore $P_e / x$ has an associated Furstenberg series, namely $G = \sum_{n \geq 1} C(n + 1) x^n$.
\end{example}

\begin{theorem}\label{kernel size upper bound - non-Furstenberg}
With the notation above, suppose that $\frac{\partial (P_e / x^s)}{\partial z}(0, 0) \nequiv 0 \mod p$, $\deg_x P_e = \deg_x(P_e \bmod p)$, and $\deg_z P_e = \deg_z(P_e \bmod p)$.
Then $\size{\ker_p((a(n) \bmod p^\alpha)_{n \geq 0})}$ is in
\[
	(1 + o(1)) \, p^{\frac{1}{6} \alpha (\alpha + 1) ((2 h d (2 d^2 - d + 1) - 1) \alpha + h d (2 d^2 - d + 1) + 1)}
\]
as any of $p$, $\alpha$, $h$, or $d$ tends to infinity and the others remain constant.
\end{theorem}

We suspect that this bound can be improved, using the techniques in \cite[Section~4]{Adamczewski--Yassawi} where the expressions $x^i (V+x^ez)^j$ replace the monomials $x^i y^j$ as generators of the relevant vector space.
Doing this here would require extending the base-$\frac{p}{Q}$ numeration system to such expressions; we do not undertake this.

Write $G = \sum_{n\geq 0} b(n) x^n$.
To prove Theorem~\ref{kernel size upper bound - non-Furstenberg}, we will first bound $\size{\ker_p((a(n) \bmod p^\alpha)_{n\geq 0})}$ by the sum of $\size{\ker_p((b(n) \bmod p^\alpha)_{n\geq 0})}$ and some secondary terms;
we do this using elementary automaton techniques in Lemma~\ref{relating kernels}.
We will then use the results of Sections~\ref{section: structure - ring}--\ref{section: orbit size univariate - ring} to show that the secondary terms are small.
Finally, we will apply Theorem~\ref{kernel size asymptotic bound - ring} to $G$ to obtain a bound on $\size{\ker_p((b(n) \bmod p^\alpha)_{n\geq 0})}$ in terms of the height and degree of $P_e$.
The height of $P_e$ is at most $h + e d$, and its degree is $d$.
Using the fact that $e \leq h (2 d - 1)$, Theorem~\ref{kernel size asymptotic bound - ring} tells us that $\size{\ker_p((b(n) \bmod p^\alpha)_{n \geq 0})}$ is in
\[
	(1 + o(1)) \, p^{\frac{1}{6} \alpha (\alpha + 1) ((2 (h+h(2d-1)d) d - 1) \alpha + (h+h(2d-1)d) d + 1)}.
\]
This leads to the bound that appears for $\size{\ker_p((a(n) \bmod p^\alpha)_{n\geq 0})}$ in Theorem~\ref{kernel size upper bound - non-Furstenberg}.

We begin by proving Lemma~\ref{separation}.
The proof is adapted from the article of Adamczewski and Bell~\cite[proof of Lemma~6.2]{Adamczewski--Bell}.

\begin{proof}[Proof of Lemma~\ref{separation}]
We first determine the coefficients in the expansion of $P_e$ as a polynomial in $z$.
Write
\[
	P(x, y) = \sum_{j = 0}^d B_j(x) y^j.
\]
Considering the Taylor expansion of $P$ as a series in $y$ centered at $0$, we obtain
\[
	B_j(x) = \frac{1}{j!} \frac{\partial^j P}{\partial y^j}(x,0).
\]
We have $P_e(x,z) = P(x, V + x^e z)$, so the Taylor coefficients of $P_e$ are, by the chain rule,
\[
	\frac{1}{j!} \frac{\partial^j P_e}{\partial z^j}(x,0)
	= \frac{1}{j!} \frac{\partial^j P}{\partial y^j}(x, V) x^{e j}
	\equalcolon C_j(x) x^{e j}.
\]
Therefore
\[
	P_e(x,z)= \sum_{j=0}^d x^{e j} C_j(x) z^j.
\]
We have $\res_y(P, \frac{\partial P}{\partial y})=x^eT$ with $T(0) \neq 0$.
Note that
\begin{align*}
	P(x, V) &= P(x, V) - P(x, F)\\
	&= (V - F) C
\end{align*}
where $C = \sum_{j=1}^d B_j \left(\sum_{k=0}^{j-1} V^k F^{j - k - 1}\right) \in \Z_p\doublebracket{x}$.
Therefore
\[
	P(x,V) = -x^e G C.
\]
This implies that $x^{e+1}$ divides $P(x, V)$ since $G(0)=0$.
On the other hand, by a property of resultants, there exist two polynomials $A(x,y)$ and $B(x,y)$ such that
\[
	x^e T=\res_y(P, \tfrac{\partial P}{\partial y}) = A(x,y)P(x,y) + B(x,y)\tfrac{\partial P}{\partial y}(x,y).
\]
It follows that
\[
	B(x,V)\tfrac{\partial P}{\partial y}(x,V) = x^eT -A(x,V)P(x,V) = x^eT +A(x,V) x^e G C.
\]
Since $T(0)\neq 0$ and $G(0)=0$, we have $\mindeg_x(B(x,V) \frac{\partial P}{\partial y}(x,V)) = e$.
Therefore $s \colonequal \mindeg_x(x^e C_1) = \mindeg_x(x^e \frac{\partial P}{\partial y}(x,V)) \leq 2 e$.
Since the mindeg of $x^{ej} C_j$, which is the coefficient of $z^j$ in $P_e$, is at least $2e$ for all $j$ in the range $2\leq j\leq d$, we obtain that $\mindeg_x(C_0) \geq s$.
It follows that $P_e(x,z) / x^s$ is a polynomial and $\frac{\partial (P_e / x^s)}{\partial z}(0,0) \neq 0$.
Furthermore, since $G(0) = 0$, we have $\mindeg_x(C_0) \geq s + 1$, so $C_0 / x^s$ is divisible by $x$.
This implies that the constant term of $P_e(x, z) / x^s$ is $0$.
If, additionally, $\frac{\partial (P_e / x^s)}{\partial z}(0, 0) \nequiv 0 \mod p$, then $G$ is the Furstenberg series associated with $P_e(x, z) / x^s$.
\end{proof}

The following preparatory lemma gives a commutation relation between the operator $\Lambda_r$ and multiplication by a power of $x$.
Its proof is straightforward.

\begin{lemma}\label{intertwining}
Let $r \in \{0, 1, \dots, p - 1\}$ and $i \geq 0$.
For every power series $G \in \ring\doublebracket{x}$, we have
\[
	\Lambda_r(x^i G)
	= x^{\ceil{(i - r)/p}} \Lambda_{(r - i) \bmod p}(G).
\]
\end{lemma}

In the next lemma, we relate the kernel size of the coefficient sequence of $F \bmod p^\alpha$ to that of $G \bmod p^\alpha$.
Applying sufficiently many Cartier operators to $F= V + x^e G$, the image of $V$ is reduced to a constant.
Since $\deg V \leq e$, the number of Cartier operators we need to apply is logarithmic in $e\leq h(2d-1)$.
At the same time, we use Lemma~\ref{intertwining} to replace a composition of Cartier operators applied to $F\bmod p^\alpha$ by a composition of Cartier operators applied to $G \bmod p^\alpha$.
Example~\ref{relating kernels example} illustrates this process.

\begin{lemma}\label{relating kernels}
Write $G = \sum_{n\geq 0} b(n) x^n$.
Then $\size{\ker_p((a(n)\bmod p^\alpha)_{n \geq 0})}$ is at most
\[
	\size{\ker_p((b(n) \bmod p^\alpha)_{n \geq 0})} +
	\size{\orb_{\Lambda_0}(G \bmod p^\alpha)}
	+ \sum_{q=0}^{e - 1} \size{\orb_{\Lambda_{p-1}}(G_q \bmod p^\alpha)}
	+ p e
\]
for some series $G_0, G_1, \dots, G_{e - 1}$ whose coefficient sequences belong to $\ker_p(b(n)_{n \geq 0})$.
\end{lemma}

\begin{proof} We essentially express the states of an automaton for $(a(n)\bmod p^\alpha)_{n \geq 0}$ in terms of an automaton for $(b(n)\bmod p^\alpha)_{n \geq 0}$.
We have $F=V+x^e G$.
Let $t_0 = \floor{\log_p \max(e, 1)}$.
Let $t \geq 0$, and let $(r_t, r_{t - 1}, \dots, r_0)$ be a tuple of digits in $\{0, 1, \dots, p - 1\}$.
We consider $(\Lambda_{r_t} \circ \Lambda_{r_{t-1}} \circ \dots \circ \Lambda_{r_0})(F\bmod p^\alpha)$.

First let $t < t_0$.
We have
\[
	\size{\{F\bmod p^\alpha\} \cup \{(\Lambda_{r_t} \circ \dots \circ \Lambda_{r_0})(F\bmod p^\alpha): t < t_0\}}
	\leq \frac{p^{t_0 + 1} - 1}{p - 1}
	\leq p^{t_0 + 1}
	\leq p e
\]
by the definition of $t_0$.
This contributes the last term to the sum in the statement of the lemma.

Now let $t \geq t_0$.
Given $(r_t, \dots, r_0)$, iterating Lemma~\ref{intertwining} gives a unique tuple $(\tilde r_t, \dots, \tilde r_0)$ and a unique $i$ such that
\[
	(\Lambda_{r_t} \circ \dots \circ \Lambda_{r_0})(x^e G \bmod p^\alpha)
	= x^i \, (\Lambda_{\tilde r_t} \circ \dots \circ \Lambda_{\tilde r_0})(G \bmod p^\alpha).
\]
Moreover, we have $i \in \{0,1\}$ since applying a Cartier operator reduces the degree of a polynomial by a factor of $p$.
Since $F=V+x^e G$, this implies
\begin{multline}\label{bijection}
	(\Lambda_{r_t} \circ \dots \circ \Lambda_{r_0})(F \bmod p^\alpha) \\
	= (\Lambda_{r_t} \circ \dots \circ \Lambda_{r_0})(V \bmod p^\alpha)
	+ x^i \, (\Lambda_{\tilde r_t} \circ \dots \circ \Lambda_{\tilde r_0})(G \bmod p^\alpha)
\end{multline}
for a unique tuple $(\tilde r_t, \dots, \tilde r_0)$.
Let $q = r_t p^t+ r_{t-1} p^{t-1}+ \dots + r_0$; note that the leading digits can be $0$.
We consider three cases.
\begin{enumerate}
\item\label{less than e}
If $0 \leq q < e$, then the definition of $t_0$ implies that
\[
	(\Lambda_{r_{t_0}} \circ \dots \circ \Lambda_{r_0})(F\bmod p^\alpha)
	= a(q) + x \, (\Lambda_{\tilde r_{t_0}} \circ \dots \circ \Lambda_{\tilde r_0})(G\bmod p^\alpha).
\]
Moreover if $t>t_0$, then $(r_t, r_{t - 1}, \dots, r_{t_0}) = (0, 0, \dots, 0)$, so that
\begin{align*}
	(\Lambda_{r_t} \circ \dots \circ \Lambda_{r_0})(F\bmod p^\alpha)
	& = a(q) + \Lambda_0^{t - t_0}\!\left(x \, (\Lambda_{\tilde r_{t_0}} \circ \dots \circ \Lambda_{\tilde r_0})(G\bmod p^\alpha)\right) \\
	& = a(q) + x \, \Lambda_{p-1}^{t - t_0}\!\left((\Lambda_{\tilde r_{t_0}} \circ \dots \circ \Lambda_{\tilde r_0})(G\bmod p^\alpha)\right)
\end{align*}
by Lemma~\ref{intertwining}.
Let $G_q = (\Lambda_{\tilde r_{t_0}} \circ \dots \circ \Lambda_{\tilde r_0})(G)$.
Then the contribution from tuples with $q < e$ to the sum in the statement of the lemma is at most $\sum_{q=0}^{e - 1} \size{\orb_{\Lambda_{p-1}}(G_q \bmod p^\alpha)}$.
\item\label{equal to e}
If $q = e$, then Equation~\eqref{bijection} becomes
\[
	(\Lambda_{r_t} \circ \dots \circ \Lambda_{r_0})(F\bmod p^\alpha)
	= a(e) + \Lambda_0^{t+1}(G\bmod p^\alpha).
\]
This contributes at most $\size{\orb_{\Lambda_0}(G \bmod p^\alpha)}$.
\item\label{greater than e}
If $q > e$, then Equation~\eqref{bijection} becomes
\[
	(\Lambda_{r_t} \circ \dots \circ \Lambda_{r_0})(F\bmod p^\alpha)
	= (\Lambda_{\tilde r_t} \circ \dots \circ \Lambda_{\tilde r_0})(G\bmod p^\alpha).
\]
That is, the set of coefficient sequences of $(\Lambda_{r_t} \circ \dots \circ \Lambda_{r_0})(F\bmod p^\alpha)$, where $q > e$, is equal to the set of coefficient sequences of $(\Lambda_{\tilde r_t} \circ \dots \circ \Lambda_{\tilde r_0})(G\bmod p^\alpha)$.
This contributes at most $\size{\ker_p((b(n) \bmod p^\alpha)_{n \geq 0})}$.
\end{enumerate}
The result follows.
\end{proof}

\begin{example}\label{relating kernels example}
Continuing Example~\ref{Catalan Furstenberg series}, we have $F = V + x^e G = 1 + x + x G$.
The argument in the proof of Lemma~\ref{relating kernels} is independent of $\alpha$, so to illustrate the mechanics we work in $\Z_p\doublebracket{x}$ rather than $\ring\doublebracket{x}$.
We work out several images of $F$ under compositions of Cartier operators, using Lemma~\ref{intertwining}.
The value of $t_0 = \floor{\log_p \max(e, 1)}$ is $0$.
For Case~\ref{less than e} in the proof of Lemma~\ref{relating kernels}, we consider all base-$2$ expansions of $q = 0$.
Let $k \geq 1$.
We have
\[
	\Lambda_0^k(F)
	= \Lambda_0^k(1 + x + x G)
	= 1 + \Lambda_0^k(x G)
	= 1 + x \Lambda_1^k(G).
\]
For Case~\ref{equal to e}, we consider all base-$2$ expansions of $q = 1$.
Let $k \geq 0$.
We have
\[
	\Lambda_0^k(\Lambda_1(F))
	= \Lambda_0^k(\Lambda_1(1 + x + x G))
	= \Lambda_0^k(1 + \Lambda_0(G))
	= 1 + \Lambda_0^{k + 1}(G).
\]
We illustrate Case~\ref{greater than e} by considering all base-$2$ expansions of $q = 2$.
Let $k \geq 0$.
We have
\[
	\Lambda_0^k(\Lambda_1(\Lambda_0(F)))
	= \Lambda_0^k(\Lambda_1(\Lambda_0(1 + x + x G)))
	= \Lambda_0^k(\Lambda_1(1 + x \Lambda_1(G)))
	= \Lambda_0^{k + 1}(\Lambda_1(G)).
\]
\end{example}

We now prove the main result of the section.
This involves bounding the sizes of the orbits in Lemma~\ref{relating kernels}.
We bound orbits under $\Lambda_0$ as we do in the proof of Theorem~\ref{kernel size upper bound - ring}.
To bound orbit sizes under $\Lambda_{p-1}$, which so far we have not needed to do in this article, we use Corollary~\ref{module - ring}, which tells us that these orbits already live in a smaller space, along with a modification of some of the results in Section~\ref{section: structure - ring}.

\begin{proof}[Proof of Theorem~\ref{kernel size upper bound - non-Furstenberg}]
Let
\[
	 N\colonequal \tfrac{1}{6} \alpha (\alpha + 1) ((2 h d (2 d^2 - d + 1) - 1) \alpha + h d (2 d^2 - d + 1) + 1),
\]
let $h' \colonequal h + e d - s$, and let
\[
	M \colonequal \tfrac{1}{6} \alpha (\alpha + 1) ((2 h' d - 1) \alpha + h' d + 1).
\]
We bound each term of the sum
\[
	\size{\ker_p((b(n) \bmod p^\alpha)_{n \geq 0})} +
	\size{\orb_{\Lambda_0}(G \bmod p^\alpha)}
	+ \sum_{q=0}^{e - 1} \size{\orb_{\Lambda_{p-1}}(G_q \bmod p^\alpha)}
	+ p e
\]
from Lemma~\ref{relating kernels}.

For the term $p e$, one checks that $N \geq 2$, so $p e$ is in $o(1) p^N$ as any of $p$, $\alpha$, $h$, or $d$ tends to infinity and the others remain constant.
To bound the first term, we use Theorem~\ref{kernel size asymptotic bound - ring}.
As $P_e$ has height at most $h'$ and degree $d$,
Theorem~\ref{kernel size asymptotic bound - ring} implies that the size of $\ker_p((b(n) \bmod p^\alpha)_{n \geq 0})$ is in $(1 + o(1)) p^M$ as any of $p$, $\alpha$, $h'$, or $d$ tends to infinity and the others remain constant.
We have $M\leq N$ since $e\leq h(2d-1)$ and $s\geq 0$.
Therefore the size of $\ker_p((b(n) \bmod p^\alpha)_{n \geq 0})$ is in $(1 + o(1)) p^N$ as any of $p$, $\alpha$, $h$, or $d$ tends to infinity and the others remain constant.

For the second term, the last inequality in the proof Theorem~\ref{kernel size upper bound - ring} gives us
\[
	\size{\orb_{\Lambda_0}(G \bmod p^\alpha)} \leq
	 p^{M- \alpha (\alpha + 1) (h' + d - 1) / 2} \Landaulcm(h', d, d) + \text{small terms}.
\]
Then
\[
	\frac{p^{M- \alpha (\alpha + 1) (h' + d - 1) / 2} \Landaulcm(h', d, d)}{p^M} \leq
	\frac{e^{(1+\epsilon(h' + 2 d)) \sqrt{(h'+ 2 d) \log (h' + 2 d)}}}{p^{\alpha (\alpha + 1) (h'+ d - 1) / 2}}
\]
and this is in $o(1) p^M$ as any of $p$, $\alpha$, $h'$, or $d$ tends to infinity and the others remain constant.
Therefore $\size{\orb_{\Lambda_0}(G \bmod p^\alpha)}$ is in $o(1)p^N$ as any of $p$, $\alpha$, $h$, or $d$ tends to infinity and the others remain constant.

It remains to bound $\size{\orb_{\Lambda_{p-1}}(G_q \bmod p^\alpha)}$ for each $q \in \{1, 2, \dots, e\}$, where $G_q$ is defined in Case~\ref{less than e} in the proof of Lemma~\ref{relating kernels}.
If $e = 0$, then there are no such terms and we are done.
Assume $e \geq 1$.
Since $G$ is a Furstenberg series, Theorem~\ref{Furstenberg - ring} implies that $G$ is the diagonal of a rational function with denominator $Q_e \colonequal P_e / (x^s z)$.
Furthermore, as in Section~\ref{section: vector space - ring}, this implies that every element of $\ker_p((b(n) \bmod p^\alpha)_{n \geq 0})$ is the coefficient sequence of the diagonal of a rational function with denominator $Q_e^{p^{\alpha - 1}}$.
In particular, $G_q$ is the diagonal of such a rational function; let its numerator be $S_q$.
Therefore we can access each $\Lambda_{p - 1}(G_q)$ by working with the following operator.
Define $\lambda_{p-1, 0} \colon \ring[x, z, z^{-1}] \to \ring[x, z, z^{-1}]$ by
\[
	\lambda_{p-1, 0}(S)
	\colonequal
	\Lambda_{p-1, 0}\!\paren{S Q_e^{p^\alpha - p^{\alpha - 1}}}.
\]
Since by assumption $\deg_x P_e = \deg_x(P_e \bmod p)$ and $\deg_z P_e = \deg_z(P_e \bmod p)$, the value of $u$ in Proposition~\ref{transient - ring} is $1$.
Furthermore, inspecting Case~\ref{less than e} in the proof of Lemma~\ref{relating kernels}, we see that we apply at least one Cartier operator $\Lambda_{p - 1}$ to each of the series $G_q \bmod p^\alpha$ defined there.
By Corollary~\ref{module - ring}, the base-$\frac{p}{Q_e}$ representation of $S_q Q_e^{p^\alpha - p^{\alpha - 1}}$ belongs to the space $\mathcal W$ defined by
\[
	\mathcal W \colonequal \{
		(T_{\alpha - 1}, \dots, T_1, T_0)
		: \text{$T_k \in W_k$ for each $k \in \{0, 1, \dots, \alpha - 1\}$}
	\}
\]
where
\[
	W_k \colonequal
	\left\{
		\sum_{i = 0}^{(k + 1) h' - 1} \sum_{j = \max(-k, -i)}^{(k + 1) (d - 1)} c_{i, j} x^i z^j
		: \text{$c_{i, j} \in D$ for each $i, j$}
	\right\}
\]
analogously to Equation~\eqref{W definition - ring}.
Accordingly, Equation~\eqref{W size} implies that the size of the basis of $\mathcal W$ is $M$.

We modify the part of the proof in Theorem~\ref{kernel size upper bound - ring} where we bounded an orbit size under $\lambda_{0,0}$ to bound an orbit size under $\lambda_{p-1,0}$.
The key idea there was to use the results in Section~\ref{section: structure - ring} in order to restrict the spaces in which the base-$\frac{p}{Q}$ representations of elements in the orbit live.
We moved from the larger space $\mathcal V$ defined in Equation~\eqref{V definition - ring} to the smaller space $\mathcal V^\circ$ defined in Equation~\eqref{V_k interior definition - ring}.
Here we do something similar; namely, we move from $\mathcal W$ to a space $\mathcal W^\circ$ defined by
\[
	\mathcal W^\circ \colonequal \{
		(T_{\alpha - 1}, \dots, T_1, T_0)
		: \text{$T_k \in W_k^\circ$ for each $k \in \{0, 1, \dots, \alpha - 1\}$}
	\}
\]
where
\[
	W_k^\circ \colonequal
	\left\{
		\sum_{i = 0}^{(k + 1) h' - 1} \sum_{j = \max(-k, -i)}^{(k + 1)(d - 1) - 1} c_{i, j} x^i z^j
		: \text{$c_{i, j} \in D$ for each $i, j$}
	\right\}
\]
for each $k \in \{0, 1, \dots, \alpha - 1\}$.
(Since only some of the results in Section~\ref{section: structure - ring} extend to $\lambda_{p-1,0}$, the space $\mathcal W^\circ$ is slightly larger than $\mathcal V^\circ$ as defined in Equation~\eqref{V_k interior definition - ring}.)
We have
\begin{align*}
	\dim \mathcal W^\circ
	&= \sum_{k = 0}^{\alpha - 1} \paren{
		(k + 1) h' \cdot (k + 1) (d - 1)
		+ \sum_{j = -k}^{-1} \paren{(k + 1) h' + j}
	} \\
	&= \tfrac{1}{6} \alpha (\alpha + 1) ((2 h' d - 1) \alpha + h' d - 3 h' + 1) \\
	&= M - \tfrac{1}{2} \alpha (\alpha + 1) h'.
\end{align*}

The results that allow us to move from $\mathcal W$ to the smaller space $\mathcal W^{\circ}$ are the following.
The first result is a variant of the third part of Proposition~\ref{univariate emulation - ring}.
Using the notation from Section~\ref{section: structure - ring}, let $R = \pi_{z, d - 1}(Q_e)$.
Then for all $S \in \ring[x, z, z^{-1}]$ with degree at most $p^{\alpha - 1} (d-1)$, we have
\[
	\pi_{z, p^{\alpha - 1} (d - 1)}(\lambda_{p-1, 0}(S))
	= \lambda_{p-1}(\pi_{z, p^{\alpha - 1} (d - 1)}(S)).
\]
The second result is a variant of the third part of Theorem~\ref{state base-p/Q representation - univariate}.
Namely, if $S$ is a state in $\val_{p/Q_e}(\mathcal V)$ and $R = \pi_{z, d - 1}(Q_e)$, then $\pr_{z, d-1}(\rep_{p/Q_e}(S)) = \rep_{p/R}(\pi_{z, p^{\alpha - 1} (d - 1)}(S))$.
With these two results, the proof in Theorem~\ref{kernel size upper bound - ring} also applies to $\lambda_{p-1,0}$, giving
\[
	\size{\orb_{\lambda_{p-1,0}} (S_q)} \leq p^{M - \alpha (\alpha + 1) h' / 2} \Landau(h') + \text{small terms}.
\]

We now proceed as in the proof of Theorem~\ref{kernel size asymptotic bound - ring} to compare the size of $\orb_{\lambda_{p-1,0}} (S_q)$ to $M$.
Ignoring the small terms, we have
\[
	\frac{\size{\orb_{\lambda_{p-1,0}} (S_q)}}{p^M}
	\leq \frac{p^{M - \alpha (\alpha + 1) h' / 2} \Landau(h')}{p^M}
	= \frac{\Landau(h')}{p^{\alpha (\alpha + 1) h' / 2}}
	= \frac{e^{(1+\epsilon(h')) \sqrt{h' \log h'}}}{p^{\alpha (\alpha + 1) h' / 2}}.
\]
This expression tends to $0$ as $p$, $\alpha$, or $h'$ tends to infinity and the others remain constant, so $\size{\orb_{\lambda_{p-1,0}} (S_q)}\in o(1)p^M$.
Since $e\geq 1$, if $h$ or $d$ tends to infinity then $h'$ tends to infinity.
Since $M\leq N$, then $\size{\orb_{\lambda_{p-1,0}} (S_q)}\in o(1)p^N$ as $p$, $\alpha$, $h$, or $d$ tend to infinity and the others remain constant.
\end{proof}

\section{Diagonals of rational functions}\label{section: diagonals}

In this section, we widen our scope from algebraic series to diagonals of multivariate rational functions.
Over a field of nonzero characteristic, Furstenberg~\cite{Furstenberg} showed that the diagonal of a multivariate rational function is algebraic.
This is a special case of the following result due to Denef and Lipshitz~\cite[Theorem~6.2 and Remark~6.6]{Denef--Lipshitz}.

\begin{theorem}\label{automatic diagonal}
Let $p$ be a prime.
Let $P(x_1, \dots, x_m)$ and $Q(x_1, \dots, x_m)$ be polynomials in $\Z_p[x_1, \dots, x_m]$ such that $Q(0,\dots,0) \nequiv 0 \mod p$, and let
\[
	F \colonequal \mathcal{D}\!\paren{\frac{P(x_1, \dots, x_m)}{Q(x_1, \dots, x_m)}}.
\]
Then, for each $\alpha\geq 1$, the coefficient sequence of $F \bmod p^\alpha$ is $p$-automatic.
\end{theorem}

The proof of Theorem~\ref{automatic diagonal} is constructive~\cite{Rowland--Yassawi}, and the relevant operators $\lambda_{r, \dots, r}$ are defined for $S \in \ring[x_1, \dots, x_m]$ by
\[
	\lambda_{r, \dots, r}(S)
	\colonequal
	\Lambda_{r, \dots, r}\!\paren{S (Q \bmod p^\alpha)^{p^\alpha - p^{\alpha - 1}}}.
\]
By Lemma~\ref{lifting-the-exponent}, we can replace $Q \bmod p^\alpha$ with a lift of $Q \bmod p$ to $\ring[x_1, \dots, x_m]$ which has the same monomial support as $Q \bmod p$.

The numeration system that we developed in Section~\ref{section: image - ring} for bivariate Laurent polynomials can be adapted to multivariate polynomials.
This will allow us to prove Theorem~\ref{kernel size upper bound - multivariate diagonals}.

We begin with a bivariate analogue of Theorem~\ref{kernel size upper bound - ring}.
To state it, we extend the function $\Landaulcm(n_1, n_2, n_3)$ from Section~\ref{section: introduction} to $\Landaulcm(n_1, n_2, n_3, n_4)$, defined analogously as the maximum value of $\lcm(\lcm(\sigma_1), \lcm(\sigma_2), \lcm(\sigma_3), \lcm(\sigma_4))$ over integer partitions $\sigma_i$ of integers in $\{1, 2, \dots, n_i\}$.
The reason for this is that Theorem~\ref{automatic diagonal} is symmetric in $x_1, \dots, x_m$, unlike Theorem~\ref{Furstenberg - ring}.
This symmetry leads to the appearance of $\Landaulcm(h, h, d, d)$ in Theorem~\ref{kernel size upper bound - diagonals} instead of $\Landaulcm(h, d, d)$ as in Theorem~\ref{kernel size upper bound - ring}.

\begin{theorem}\label{kernel size upper bound - diagonals}
Let $p$ be a prime, and let $\alpha \geq 1$.
Let $P(x, y)$ and $Q(x, y)$ be polynomials in $\Z_p[x, y]$ such that $Q(0, 0) \nequiv 0 \mod p$.
Let
\[
	F = \mathcal{D}\!\paren{\frac{P(x, y)}{Q(x, y)}},
\]
and write $F = \sum_{n \geq 0} a(n) x^n$.
Let
\begin{align*}
	h &= \max(\deg_x(P \bmod p), \deg_x(Q \bmod p)) \\
	d &= \max(\deg_y(P \bmod p), \deg_y(Q \bmod p)),
\end{align*}
and assume $h\geq 1$ and $d\geq 1$.
Define $N = \frac{1}{6} \alpha (\alpha + 1) (2 \alpha + 1) h d$,
\begin{align*}
	u &= \big\lfloor\!\log_p \max\!\big( \\
		&\qquad \quad \alpha \, (\max(\deg_x(P \bmod p^\alpha), \deg_x(Q \bmod p^\alpha)) - h), \\
		&\qquad \quad \alpha \, (\max(\deg_y(P \bmod p^\alpha), \deg_y(Q \bmod p^\alpha)) - d), \\
		&\qquad \quad 1 \\
	&\phantom{{} = {}} \,\big)\big\rfloor,
\end{align*}
and
\begin{align*}
	u\subl &= \floor{\log_p \max\!\paren{p^{\alpha - 1} (d - \deg \pi_{x, 0}(Q)), 1}} + 1 \\
	u\subr &= \floor{\log_p \max\!\paren {p^{\alpha - 1} (d - \deg \pi_{x, h}(Q)), 1}} + 1 \\
	u\subb &= \floor{\log_p \max\!\paren {p^{\alpha - 1} (h - \deg \pi_{y, 0}(Q)), 1}} + 1\\
	u\subt &= \floor{\log_p \max\!\paren {p^{\alpha - 1} (h - \deg \pi_{y, d}(Q)), 1}} + 1.
\end{align*}
Then the size of $\ker_p((a(n) \bmod p^\alpha)_{n \geq 0})$ is at most
\begin{multline*}
	p^N + p^{N - \alpha ((\alpha + 1) (h + d) - 2) / 2} \Landaulcm(h, h, d, d)
	\\
	+ \max(u\subl, u\subr, u\subb, u\subt)
	+ \floor{\log_p \max(h, d)} + 2 \alpha - 1 + \tfrac{p^u - 1}{p - 1}.
\end{multline*}
Consequently, if $h = \max(\deg_x P, \deg_x Q)$ and $d = \max(\deg_y P, \deg_y Q)$, then $\size{\ker_p((a(n) \bmod p^\alpha)_{n \geq 0})}$ is in $(1 + o(1)) p^N$ as any of $p$, $\alpha$, $h$, or $d$ tends to infinity and the others remain constant.
\end{theorem}

Using $\max(u\subl, u\subr, u\subb, u\subt) \leq \floor{\log_p \max(h, d)} + \alpha$, we obtain the simpler bound
\[
	p^N + p^{N - \alpha ((\alpha + 1) (h + d) - 2) / 2} \Landaulcm(h, h, d, d)
	+ 2 \floor{\log_p \max(h, d)}
	+ 3 \alpha - 1
	+ \tfrac{p^u - 1}{p - 1}.
\]
We remark that one could further refine the definitions of $h$ and $d$ to obtain more refined bounds;
in particular, the bounds on the degrees of the digits of states in the automaton depend more on the degrees of $Q$ than $P$.

The structure of the proof of Theorem~\ref{kernel size upper bound - diagonals} is similar to that of Theorem~\ref{kernel size upper bound - ring}.
One difference is that the diagonal in Theorem~\ref{Furstenberg - ring} contains expressions of the form $P(x y, y)$, which led us to shear and to consider the maps $\lambda_{r, 0}$ on Laurent polynomials.
For general diagonals, the symmetry in $x, y$ means that no shearing is required, and no Laurent polynomials enter the picture.
We define the main objects and state the modifications of relevant results used in the proof.

Define $h$ and $d$ as in Theorem~\ref{kernel size upper bound - diagonals}.
Let
\[
	W_k \colonequal
	\left\{
		\sum_{i = 0}^{(k + 1) h - 1}
		\sum_{j = 0}^{(k + 1) d - 1}
		c_{i, j} x^i y^j
		: \text{$c_{i, j} \in D$ for each $i, j$}
	\right\}
\]
and
\[
	V_k \colonequal \left\{
		\sum_{i = 0}^{(k + 1) h}
		\sum_{j = 0}^{(k + 1) d}
		c_{i, j} x^i y^j
		: \text{$c_{i, j} \in D$ for each $i, j$}
	\right\}.
\]
Define
\begin{align*}
	\mathcal W &\colonequal \{
		(T_{\alpha - 1}, \dots, T_1, T_0)
		: \text{$T_k \in W_k$ for each $k \in \{0, 1, \dots, \alpha - 1\}$}
	\} \\
	\mathcal V &\colonequal \{
		(T_{\alpha - 1}, \dots, T_1, T_0)
		: \text{$T_k \in V_k$ for each $k \in \{0, 1, \dots, \alpha - 1\}$}
	\}.
\end{align*}
The dimension of $\mathcal W$ is
\[
	N
	\colonequal \sum_{k = 0}^{\alpha - 1} (k + 1) h \cdot (k + 1) d
	= \tfrac{1}{6} \alpha (\alpha + 1) (2 \alpha + 1) h d.
\]
The initial state of the automaton is $S_0 = \paren{P Q^{p^{\alpha - 1} - 1} \bmod p^\alpha}$.
An analogue of Theorem~\ref{state base-p/Q representation} tells us that every state in the automaton has a base-$\frac{p}{Q \bmod p^\alpha}$ representation.
To simplify notation, for the remainder of this section, we refer to this representation as a base-$\frac{p}{Q}$ representation.

To prove Theorem~\ref{kernel size upper bound - diagonals}, we provide analogues of Lemma~\ref{initial state}, Proposition~\ref{transient - ring}, Proposition~\ref{univariate emulation - ring}, and Lemma~\ref{right degree - ring}.
Let
\begin{align*}
	h_k &= \max(\deg_x(P \bmod p^k), \deg_x(Q \bmod p^k)) \\
	d_k &= \max(\deg_y(P \bmod p^k), \deg_y(Q \bmod p^k)).
\end{align*}

\begin{lemma}\label{initial state - diagonals}
The base-$\frac{p}{Q}$ digits $T_k$ of the initial state $S_0$ satisfy
\begin{align*}
	\deg_x T_k &\leq (k + 1) h_k \\
	\deg_y T_k &\leq (k + 1) d_k.
\end{align*}
\end{lemma}

\begin{proposition}\label{transient - diagonals}
Let $S_0$ be the initial state, and let $(T_{\alpha - 1}, \dots, T_1, T_0) \colonequal \rep_{p/Q}(S_0)$.
Let
\[
	u = \floor{\log_p \max\!\paren{\alpha (h_{\alpha - 1} - h), \alpha (d_{\alpha - 1} - d), 1}} + 1.
\]
Then, for all $r_1, r_2, \dots, r_u \in \{0, 1, \dots, p - 1\}$, we have $\rep_{p/Q}((\lambda_{r_u, 0} \circ \dots \circ \lambda_{r_2, 0} \circ \lambda_{r_1, 0})(S_0)) \in \mathcal V$.
\end{proposition}

\begin{proposition}\label{univariate emulation - diagonals}
We have the following.
\begin{enumerate}
\item
Let $R = \pi_{x, 0}(Q)$.
For all $S \in \ring[x, y]$,
\[
	\pi_{x, 0}(\lambda_{0, 0}(S))
	= \lambda_0(\pi_{x, 0}(S)).
\]
\item
Let $R = \pi_{x, h}(Q)$.
For all $S \in \ring[x, y]$ with height at most $p^{\alpha - 1} h$,
\[
	\pi_{x, p^{\alpha - 1} h}(\lambda_{0, 0}(S))
	= \lambda_0(\pi_{x, p^{\alpha - 1} h}(S)).
\]
\item
Let $R = \pi_{y, 0}(Q)$.
For all $S \in \ring[x, y]$,
\[
	\pi_{y, 0}(\lambda_{0, 0}(S))
	= \lambda_0(\pi_{y, 0}(S)).
\]
\item
Let $R = \pi_{y, d}(Q)$.
For all $S \in \ring[x, y]$ with degree at most $p^{\alpha - 1} d$,
\[
	\pi_{y, p^{\alpha - 1} d}(\lambda_{0, 0}(S))
	= \lambda_0(\pi_{y, p^{\alpha - 1} d}(S)).
\]
\end{enumerate}
\end{proposition}

\begin{lemma}\label{right degree - diagonals}
Define $u\subl$, $u\subr$, $u\subb$, and $u\subt$ as in Theorem~\ref{kernel size upper bound - diagonals}.
For all $S \in \val_{p/Q}(\mathcal V)$, we have
\begin{enumerate}
\item
$\deg \pi_{x, 0}(\lambda_{0, 0}^n(S)) \leq p^{\alpha - 1} \deg \pi_{x, 0}(Q)$ for all $n \geq u\subl$,
\item
$\deg \pi_{x, p^{\alpha - 1} h}(\lambda_{0, 0}^n(S)) \leq p^{\alpha - 1} \deg \pi_{x, h}(Q)$ for all $n \geq u\subr$,
\item
$\deg \pi_{y, 0}(\lambda_{0,0}^n(S)) \leq p^{\alpha - 1} \deg \pi_{y, 0}(Q)$ for all $n \geq u\subb$,
and
\item
$\deg \pi_{y, p^{\alpha - 1} d}(\lambda_{0,0}^n(S)) \leq p^{\alpha - 1} \deg \pi_{y, d}(Q)$ for all $n \geq u\subt$.
\end{enumerate}
\end{lemma}

Define
\[
	V_k^\circ \colonequal \left\{
		\sum_{i = 1}^{(k + 1) h - 1}
		\sum_{j = 1}^{(k + 1) d - 1}
		c_{i, j} x^i y^j
		: \text{$c_{i, j} \in D$ for each $i, j$}
	\right\},
\]
and define
\[
	\mathcal V^\circ \colonequal \{
		(T_{\alpha - 1}, \dots, T_1, T_0)
		: \text{$T_k \in V_k^\circ$ for each $k \in \{0, 1, \dots, \alpha - 1\}$}
	\}.
\]
We have
\begin{align*}
	\dim \mathcal V^\circ
	&= \tfrac{1}{6} \alpha (2 \alpha^2 h d + 3 \alpha (h d - h - d) + h d - 3 h - 3 d + 6) \\
	&= N - \tfrac{1}{2} \alpha ((\alpha + 1) (h + d) - 2).
\end{align*}
With Lemma~\ref{initial state - diagonals}, Proposition~\ref{transient - diagonals}, Proposition~\ref{univariate emulation - diagonals}, and Lemma~\ref{right degree - diagonals}, one can follow the proof of Theorem~\ref{kernel size upper bound - ring} to give a proof of Theorem~\ref{kernel size upper bound - diagonals}.

Finally, we prove Theorem~\ref{kernel size upper bound - multivariate diagonals}.

\begin{main theorem - diagonals}
Let $p$ be a prime, let $\alpha \geq 1$, and let
\[
	F \colonequal \mathcal{D}\!\paren{\frac{P(x_1, \dots, x_m)}{Q(x_1, \dots, x_m)}}
\]
where $P(x_1, \dots, x_m)$ and $Q(x_1, \dots, x_m)$ are polynomials in $\Z_p[x_1, \dots, x_m]$ such that $Q(0, \dots, 0) \nequiv 0 \mod p$ and $m\geq 2$.
Write $F = \sum_{n \geq 0} a(n) x^n$.
Let $h_i = \max(\deg_{x_i}(P \bmod p), \deg_{x_i}(Q \bmod p))$, and assume that $h_i \geq 1$ for each $i$.
Let $M = \sum_{k = 0}^{\alpha - 1} \prod_{i = 1}^m ((k + 1) h_i + 1)$; then
\[
	\size{\ker_p((a(n) \bmod p^\alpha)_{n \geq 0})}\leq p^M.
\]
\end{main theorem - diagonals}

\begin{proof}
The proof of the bound is similar to the proof of Corollary~\ref{kernel initial upper bound - ring}.
Namely, let
\[
	V_k \colonequal \left\{
		S \in D[x_1, \dots, x_m]
		:
		\text{$\deg_{x_i} S \leq (k + 1) h_i$ for each $i$}
	\right\}
\]
and
\[
	\mathcal V \colonequal \{
		(T_{\alpha - 1}, \dots, T_1, T_0)
		: \text{$T_k \in V_k$ for each $k \in \{0, 1, \dots, \alpha - 1\}$}
	\}.
\]
Then we have
\[
	\dim \mathcal V
	= \sum_{k = 0}^{\alpha - 1} \prod_{i = 1}^m ((k + 1) h_i + 1)
	= M.
\]
By Lemma~\ref{initial state - diagonals}, the initial state belongs to $\mathcal V$.
By a version of Corollary~\ref{module - ring}, we have $\lambda_{r, 0}(\val_{p/Q}(\mathcal V)) \subseteq \val_{p/Q}(\mathcal V)$.
The statement follows.
\end{proof}

\begin{remark}
For $m \geq 3$, we cannot do better without a multivariate version of the results of Section~\ref{section: orbit size univariate - ring}, in particular Corollary~\ref{combined univariate orbit size upper bound - ring}.
For example, let $m = 3$.
One can modify the beginning of the proof of Theorem~\ref{kernel size upper bound - ring} to bound $\size{\orb_{\lambda_{0, 0, 0}}(S_0)}$ using $p^{\dim \mathcal V^\circ}$ and six bivariate operators $\lambda_{0, 0}$, where
\[
	V_k^\circ \colonequal
	\left\{
		S \in D[x_1, x_2, x_3]
		:
		\text{$1\leq \mindeg_{x_i} S$ and $\deg_{x_i} S \leq (k + 1) h_i -2$ for each $i$}
	\right\}
\]
and
\[
	\mathcal V^\circ \colonequal \{
		(T_{\alpha - 1}, \dots, T_1, T_0)
		: \text{$T_k \in V_k^\circ$ for each $k \in \{0, 1, \dots, \alpha - 1\}$}
	\}.
\]
We have
\[
	\dim \mathcal V^\circ
	= \sum_{k = 0}^{\alpha - 1} \prod_{i = 1}^3 ((k + 1) h_i -2).
\]
By Theorem~\ref{kernel size upper bound - diagonals}, the orbit size of the projection of the initial state onto one of the six borders under the relevant operator $\lambda_{0, 0}$ is in $(1 + o(1)) p^{\frac{1}{6} \alpha (\alpha + 1) (2 \alpha + 1) h_i h_j}$ for the appropriate $i$ and $j$ where $i \neq j$.
However, this is too large relative to the size of $\mathcal W$, defined by
\[
	W_k \colonequal \left\{
		S \in D[x_1, x_2, x_3]
		:
		\text{$\deg_{x_i} S \leq (k + 1) h_i -1$ for each $i$}
	\right\}
\]
and
\[
	\mathcal W \colonequal \{
		(T_{\alpha - 1}, \dots, T_1, T_0)
		: \text{$T_k \in W_k$ for each $k \in \{0, 1, \dots, \alpha - 1\}$}
	\}.
\]
Namely, we have
\[
	N\colonequal \dim \mathcal W
	= \sum_{k = 0}^{\alpha - 1} \prod_{i = 1}^3 (k + 1) h_i
	= \tfrac{1}{4} \alpha^2 (\alpha +1)^2 h_1 h_2 h_3.
\]
Let $S_0 = \paren{P Q^{p^{\alpha - 1} - 1} \bmod p^\alpha}$.
If $r>0$, then a version of Corollary~\ref{module - ring} gives $\rep_{p/Q}(\lambda_{r, r, r} (S_0)) \in \mathcal W$.
It remains to consider $\size{\orb_{\lambda_{0, 0, 0}}(S_0)}$.
The ratio of the orbit size of $S_0$ under this $\lambda_{0, 0, 0}$ to the size of $\mathcal W$ is in
\begin{multline*}
	\frac{
		p^{\dim \mathcal V^\circ} \cdot (1+o(1)) p^{\frac{1}{6} \alpha (\alpha + 1) (2 \alpha + 1) (2 h_1 h_2 + 2 h_1 h_3 + 2 h_2 h_3)}
	}{
		p^{\frac{1}{4} \alpha^2 (\alpha +1)^2 h_1 h_2 h_3}
	} \\
	= (1 + o(1)) p^{2 \alpha (\alpha +1) (h_1+h_2+h_3) - 8 \alpha},
\end{multline*}
which clearly is not in $o(1)$ as $p$, $\alpha$, $h_1$, $h_2$, or $h_3$ tends to infinity.
\end{remark}

\section*{Acknowledgments}

We thank Frits Beukers and Armin Straub for discussions about their related work.
We also thank the reviewers for constructive feedback.

\section*{Appendix: Explicit automaton sizes}

The tables in this appendix contain polynomials $P \in \ring[x, y]$, found through systematic searches, whose Furstenberg series achieve maximum automaton size or maximum orbit size under $\lambda_{0, 0}$.
By definition, $P(0, 0) = 0$ and $\frac{\partial P}{\partial y}(0, 0) \nequiv 0 \mod p$; for simplicity, we require that the coefficient of $x^0 y^1$ in $P$ is $1$.
All polynomials listed happen to satisfy $h= \deg_x P$ and $d = \deg_y P$, so the value of $u$ in Theorem~\ref{kernel size upper bound - ring} is always $u = 1$.
The automata were computed with the Mathematica package \textsc{IntegerSequences}~\cite{IntegerSequences, IntegerSequences article}.

Table~\ref{automaton size table - ring} lists the maximum unminimized automaton size for several values of $p$, $\alpha$, $h$, and $d$, along with one polynomial that achieves this size and the value of the bound in Theorem~\ref{kernel size upper bound - ring}.

Table~\ref{orbit size table - ring} contains data on maximal orbit sizes under $\lambda_{0,0}$.
For these searches, we assume that the coefficients in $P$ belong to $\{0, 1, \dots, p - 1\}$ to make the search space more accessible.
In practice, this restriction does not seem to be consequential.
For $p^\alpha \in \{4, 8\}$, $d = 1$, and $h \leq 4$, only one polynomial with a coefficient outside this range results in a larger orbit size, namely $(2 x + 1) y + x$, which yields orbit size $3$ modulo~$4$ and also $3$ modulo~$8$.
For $p^\alpha \in \{4, 8\}$, $d = 2$, and $h \leq 2$, no polynomials result in a larger orbit size.

Surprisingly, all but one of the polynomials that maximize the orbit size for $p^\alpha = 4$ in Table~\ref{orbit size table - ring} also maximize the orbit size for $p^\alpha = 8$.
These polynomials are therefore good candidates for obtaining maximal orbit sizes for larger values of $\alpha$.
For example, let $P = x^2 y^2 + (x^2 + x + 1) y + x^2$.
For each $\alpha \in \{1, 2, \dots, 14\}$, the transient length under $\lambda_{0, 0}$ is $\alpha + 1$ and the eventual period length is $2^{\alpha + 1}$, leading one to conjecture that the orbit size is exactly $2^{\alpha + 1} + \alpha + 1$ for all $\alpha \geq 1$.
This perhaps suggests the true growth rate, indicating a possible direction for future work.

\small

\begin{table}[h]
	\raggedright
	$p^\alpha = 4$:
	\[
		\begin{array}{c|c|c|c|c}
			h & d & P & \text{aut.\ size} & p^N \\
			\hline
			\hline
			1 & 1 & (x + 1) y + x & 6 & 16 \\
			2 & 1 & (3 x^2 + x + 1) y + x^2 & 18 & 512 \\
			3 & 1 & (x^3 + x + 1) y + x^3 & 70 & 16384 \\
			4 & 1 & (3 x^4 + x + 1) y + x^4 + x & 189 & 524288 \\
			\hline
			1 & 2 & (x + 2) y^2 + (x + 1) y + x & 18 & 512 \\
			2 & 2 & (x^2 + x + 3) y^2 + (2 x^2 + x + 1) y + x & 222 & 524288 \\
			3 & 2 & (x^3 + x^2 + 1) y^2 + (x^3 + 1) y + x & 4826 & 536870912
		\end{array}
	\]
	$p^\alpha = 8$:
	\[
		\begin{array}{c|c|c|c|c}
			h & d & P & \text{aut.\ size} & p^N \\
			\hline
			\hline
			1 & 1 & (x + 1) y + x & 10 & 1024 \\
			2 & 1 & (3 x^2 + x + 1) y + x^2 + x & 61 & 16777216 \\
			3 & 1 & (x^3 + x + 1) y + x^3 & 246 & 274877906944 \\
			\hline
			1 & 2 & (x + 2) y^2 + (5 x + 1) y + x & 56 & 16777216 \\
			2 & 2 & (x^2 + x + 3) y^2 + (x^2 + 2 x + 1) y + x^2 & 2571 & 4503599627370496
		\end{array}
	\]
	$p^\alpha = 9$:
	\[
		\begin{array}{c|c|c|c|c}
			h & d & P & \text{aut.\ size} & p^N \\
			\hline
			\hline
			1 & 1 & (4 x + 1) y + x & 14 & 81 \\
			2 & 1 & (2 x^2 + 7 x + 1) y + x^2 + x & 123 & 19683 \\
			3 & 1 & (4 x^3 + 2 x + 1) y + x^3 & 562 & 4782969 \\
			\hline
			1 & 2 & (x + 4) y^2 + y + x & 171 & 19683 \\
			2 & 2 & (x^2 + x + 5) y^2 + (x^2 + 1) y + x & 11073 & 1162261467
		\end{array}
	\]
	\caption{Polynomials in $\ring[x, y]$ achieving the maximum unminimized automaton size for given values of $p$, $\alpha$, $h$, and $d$. The value of $N$ in the final column is $N = \frac{1}{6} \alpha (\alpha + 1) ((2 h d - 1) \alpha + h d + 1)$ from Theorem~\ref{kernel size upper bound - ring}.}
	\label{automaton size table - ring}
\end{table}

\begin{table}[h]
	\raggedright
	$p^\alpha = 4$:
	\[
		\begin{array}{c|c|c|c|c}
			h & d & P & \text{orbit size} & \text{bound} \\
			\hline
			\hline
			1 & 1 & y + x & 2 & 10 \\
			2 & 1 & y + x^2 + x & 4 & 136 \\
			3 & 1 & y + x^3 + x^2 & 4 & 3080 \\
			4 & 1 & y + x^4 + x & 5 & 65546 \\
			5 & 1 & (x^5 + x + 1) y + x & 7 & 1572874 \\
			6 & 1 & (x^5 + x + 1) y + x^6 + x & 8 & 25165834 \\
			7 & 1 & (x^7 + x^6 + x^2 + x + 1) y + x & 13 & 805306378 \\
			8 & 1 & (x^8 + x^3 + 1) y + x & 16 & 1.61 \times 10^{10} \\
			9 & 1 & (x^9 + x^2 + 1) y + x & 21 & 3.43 \times 10^{11} \\
			10 & 1 & (x^{10} + x^6 + x^3 + x^2 + 1) y + x & 31 & 8.24 \times 10^{12} \\
			\hline
			1 & 2 & x y^2 + (x + 1) y + x & 4 & 135 \\
			2 & 2 & x^2 y^2 + (x^2 + x + 1) y + x^2 & 11 & 2097160 \\
			3 & 2 & (x^3 + x + 1) y^2 + (x^3 + 1) y + x^3 & 25 & 1.03 \times 10^{11} \\
			4 & 2 & (x^4 + x^2 + x) y^2 + (x^4 + x + 1) y + x^4 & 50 & 1.68 \times 10^{15} \\
			5 & 2 & (x^5 + x^3 + 1) y^2 + (x^5 + x + 1) y + x & 121 & 4.61 \times 10^{19} \\
			6 & 2 & (x^6 + x^4 + x) y^2 + (x^5 + x + 1) y + x^3 & 122 & 7.55 \times 10^{23} \\
			7 & 2 & (x^7 + x + 1) y^2 + (x^7 + x^6 + x^5 + x + 1) y + x & 337 & 1.73 \times 10^{28}
		\end{array}
	\]
	$p^\alpha = 8$:
	\[
		\begin{array}{c|c|c|c|c}
			h & d & P & \text{orbit size} & \text{bound} \\
			\hline
			\hline
			1 & 1 & y + x & 2 & 4105 \\
			2 & 1 & y + x^2 + x & 5 & 1.37 \times 10^{11} \\
			3 & 1 & y + x^3 + x^2 & 5 & 3.45 \times 10^{18} \\
			4 & 1 & y + x^4 + x & 6 & 7.73 \times 10^{25} \\
			5 & 1 & (x^5 + x + 1) y + x & 8 & 1.94 \times 10^{33} \\
			6 & 1 & (x^5 + x + 1) y + x^6 & 9 & 3.26 \times 10^{40} \\
			7 & 1 & (x^7 + x^6 + x^2 + x + 1) y + x & 14 & 1.09 \times 10^{48} \\
			8 & 1 & (x^8 + x^3 + 1) y + x & 17 & 2.29 \times 10^{55} \\
			9 & 1 & (x^9 + x^2 + 1) y + x & 22 & 5.14 \times 10^{62} \\
			10 & 1 & (x^{10} + x^6 + x^3 + x^2 + 1) y + x & 32 & 1.29 \times 10^{70} \\
			\hline
			1 & 2 & x y^2 + (x + 1) y + x & 5 & 1.37 \times 10^{11} \\
			2 & 2 & x^2 y^2 + (x^2 + x + 1) y + x^2 & 20 & 1.01 \times 10^{31} \\
			3 & 2 & (x^3 + x + 1) y^2 + (x^3 + 1) y + x^3 & 50 & 2.24 \times 10^{51} \\
			4 & 2 & (x^4 + x^2 + x) y^2 + (x^4 + x + 1) y + x^4 & 99 & 1.65 \times 10^{71} \\
			5 & 2 & (x^5 + x^3 + 1) y^2 + (x^5 + x + 1) y + x & 242 & 2.03 \times 10^{91} \\
			6 & 2 & (x^6 + x^4 + x) y^2 + (x^5 + x + 1) y + x^3 & 243 & 1.50 \times 10^{111} \\
			7 & 2 & (x^7+ x+ 1) y^2+ (x^7+ x^6+ x^5+ x+ 1) y+ x & 674 & 1.55 \times 10^{131}
		\end{array}
	\]
	$p^\alpha = 9$:
	\[
		\begin{array}{c|c|c|c|c}
			h & d & P & \text{orbit size} & \text{bound} \\
			\hline
			\hline
			1 & 1 & y + x & 2 & 15 \\
			2 & 1 & y + x^2 + x & 3 & 1464 \\
			3 & 1 & y + x^3 + x & 4 & 177155 \\
			4 & 1 & (2 x^4 + x + 1) y + x & 5 & 19131884 \\
			5 & 1 & (x^5 + 2 x^2 + 1) y + x & 7 & 2324522942 \\
			6 & 1 & (x^5 + 2 x^2 + 1) y + x^6 & 8 & 1.88 \times 10^{11} \\
			\hline
			1 & 2 & x y^2 + y + x & 4 & 1464 \\
			2 & 2 & (x^2 + x + 1) y^2 + y + x^2 + x & 20 & 6973568808 \\
			3 & 2 & (x^3 + x^2 + x + 2) y^2 + (x + 1) y + x^3 + x & 57 & 1.00 \times 10^{17} \\
			4 & 2 & (x^3 + 2 x + 1) y^2 + (x^4 + 1) y + x & 218 & 4.78 \times 10^{23}
		\end{array}
	\]
	\caption{Polynomials with coefficients in $\{0, 1, \dots, p - 1\}$ for which the initial state achieves the maximum orbit size under $\lambda_{0, 0}$ for given values of $p$, $\alpha$, $h$, and $d$. The final column contains the value of $p^{N - \alpha (\alpha + 1) (h + d - 1) / 2} \Landaulcm(h, d, d) + \floor{\log_p \max\!\paren{p^{\alpha - 1} h, p^{\alpha - 1} (d - 1)}} + 1 + \floor{\log_p \max(h, d)} + 2 \alpha - 1 + 1$ from Theorem~\ref{kernel size upper bound - ring}.}
	\label{orbit size table - ring}
\end{table}

\end{document}